\documentclass{amsart}
\usepackage{graphicx,verbatim, amsmath, amssymb, amsthm, amsfonts, epsfig, ifthen,mathtools,epstopdf,caption,enumerate,hyperref}	
\DeclareFontFamily{U}{mathx}{\hyphenchar\font45}
\DeclareFontShape{U}{mathx}{m}{n}{<-> mathx10}{}
\DeclareSymbolFont{mathx}{U}{mathx}{m}{n}
\DeclareMathAccent{\widebar}{0}{mathx}{"73}

\newtheorem{proposition}{Proposition}[section]
\newtheorem{theorem}[proposition]{Theorem}

\newtheorem{lemma}[proposition]{Lemma}
\newtheorem{prop}[proposition]{Proposition}
\newtheorem{cor}[proposition]{Corollary}
\newtheorem{conj}[proposition]{Conjecture}

\theoremstyle{definition}
\newtheorem{example}[proposition]{Example}
\newtheorem{definition}[proposition]{Definition}
\newtheorem{question}[proposition]{Question}

\theoremstyle{remark}
\newtheorem{remark}[proposition]{Remark}

\numberwithin{equation}{section}
\usepackage[usenames]{color}

\usepackage[usenames]{color}

\newcounter{margincounter}
\newcommand{\newword}[1]{\textbf{\emph{#1}}}

\newcommand{\integers}{\mathbb Z}
\newcommand{\rationals}{\mathbb Q}

\newcommand{\reals}{\mathbb R}

\newcommand{\ep}{\varepsilon}
\newcommand{\col}{\mathbf{col}}

\newcommand{\sgn}{\operatorname{sgn}}
\newcommand{\vsgn}{\mathbf{sgn}}

\newcommand{\set}[1]{{\left\lbrace #1 \right\rbrace}}

\newcommand{\br}[1]{{\langle #1 \rangle}}
\newcommand{\A}{{\mathcal A}}

\newcommand{\F}{{\mathcal F}}

\newcommand{\ck}{^\vee}

\newcommand{\dashname}[1]{\stackrel{#1}{\mbox{---\!---}}}
\newcommand{\st}{^\mathrm{st}}
\renewcommand{\th}{^\text{th}}
\newcommand{\nd}{^\text{nd}}

\newcommand{\0}{{\hat{0}}}
\newcommand{\1}{{\hat{1}}}

\newcommand{\Tits}{\mathrm{Tits}}
\newcommand{\Cone}{\mathrm{Cone}}

\newcommand{\relint}{\mathrm{relint}}

\newcommand{\diag}{\mathrm{diag}}

\newcommand{\g}{\mathbf{g}}
\renewcommand{\b}{\mathbf{b}}
\renewcommand{\k}{\mathbf{k}}
\renewcommand{\a}{\mathbf{a}}
\newcommand{\e}{\mathbf{e}}
\newcommand{\x}{\mathbf{x}}
\newcommand{\y}{\mathbf{y}}

\renewcommand{\t}{\mathbf{t}}
\renewcommand{\v}{\mathbf{v}}

\newcommand{\w}{\mathbf{w}}
\newcommand{\tB}{\tilde{B}}
\newcommand{\T}{\mathbb{T}}

\newcommand{\R}{\mathcal{R}}

\renewcommand{\P}{\mathbb{P}}
\newcommand{\K}{\mathbb{K}}
\newcommand{\Trop}{\operatorname{Trop}}

\newcommand{\bB}{\widebar{B}}

\title{Universal geometric cluster algebras}
\author{Nathan Reading}
\thanks{This material is based upon work partially supported by the National Security Agency under Grant Number H98230-09-1-0056, by the Simons Foundation under Grant Number 209288 and by the National Science Foundation under Grant Number DMS-1101568.}

\begin{document}

\begin{abstract}
We consider, for each exchange matrix $B$, a category of geometric cluster algebras over $B$ and coefficient specializations between the cluster algebras.
The category also depends on an underlying ring $R$, usually $\integers$, $\rationals$, or $\reals$.
We broaden the definition of geometric cluster algebras slightly over the usual definition and adjust the definition of coefficient specializations accordingly.
If the broader category admits a universal object, the universal object is called the cluster algebra over $B$ with universal geometric coefficients, or the universal geometric cluster algebra over $B$.
Constructing universal geometric coefficients is equivalent to finding an $R$-basis for $B$ (a ``mutation-linear'' analog of the usual linear-algebraic notion of a basis).
Polyhedral geometry plays a key role, through the mutation fan $\F_B$, which we suspect to be an important object beyond its role in constructing universal geometric coefficients.
We make the connection between $\F_B$ and $\g$-vectors.
We construct universal geometric coefficients in rank $2$ and in finite type and discuss the construction in affine type.
\end{abstract}
\maketitle

\setcounter{tocdepth}{1}
\tableofcontents

\section{Introduction}
In this paper, we consider the problem of constructing universal geometric cluster algebras:  cluster algebras that are universal (in the sense of coefficient specialization) among cluster algebras of geometric type with a fixed exchange matrix $B$.
In order to accommodate universal geometric cluster algebras beyond the case of finite type, we broaden the definition of geometric type by allowing extended exchange matrices to have infinitely many coefficient rows.
We have also chosen to broaden the definition in another direction, by allowing the coefficient rows to have entries in some underlying ring (usually $\integers$, $\rationals$, or $\reals$).
We then narrow the definition of coefficient specialization to rule out pathological coefficient specializations.

There are at least two good reasons to allow the underlying ring $R$ to be something other than $\integers$.
First, when $R$ is a field, a universal geometric cluster algebra over $R$ exists for any $B$ (Corollary~\ref{univ exists}).
Second, the polyhedral geometry of mutation fans, defined below, strongly suggests the use of underlying rings larger than~$\rationals$.
(See in particular Section~\ref{rk2 sec}.)

On the other hand, there are good reasons to focus on the case where $R=\integers$.
In this case, the broadening of the definition of geometric type (allowing infinitely many coefficient rows) is mild:  
Essentially, we pass from polynomial coefficients to formal power series coefficients.
The only modification of the definition of coefficient specializations is to require that they respect the formal power series topology.
The case $R=\integers$ is the most natural in the context of the usual definition of geometric type.
We have no proof, for general $B$, that universal geometric cluster algebras exist over $\integers$, but we also have no counterexamples.
In this paper and \cite{unisurface,unitorus} we show that universal geometric cluster algebras over $\integers$ exist for many exchange matrices.

Given an underlying ring $R$ and an exchange matrix $B$, the construction of a universal geometric cluster algebra for $B$ over $R$ is equivalent to finding an $R$-basis for $B$ (Theorem~\ref{basis univ}).
The notion of an $R$-basis for $B$ is a ``mutation-linear'' analog of the usual linear-algebraic notion of a basis.
The construction of a basis is in turn closely related to the mutation fan $\F_B$, a (usually infinite) fan of convex cones that are essentially the domains of linearity of the action of matrix mutation on coefficients.
In many cases, a basis is obtained by taking one nonzero vector in each ray of $\F_B$, or more precisely in each ray of the $R$-part of $\F_B$.
(See Definition~\ref{R-part def}.)
Indeed, we show (Corollary~\ref{FB basis ray}) that a basis with the nicest possible properties only arises in this way.

The fan $\F_B$ appears to be a fundamental object.
For example, conditioned on a well-known conjecture of Fomin and Zelevinsky (sign-coherence of principal coefficients), we show (Theorem~\ref{g subfan}) that there is a subfan of $\F_B$ containing all cones spanned by $\g$-vectors of clusters of cluster variables for the transposed exchange matrix $B^T$.
In particular, for $B$ of finite type, the fan $\F_B$ coincides with the $\g$-vector fan for $B^T$, and as a result, universal coefficients for $B$ are obtained by making an extended exchange matrix whose coefficient rows are exactly the $\g$-vectors of cluster variables for $B^T$ (Theorem~\ref{finite g univ}). 
This is a new interpretation of the universal coefficients in finite type, first constructed in~\cite[Theorem~12.4]{ca4}.
We conjecture that a similar statement holds for $B$ of affine type as well, except that one must adjoin one additional coefficient row beyond those given by the $\g$-vectors for $B^T$ (Conjecture~\ref{affine univ conj}).
We intend to prove this conjecture in general in a future paper using results of \cite{framework, afframe}.
We will also describe the additional coefficient row in terms of the action of the Coxeter element.
Here, we work out the rank-$2$ case and one rank-$3$ case of the conjecture.
The cases of the conjecture where $B$ arises from a marked surface are proved in~\cite{unisurface}.  (See \cite[Remark~7.15]{unisurface}).  
Further, but more speculatively, we suspect that $\F_B$ should play a role in the problem of finding nice additive bases for cluster algebras associated to $B$.
In finite type, the cluster monomials are an additive basis.
Each cluster variable indexes a ray in the $\g$-vector fan (i.e.\ in $\F_B$), and cluster monomials are obtained as combinations of cluster variables whose rays are in the same cone of $\F_B$.
Beyond finite type, the rays of $\F_B$ may in some cases play a similar role to provide basic building blocks for constructing additive bases.

When $B$ arises from a marked surface in the sense of \cite{cats1,cats2}, the rational part of the mutation fan $\F_B$ can in most cases be constructed by means of a slight modification of the notion of laminations.
This leads to the construction, in \cite{unisurface}, of universal geometric coefficients for a family of surfaces including but not limited to the surfaces of finite or affine type.  
(The family happens to coincide with the surfaces of polynomial growth identified in \cite[Proposition~11.2]{cats1}.) 
In \cite{unitorus}, we explicitly construct universal geometric coefficients for an additional surface: the once-punctured torus.
A comparison of \cite{FG2011} and \cite{unitorus} suggests a connection between universal geometric coefficients and the cluster $\mathcal{X}$-varieties of Fock and Goncharov \cite{FGDilog,FG2011}, but the precise nature of this connection has not been worked out.

The definitions and results given here are inspired by the juxtaposition of several results on cluster algebras of finite type from \cite{ca4,cambrian,camb_fan,YZ}.
Details on this connection are given in Remark~\ref{me and ca4}.

Throughout this paper, $[n]$ stands for $\set{1,2,\ldots,n}$.
The notation $[a]_+$ stands for $\max(a,0)$, and $\sgn(a)$ is $0$ if $a=0$ or $a/|a|$ if $a\neq 0$.
Given a vector $\a=(a_1,\ldots,a_n)$ in $\reals^n$, the notation $\mathbf{min}(\a,\mathbf{0})$ stands for $(\min(a_1,0),\ldots,\min(a_n,0))$.
Similarly, $\vsgn(\a)$ denotes the vector $(\sgn(a_1),\ldots,\sgn(a_n))$.

Notation such as $(u_i:i\in I)$ stands for a list of objects indexed by a set $I$ of arbitrary cardinality, generalizing the notation $(u_1,\ldots,u_n)$ for an $n$-tuple.
We use the Axiom of Choice throughout the paper without comment.

\section{Cluster algebras of geometric type}\label{defs sec}
In this section we define cluster algebras of geometric type.
We define geometric type more broadly than the definition given in \cite[Section~2]{ca4}.
The exposition here is deliberately patterned after \cite[Section~2]{ca4} to allow easy comparison.

\begin{definition}[\emph{The underlying ring}]\label{def R}
All of the definitions in this section depend on a choice of an \newword{underlying ring} $R$, which is required be either the integers $\integers$ or some field containing the rationals $\rationals$ as a subfield and contained as a subfield of the reals $\reals$.
At present, we see little need for the full range of possibilities for $R$.
Rather, allowing $R$ to vary lets us deal with a few interesting cases simultaneously.
Namely, the case $R=\integers$ allows us to see the usual definitions as a special case of the definitions presented here, while examples in Section~\ref{rk2 sec} suggest taking $R$ to be the field of algebraic real numbers, or some finite-degree extension of $\rationals$.
The case $R=\reals$ also seems quite natural given the connection that arises with the discrete (real) geometry of the mutation fan.
(See Definition~\ref{mut fan def}.)
\end{definition}

\begin{definition}[\emph{Tropical semifield over $R$}]\label{trop semifield}
Let $I$ be an indexing set.
Let $(u_i:i\in I)$ be a collection of formal symbols called \newword{tropical variables}.
Define $\Trop_R(u_i:i\in I)$ to be the abelian group whose elements are formal products of the form $\prod_{i\in I}u_i^{a_i}$ with each $a_i\in R$ and multiplication given by 
\[\prod_{i\in I}u_i^{a_i}\cdot\prod_{i\in I}u_i^{b_i}=\prod_{i\in I}u_i^{a_i+b_i}.\]
Thus, as a group, $\Trop_R(u_i:i\in I)$ is isomorphic to the direct product, over the indexing set $I$, of copies of $R$.
The multiplicative identity $\prod_{i\in I}u_i^0$ is abbreviated as $1$.

We define an \newword{auxiliary addition} $\oplus$ in $\Trop_R(u_i:i\in I)$ by 
\[\prod_{i\in I}u_i^{a_i}\oplus\prod_{i\in I}u_i^{b_i}=\prod_{i\in I}u_i^{\min(a_i,b_i)}.\]
The triple $(\Trop_R(u_i:i\in I)\,,\oplus\,,\,\cdot\,)$ is a \newword{semifield}: an abelian multiplicative group equipped with a commutative, associative addition $\oplus$, with multiplication distributing over $\oplus$.
Specifically, this triple is called a \newword{tropical semifield over $R$}.

We endow $R$ with the discrete topology and endow $\Trop_R(u_i:i\in I)$ with the \newword{product topology} as a product of copies of the discrete set $R$.
The product topology is sometimes called the \newword{Tychonoff} topology or the \newword{topology of pointwise convergence}.
Details on the product topology are given later in Section~\ref{univ sec}.
The reason we impose this topology is seen in Proposition~\ref{continuous linear}.
When $I$ is finite, the product topology on $\Trop_R(u_i:i\in I)$ is the discrete topology, and when $I$ is countable, the product topology on $\Trop_R(u_i:i\in I)$ is the usual topology of formal power series (written multiplicatively).

As a product of copies of $R$, the semifield $\Trop_R(u_i:i\in I)$ is a module over~$R$.
Since the group operation in $\Trop_R(u_i:i\in I)$ is written multiplicatively, the scalar multiplication corresponds to exponentiation.
That is, $c\in R$ acts by sending $\prod_{i\in I}u_i^{a_i}$ to $\prod_{i\in I}u_i^{c a_i}$.
\end{definition}
One recovers \cite[Definition~2.2]{ca4} by requiring $I$ to be finite and taking the underlying ring $R$ to be $\integers$.
One may then ignore the topological considerations.

The role of the tropical semifield $\P$ in this paper is to provide a coefficient ring $\integers\P$ for cluster algebras.
The notation $\integers\P$ denotes the group ring, over $\integers$, of the multiplicative group $\P$, ignoring the addition $\oplus$.
The larger group ring $\rationals\P$ also makes a brief appearance in Definition~\ref{seed}.

\begin{definition}[\emph{Labeled geometric seed of rank $n$}]\label{seed}
Let $\P=\Trop_R(u_i:i\in I)$ be a tropical semifield.
Let $\K$ be a field isomorphic to the field of rational functions in $n$ independent variables with coefficients in $\rationals\P$.
A \newword{(labeled) geometric seed} of rank $n$ is a pair $(\x,\tB)$ described as follows:
\begin{enumerate}
\item[$\bullet$] $\tB$ is a function from $\left([n]\cup I\right)\times [n]$ to $R$.
For convenience, the function $\tB$ is referred to as a matrix over $R$.
\item[$\bullet$] The first rows of $\tB$ indexed by $[n]$ have integer entries and form a \newword{skew-symmetrizable} matrix $B$.  
(That is, there exist positive integers $d_1,\ldots,d_n$ such that $d_ib_{ij}=-d_jb_{ji}$ for all $i,j\in[n]$.  
If the integers $d_i$ can be taken to all be $1$, then $B$ is \newword{skew-symmetric}.)
\item[$\bullet$] $\x=(x_1,\ldots,x_n)$ is an $n$-tuple of algebraically independent elements of $\K$ which generate $\K$ (i.e.\ $\K=\rationals\P(x_1,\ldots,x_n)$).
\end{enumerate}
The $n$-tuple $\x$ is the \newword{cluster} in the seed and the entries of $\x$ are called the \newword{cluster variables}.
The square matrix $B$ is the \newword{exchange matrix} or the \newword{principal part} of $\tB$ and the full matrix $\tB$ is the \newword{extended exchange matrix}.
The rows of $\tB$ indexed by $I$ are called \newword{coefficient rows};  the coefficient row indexed by $i$ is written $\b_i=(b_{i1},\ldots,b_{in})$.
The semifield $\P$ is called the \newword{coefficient semifield} and is determined up to isomorphism by the number of coefficient rows of $\tB$ (i.e.\ by the cardinality of~$I$).
The elements $y_j=\prod_{i\in I}u_i^{b_{ij}}$ of $\P$ are the \newword{coefficients} associated to the seed.
\end{definition}

\begin{definition}[\emph{Mutation of geometric seeds}]\label{mutation}
Fix $I$, $\P=\Trop_R(u_i:i\in I)$, and $\K$ as in Definitions~\ref{trop semifield} and~\ref{seed}.
For each $k\in[n]$ we define an involution $\mu_k$ on the set of labeled geometric seeds of rank $n$.
Let $(\x,\tB)$ be a labeled geometric seed.

We define a new seed $\mu_k(\x,\tB)=(\x',\tB')$ as follows.
The new cluster $\x'$ is $(x_1',\ldots,x_n')$ with $x_j'=x_j$ whenever $j\neq k$ and 
\begin{equation}\label{x mut}
x_k'=\frac{y_k\prod_{i=1}^nx_i^{[b_{ik}]_+}+\prod_{i=1}^nx_i^{[-b_{ik}]_+}}{x_k(y_k\oplus 1)}\,\,,
\end{equation}
where, as above, $y_k$ is the coefficient $\prod_{i\in I}u_i^{b_{ik}}$ and $1$ is the multiplicative identity $\prod_{i\in I}u_i^0$.
Thus $x_k'$ is a rational function in $\x$ with coefficients in $\integers\P$, or in other words, $x_k'\in\K$.

The new extended exchange matrix $\tB'$ has entries
\begin{equation}\label{b mut}
b_{ij}'=\left\lbrace\!\!\begin{array}{ll}
-b_{ij}&\mbox{if }i=k\mbox{ or }j=k;\\
b_{ij}+\sgn(b_{kj})\,[b_{ik}b_{kj}]_+&\mbox{otherwise.}
\end{array}\right.
\end{equation}

The top part of $\mu_k(\tB)$ is skew-symmetrizable with the same skew-symmetrizing constants $d_i$ as the top part of $\tB$.
The map $\mu_k$ is an involution.

The notation $\mu_k$ also denotes the mutation $\mu_k(\tB)=\tB'$ of extended exchange matrices, which does not depend on the cluster $\x$.
Given a finite sequence $\k=k_q,\ldots,k_1$ of indices in~$[n]$, the notation $\mu_\k$ stands for $\mu_{k_q}\circ\mu_{k_{q-1}}\circ\cdots\circ\mu_{k_1}$.
We have indexed the sequence $\k$ so that the first entry in the sequence is on the right.
\end{definition}

\begin{definition}[\emph{Mutation equivalence of matrices}]\label{mut eq}
Two exchange matrices (or extended exchange matrices) are called \newword{mutation equivalent} if there is a sequence~$\k$ such that one matrix is obtained from the other by applying $\mu_\k$.
The set of all matrices mutation equivalent to $B$ is the \newword{mutation class} of $B$.
\end{definition}

\begin{definition}[\emph{Regular $n$-ary tree}]\label{tree}
Let $\T_n$ denote the $n$-regular tree with edges labeled by the integers $1$ through $n$ such that each vertex is incident to exactly one edge with each label.
The notation $t\dashname{k}t'$ indicates that vertices $t$ and $t'$ are connected by an edge labeled $k$.
\end{definition}

\begin{definition}[\emph{Cluster pattern and $Y$-pattern of geometric type}]\label{pattern}
Fix a vertex $t_0$ in $\T_n$.
Given a seed $(\x,\tB)$, the assignment $t_0\mapsto (\x,\tB)$ extends to a unique map $t\mapsto(\x_t,\tB_t)$ from the vertices of $\T_n$ by requiring that $(\x_{t'},\tB_{t'})=\mu_k(\x_t,\tB_t)$ whenever $t\dashname{k}t'$.
This map is called a \newword{cluster pattern (of geometric type)}.
The map $t\mapsto\tB_t$ is called a \newword{$Y$-pattern (of geometric type)}.
The cluster variables and matrix entries of $(\x_t,\tB_t)$ are written $(x_{1;t},\ldots,x_{n;t})$ and $(b_{ij}^t)$.
For each $j\in[n]$, the coefficient $\prod_{i\in I}u_i^{b_{ij}^t}$ is represented by the symbol $y_{j;t}$.
\end{definition}

\begin{definition}[\emph{Cluster algebra of geometric type}]\label{cluster algebra}
The \newword{cluster algebra} $\A$ associated to a cluster pattern $t\mapsto(\x_t,\tB_t)$ is the $\integers\P$-subalgebra of $\K$ generated by all cluster variables occurring in the cluster pattern.
That is, setting
\[\mathcal{X}=\set{x_{i;t}:t\in\T_n,i\in[n]},\]
we define $\A=\integers\P[\mathcal{X}]$.
Since the seed $(\x,\tB)=(\x_{t_0},\tB_{t_0})$ uniquely determines the cluster pattern, we write $\A=\A_R(\x,\tB)$.
Up to isomorphism, the cluster algebra is determined entirely by $\tB$, so we can safely write $\A_R(\tB)$ for $\A_R(\x,\tB)$.
\end{definition}

\begin{remark}\label{comparison}
Definitions~\ref{seed}, \ref{mutation}, \ref{pattern}, and~\ref{cluster algebra} are comparable to \cite[Definition~2.3]{ca4}, \cite[Definition~2.4]{ca4}, \cite[Definition~2.9]{ca4}, and \cite[Definition~2.11]{ca4} restricted to the special case described in \cite[Definition~2.12]{ca4}.
Essentially, the only difference is the broader definition of the notion of a tropical semifield, which forces us to allow $\tB$ to have infinitely many rows, and to allow non-integer entries in the coefficient rows of $\tB$.
Definition~\ref{tree} is identical to \cite[Definition~2.8]{ca4}.
\end{remark}

\section{Universal geometric coefficients}\label{univ sec}
In this section, we define and discuss a notion of coefficient specialization of cluster algebras of geometric type and define cluster algebras with universal geometric coefficients.
These definitions are comparable, but not identical, to the corresponding definitions in \cite[Section~12]{ca4}, which apply to arbitrary cluster algebras.

\begin{definition}[\emph{Coefficient specialization}]\label{specialization}
Let $(\x,\tB)$ and $(\x',\tB')$ be seeds of rank~$n$, and let $\P$ and $\P'$ be the corresponding tropical semifields over the same underlying ring $R$.
A ring homomorphism $\varphi:\A_R(\x,\tB)\to\A_R(\x',\tB')$ is a \newword{coefficient specialization} if
\begin{enumerate}[(i)]
\item the exchange matrices $B$ and $B'$ coincide;
\item $\varphi(x_j)=x'_j$ for all $j\in[n]$;
\item \label{spec restr}
the restriction of $\varphi$ to $\P$ is a continuous $R$-linear map to $\P'$ with $\varphi(y_{j;t})=y'_{j;t}$ and $\varphi(y_{j;t}\oplus 1)=y'_{j;t}\oplus 1$ for all $j\in[n]$ and $t\in\T_n$.
\end{enumerate}
Continuity in \eqref{spec restr} refers to the product topology described in Definition~\ref{trop semifield}.
\end{definition}

\begin{definition}[\emph{Cluster algebra with universal geometric coefficients}]\label{universal}
A cluster algebra $\A=\A_R(\tB)$ of geometric type with underlying ring $R$ is \newword{universal over~$R$} if for every cluster algebra $\A'=\A_R(\tB')$ of geometric type with underlying ring~$R$ sharing the same initial exchange matrix $B$, there exists a unique coefficient specialization from $\A$ to $\A'$.
In this case, we also say $\tB$ is \newword{universal over $R$} and call the coefficient rows of $\tB$ \newword{universal geometric coefficients} for $B$ over $R$.
\end{definition}

The local conditions of Definition~\ref{specialization} imply some global conditions, as recorded in the following proposition, whose proof follows immediately from \eqref{x mut} and \eqref{b mut}.

\begin{prop}\label{global local}
Continuing the notation of Definition~\ref{specialization}, if $\varphi$ is a coefficient specialization, then 
\begin{enumerate}[(i$'$)]
\item \label{global local B}
the exchange matrices $B_t$ and $B'_t$ coincide for any $t$.  
\item \label{global local x}
$\varphi(x_{j;t})=x'_{j;t}$ for all $j\in[n]$ and $t\in\T_n$.
\end{enumerate}
\end{prop}

\begin{remark}\label{FZ def}
A converse to Proposition~\ref{global local} is true for certain $B$.
Specifically, suppose the exchange matrix $B$ has no column composed entirely of zeros.
If the restriction of $\varphi$ to $\P$ is a continuous $R$-linear map and conditions~(\ref{global local B}$'$) and (\ref{global local x}$'$) of Proposition~\ref{global local} both hold, then $\varphi$ is a coefficient specialization.
This fact is argued in the proof of \cite[Proposition~12.2]{ca4}, with no restrictions on $B$.
(Naturally, the argument there makes no reference to continuity and $R$-linearity, which are not part of \cite[Definition~12.1]{ca4}.)
The argument given there is valid in the context of Definition~\ref{specialization}, over any $R$, as long as $B$ has no zero column.

The reason \cite[Proposition~12.2]{ca4} can fail when $B$ has a zero column is that it may be impossible to distinguish $p_{j;t}^+$ from $p_{j;t}^-$ in \cite[(12.2)]{ca4}.
Rather than defining $p_{j;t}^+$ and $p_{j;t}^-$ here, we consider a simple case in the notation of this paper.
Take $R=\integers$ and let $\tB$ have a $1\times1$ exchange matrix $[0]$ and a single coefficient row~$1$, so that in particular $\P=\Trop_\integers(u_1)$.
The cluster pattern associated with $(x_1,\tB)$ has two seeds:  One with cluster variable $x_1$ and coefficient $u_1$, and the other with cluster variable $x_1^{-1}(u_1+1)$ and coefficient $u_1^{-1}$.
On the other hand, let $\tB'$ also have exchange matrix $[0]$ but let its single coefficient row be~$-1$.
Let $\P'=\Trop_\integers(u'_1)$.
The cluster pattern associated with $(x'_1,\tB')$ has a seed with cluster variable $x_1'$ and coefficient $(u'_1)^{-1}$ and a seed with cluster variable $(x_1')^{-1}(1+u_1')$ and coefficient $u_1'$.
The map $\varphi$ sending $x_1$ to $x'_1$ and $u_1$ to $u_1'$ extends to a ring homomorphism satisfying condition (\ref{global local x}$'$).
However, condition \eqref{spec restr} fails.

The characterization in \cite[Theorem~12.4]{ca4} of universal coefficients in finite type is valid except when $B$ has a zero column, or in other words when $A(B)$ has a irreducible component of type $A_1$.
It can be fixed by separating $A_1$ as an exceptional case or by taking \cite[Proposition~12.2]{ca4} as the definition of coefficient specialization.
\end{remark}

\begin{remark}\label{why linear}
Here we discuss the requirement in Definition~\ref{specialization} that the restriction of $\varphi$ be $R$-linear.
This is strictly stronger than the requirement that the restriction be a group homomorphism, which is all that is required in \cite[Definition~12.1]{ca4}.
The additional requirement in Definition~\ref{specialization} is that the restriction of $\varphi$ commutes with scaling (i.e.\ exponentiation) by elements of $R$.
This requirement is forced by the fact that Definitions~\ref{def R} and~\ref{trop semifield} allow the underlying ring $R$ to be larger than $\rationals$.
When $R$ is $\integers$ or $\rationals$, the requirement of commutation with scaling is implied by the requirement of a group homomorphism.

In general, however, consider again the example of a cluster algebra $\A$ of rank~$1$, given by the $2\times1$ matrix with rows $0$ and $1$ and the initial cluster $\x$ with a single entry $x_1$.  
The coefficient semifield is $\Trop_R(u_1)$.
The cluster variables are $x_1$ and $x_1^{-1}(u_1+1)$.
Suppose the underlying ring $R$ is a field.
Then $R$ is a vector space over $\rationals$, and homomorphisms of additive groups from $R$ to itself correspond to $\rationals$-linear maps of vector spaces.
If $R$ strictly contains $\rationals$, then there are infinitely many such maps fixing $1$.
Thus, if we didn't require commutation with scaling, there would be infinitely many coefficient specializations from $\A$ to itself fixing $1$, $x_1$ and $x_1^{-1}(u_1+1)$.
The requirement of linearity ensures that the identity map is the only coefficient specialization fixing $1$, $x_1$ and $x_1^{-1}(u_1+1)$.
\end{remark}

\begin{remark}\label{why continuous}
Here we discuss the requirement in Definition~\ref{specialization} that the restriction of $\varphi$ be continuous.
This requirement is forced, once we allow the set of tropical variables to be infinite. 
Take $I$ to be infinite, let $\P=\Trop_R(u_i:i\in I)$ and consider a cluster algebra $\A$ of rank $1$ given by the exchange matrix $[0]$ and each coefficient row equal to $1$.
Consider the set of $R$-linear maps $\varphi:\P\to\P$ with $\varphi(u_i)=u_i$ for all $i\in I$.
Linearity plus the requirement that $u_i\mapsto u_i$ only determine $\varphi$ on elements of $\P$ with finite support (elements $\prod_{i\in I}u_i^{a_i}$ with only finitely many nonzero $a_i$).
This leaves infinitely many $R$-linear maps from $\P$ to itself that send every tropical variable $u_i$ to itself.
Dropping the requirement of continuity in Definition~\ref{specialization}, each of these maps would be a coordinate specialization.
\end{remark}

The following proposition shows that the requirements of $R$-linearity and continuity resolve the issues illustrated in Remarks~\ref{why linear} and~\ref{why continuous}.

\begin{prop}\label{continuous linear}
Let $\Trop_R(u_i:i\in I)$ and $\Trop_R(v_k:k\in K)$ be tropical semifields over $R$ and fix a family $(p_{ik}:i\in I,k\in K)$ of elements of $R$.
Then the following are equivalent.
\begin{enumerate}[(i)]
\item There exists a continuous $R$-linear map $\varphi:\Trop_R(u_i:i\in I)\to\Trop_R(v_k:k\in K)$ with $\varphi(u_i)=\prod_{k\in K}v_k^{p_{ik}}$ for all $i\in I$.
\item For all $k\in K$, there are only finitely many indices $i\in I$ such that $p_{ik}$ is nonzero.
\end{enumerate}
When these conditions hold, the unique map $\varphi$ is
\[\varphi\Bigl(\prod_{i\in I}u_i^{a_i}\Bigr)=\prod_{k\in K}v_k^{\sum_{i\in I}p_{ik}a_i}.\]
\end{prop}

We conclude this section by giving the details of the definition of the product topology and proving Proposition~\ref{continuous linear}.
We assume familiarity with the most basic ideas of point-set topology.
Details and proofs regarding the product topology are found, for example, in \cite[Chapter~3]{Kelley}.
Let $(X_i:i\in I)$ be a family of topological spaces, where $I$ is an arbitrary indexing set.
Let $X$ be the direct product $\prod_{i\in I}X_i$, and write the elements of $X$ as $(a_i:i\in I)$ with each $a_i\in X_i$.
Let $P_j:X\to X_j$ be the $j\th$ projection map, sending $(a_i:i\in I)$ to $a_j$.
The product topology on $X$ is the coarsest topology on $X$ such that each $P_j$ is continuous.
Open sets in $X$ are arbitrary unions of sets of the form $\prod_{i\in I}O_i$ where each $O_i$ is an open set in $X_i$ and the identity $O_i=X_i$ holds for all but finitely many $i\in I$.

We are interested in the case where each $X_i$ is $R$ with the discrete topology (meaning that every subset of $R$ is open).
We write $X_I$ for $\prod_{i\in I}X_i$ in this case.
For each $i\in I$, let $e_i\in X_I$ be the element whose entry is~$1$ in the position indexed by~$i$, with entries~$0$ in every other position.
Later in the paper, we use the bold symbol $\e_i$ for a standard unit basis vector in $\reals^n$.
Here, we intend the non-bold symbol $e_i$ to suggest a similar idea, even though a basis for $X_I$ may have many more elements than $I$.
The open sets in $X_I$ are arbitrary unions of sets of the form $\prod_{i\in I}O_i$ where each $O_i$ is a subset of $R$ and the identity $O_i=R$ holds for all but finitely many $i\in I$.
The following is a rephrasing of Proposition~\ref{continuous linear}.

\begin{prop}\label{continuous linear rewrite}
Let $X_I$ and $X_K$ be as above for indexing sets $I$ and $K$.
Fix a family $(p_{ik}:i\in I,k\in K)$ of elements of $R$.
Then the following are equivalent.
\begin{enumerate}[(i)]
\item There exists a continuous $R$-linear map $\varphi:X_I\to X_K$ with $\varphi(e_i)=(p_{ik}:k\in K)$ for all $i\in I$.
\item For all $k\in K$, there are only finitely many indices $i\in I$ such that $p_{ik}$ is nonzero.
\end{enumerate}
When these conditions hold, the unique map $\varphi$ is
\[\varphi(a_i:i\in I)=\biggl(\,\sum_{i\in I}p_{ik}a_i:k\in K\biggr).\]
\end{prop}

To simplify the proof of Proposition~\ref{continuous linear rewrite}, we appeal to the following lemma.  

\begin{lemma}\label{continuous linear one}
A map $\varphi:X_I\to X_K$ is continuous and $R$-linear if and only if the map $P_k\circ\varphi:X_I\to R$ is continuous and $R$-linear for each  $k\in K$.
\end{lemma}
The assertions of the lemma for continuity and linearity are independent of each other.
A direct proof for continuity is found in \cite[Chapter~3]{Kelley}.
The assertion for linearity is immediate because the linearity of each $P_k\circ\varphi$ is the linearity of $\varphi$ in each entry, which is equivalent to the linearity of $\varphi$.
Lemma~\ref{continuous linear one} reduces the proof of Proposition~\ref{continuous linear rewrite} to the following proposition:

\begin{prop}\label{continuous linear lite}
Let $X_I$ be as above and endow $R$ with the discrete topology.
Fix a family $(p_i:i\in I)$ of elements of $R$.
Then the following are equivalent.
\begin{enumerate}[(i)]
\item \label{contin lin lite map}
There exists a continuous $R$-linear map $\varphi:X_I\to R$ with $\varphi(e_i)=p_i$ for all $i\in I$.
\item \label{contin lin lite finite}
There are only finitely many indices $i\in I$ such that $p_i$ is nonzero.
\end{enumerate}
When these conditions hold, the unique map $\varphi$ is $\varphi(a_i:i\in I)=\sum_{i\in I}p_ia_i$.
\end{prop}

\begin{proof}
Suppose $\varphi:X_I\to R$ is a continuous $R$-linear map with $\varphi(e_i)=p_i$ for all $i\in I$.
The continuity of $\varphi$ implies in particular that $\varphi^{-1}(\set{0})$ is some open set $O$ in $X_I$.
The set $O$ is the union, indexed by some set $M$, of sets $(O_m:m\in M)$, where each $O_m$ is a Cartesian product $\prod_{i\in I}O_{mi}$ with $O_{mi}\subseteq R$ for all $i\in I$ and with the identity $O_{mi}=R$ holding for all but finitely many $i\in I$.
The set $M$ is nonempty because $R$-linearity implies that $\phi(0\cdot e_i)=0$ for any $i\in I$.
We choose $M$ and the sets $O_m$ to be maximally inclusive, in the following sense:  
If $O$ contains some set $U=\prod_{i\in I}U_i$  with $U_i\subseteq R$ for all $i\in I$ and with the identity $U_i=R$ holding for all but finitely many $i\in I$, then there exists $m\in M$ such that $U=O_m$.

For each $m\in M$, let $S_m$ be the finite set of indices $i\in I$ such that $O_{mi}\neq R$ and choose $\mu\in M$ to minimize the set $S_\mu$ under containment.
\textit{A priori}, there may not be a unique minimum set, but some minimal set exists because $M$ is nonempty and the sets $S_m$ are finite.
We now show that $p_i=0$ if and only if $i\not\in S_\mu$.

If $p_i=0$ for some $i\in S_\mu$, then for any element $z$ of $O_{\mu}$, the element $z+ce_i$ is in~$O$.
In particular, the set obtained from $O_\mu=\prod_{j\in I}O_{\mu j}$ by replacing $O_{\mu i}$ with $R$ is contained in $O$, and so equals $O_{\mu'}$ for some $\mu'\in M$.
This contradicts our choice of $\mu$ to minimize $S_\mu$, and we conclude that $p_i\neq0$ for all $i\in S_\mu$.
Now suppose $i\not\in S_\mu$.
Let $x\in O_\mu$ and let $y$ be obtained from $x$ by subtracting~$1$ in the entry indexed by~$i$.
Since $O_{\mu i}=R$, we have $y\in O_\mu$ as well.
Thus $p_i=\varphi(e_i)=\varphi(x-y)=0-0=0$.

Given \eqref{contin lin lite map}, we have established \eqref{contin lin lite finite}, and now we show that the map $\varphi$ is given by $\varphi(a_i:i\in I)=\sum_{i\in I}p_ia_i$, which is a finite sum by \eqref{contin lin lite finite}.
Let $a=(a_i:i\in I)\in X_I$.
Let $b=(b_i:i\in I)\in X_I$ have $b_i=0$ whenever $i\in I\setminus S_\mu$ and $b_i=a_i$ whenever $i\in S_\mu$.
Then $a$ and $b$ differ by an element of $O_\mu$, so $\varphi(a-b)=0$ by linearity, and thus $\varphi(a)=\varphi(b)$.
But $b$ is a finite linear combination $\sum_{i\in S_\mu}b_ie_i$, so $\varphi(a)=\varphi(b)=\sum_{i\in S_\mu}p_ib_i=\sum_{i\in S_\mu}p_ia_i$.
Since $p_i=0$ for $i\not\in S_\mu$, we have $\varphi(a)=\sum_{i\in I}p_ia_i$.

Finally, suppose \eqref{contin lin lite finite} holds, let $S$ be the finite set $\set{i\in I:p_i\neq0}$, and define $\varphi:X_I\to R$ by $\varphi(a_i:i\in I)=\sum_{i\in I}p_ia_i$.
It is immediate that this map is $R$-linear.
The map is also continuous:
Given any subset $U$ of $R$, let $O=\varphi^{-1}(U)$.
Then, for every $a=(a_i:i\in I)\in O$, let $O_a$ be the product $\prod_{i\in I}O_{ai}$ where $O_{ai}=\set{a_i}$ if $i\in S$ or $O_{ai}=R$ if $i\not\in S$.
The map $\varphi$ is constant on $O_a$, so $\set{a}\subseteq O_a\subseteq O$ for each $a\in O$.
Thus $O=\bigcup_{a\in A}O_a$, and this is an open set in $X_I$.
\end{proof}

\section{Bases for \texorpdfstring{$B$}{B}}\label{bases sec}
In this section, we define the notion of an $R$-basis for an exchange matrix $B$ and show that an extended exchange matrix is universal for $B$ over $R$ if and only if its coefficient rows constitute an $R$-basis for $B$.
This amounts to a simple rephrasing of the definition combined with a reduction to single components.

\begin{definition}[\emph{Mutation maps}]\label{map}
The exchange matrix $B$ defines a family of piecewise linear maps $(\eta^B_k:\reals^n\to\reals^n:k\in[n])$ as follows.
(The use of the real numbers $\reals$ here rather than the more general $R$ is deliberate.)
Given a vector $\a=(a_1,\ldots,a_n)\in\reals^n$, let $\tB$ be an extended exchange matrix having exchange matrix $B$ and having a single coefficient row $\a=(a_1,\ldots,a_n)$.
Let $\eta_k^B(\a)$ be the coefficient row of $\mu_k(\tB)$.
By \eqref{b mut}, $\eta_k^B(\a)=(a_1',\ldots,a_n')$, where, for each $j\in[n]$:  
\begin{equation}\label{mutation map def}
a'_j=\left\lbrace\begin{array}{ll}
-a_k&\mbox{if }j=k;\\
a_j+a_kb_{kj}&\mbox{if $j\neq k$, $a_k\ge 0$ and $b_{kj}\ge 0$};\\
a_j-a_kb_{kj}&\mbox{if $j\neq k$, $a_k\le 0$ and $b_{kj}\le 0$};\\
a_j&\mbox{otherwise.}
\end{array}\right.
\end{equation}
The map $\eta_k^B$ is a continuous, piecewise linear, invertible map with inverse $\eta_k^{\mu_k(B)}$, so in particular it is a homeomorphism from $\reals^n$ to itself.

For any finite sequence $\k=k_q,k_{q-1},\ldots,k_1$ of indices in $[n]$, let $B_1=B$ and define $B_{i+1}=\mu_{k_i}(B_i)$ for $i\in[q]$.
Equivalently, $B_{i+1}=\mu_{k_i,\ldots,k_1}(B)$ for $i\in[q]$.
As before, we index the sequence $\k$ so that the first entry in the sequence is on the right.
Define
\begin{equation}\label{eta def}
\eta_\k^B=\eta^B_{k_q,k_{q-1}\ldots,k_1}=\eta_{k_q}^{B_{q}}\circ\eta_{k_{q-1}}^{B_{q-1}}\circ\cdots\circ\eta_{k_1}^{B_1}
\end{equation}
This is again a piecewise linear homeomorphism from $\reals^n$ to itself, with inverse $\eta^{B_{q+1}}_{k_1,\ldots,k_q}$.
When $\k$ is the empty sequence, $\eta^B_\k$ is the identity map.
The maps $\eta_\k^B$ are collectively referred to as the \newword{mutation maps} associated to $B$.
\end{definition}

\begin{definition}[\emph{$B$-coherent linear relations}]\label{B coherent}
Let $B$ be an $n\times n$ exchange matrix.
Let $S$ be a finite set, let $(\v_i:i\in S)$ be vectors in $\reals^n$ and let $(c_i:i\in S)$ be elements of $R$.
Then the formal expression $\sum_{i\in S}c_i\v_i$ is a \newword{$B$-coherent linear relation with coefficients in $R$} if the equalities
\begin{eqnarray}
\label{linear eta}
&&\sum_{i\in S}c_i\eta^B_\k(\v_i)=\mathbf{0},\mbox{ and}\\
\label{piecewise eta}
&&\sum_{i\in S}c_i\mathbf{min}(\eta^B_\k(\v_i),\mathbf{0})=\mathbf{0}
\end{eqnarray}
hold for every finite sequence $\k=k_q,\ldots,k_1$ of indices in~$[n]$.
The requirement that \eqref{linear eta} holds when $\k$ is the empty sequence ensures that a $B$-coherent linear relation is in particular a linear relation in the usual sense. 
A $B$-coherent linear relation $\sum_{i\in S}c_i\v_i$ is \newword{trivial} if $c_i=0$ for all $i\in S$.
\end{definition}

The definition of an $R$-basis for $B$ is parallel to the definition of a linear-algebraic basis, with $B$-coherent linear relations replacing ordinary linear relations.

\begin{definition}[\emph{$R$-Basis for $B$}]\label{basis B def}
Let $I$ be some indexing set and let $(\b_i:i\in I)$ be a collection of vectors in $R^n$.
Then $(\b_i:i\in I)$ is an \newword{$R$-basis} for $B$ if and only if the following two conditions hold.
\begin{enumerate}[(i)]
\item \label{basis span}
If $\a\in R^n$, then there exists a finite subset $S\subseteq I$ and elements $(c_i:i\in S)$ of $R$ such that $\a-\sum_{i\in S}c_i\b_i$ is a $B$-coherent linear relation.
\item \label{basis indep}
 If $S$ is a finite subset of $I$ and $\sum_{i\in S}c_i\b_i$ is a $B$-coherent linear relation with coefficients in $R$, then $c_i=0$ for all $i\in S$.
\end{enumerate}
If condition \eqref{basis span} holds, then $(\b_i:i\in I)$ is called an \newword{$R$-spanning set for $B$}, and if condition \eqref{basis indep} holds, then $(\b_i:i\in I)$ is called an \newword{$R$-independent set for $B$}.
\end{definition}

\begin{theorem}\label{basis univ}
Let $\tB$ be an extended exchange matrix with exchange matrix $B$ and whose coefficient rows have entries in $R$.
Then $\tB$ is universal over $R$ if and only if the coefficient rows of $\tB$ are an $R$-basis for $B$.
\end{theorem}
\begin{proof}
Let $(\b_i:i\in I)$ be the coefficient rows of $\tB$.
We show that the following conditions are equivalent.
\begin{enumerate}[(a)]
\item \label{basis univ univ}
$\tB$ is universal over $R$.
\item \label{basis univ one row}
For every extended exchange matrix $\tB'$ with exactly one coefficient row, sharing the exchange matrix $B$ with $\tB$ and having entries in $R$, there exists a unique coefficient specialization from $\A_R(\tB)$ to $\A_R(\tB')$.
\item \label{basis univ coh}
 For every extended exchange matrix $\tB'$ with exactly one coefficient row $\a=(a_1,\ldots,a_n)\in R^n$, sharing the exchange matrix $B$ with $\tB$, there exists a unique choice $(p_i:i\in I)$ of elements of $R$, finitely many nonzero, such that both of the following  conditions hold:
\begin{enumerate}[(i)]
\item $\displaystyle\sum_{i\in I}p_i\b_i^t=\a^t$ for every $t\in\T_n$.
\item $\displaystyle\sum_{i\in I}p_i\mathbf{min}(\b_i^t,\mathbf{0})=\mathbf{min}(\a^t,\mathbf{0})$ for every $t\in\T_n$.
\end{enumerate}
Here $\b_i^t$ is the coefficient row of $\tB_t$ indexed by $i\in I$ in the $Y$-pattern $t\mapsto \tB_t$ with $\tB_{t_0}=\tB$ and $\a_t$ is the coefficient row of $\tB'_t$ in the $Y$-pattern $t\mapsto\tB'_t$ with $\tB'_{t_0}=\tB'$.
\end{enumerate}

If \eqref{basis univ univ} holds, then \eqref{basis univ one row} holds by definition.
Condition \eqref{basis univ one row} is the assertion that there exists a unique map satisfying Definition~\ref{specialization}\eqref{spec restr}.
By Proposition~\ref{continuous linear}, choosing a continuous $R$-linear map is equivalent to choosing elements $(p_i:i\in I)$ of $R$ with finitely many nonzero.
The remainder of Definition~\ref{specialization}\eqref{spec restr} is rephrased in \eqref{basis univ coh} as conditions (i) and (ii).
We see that Condition \eqref{basis univ coh} is a rephrasing of condition \eqref{basis univ one row}.
Now suppose \eqref{basis univ coh} holds and let $\tB''$ be an extended exchange matrix with coefficient rows indexed by an arbitrary set~$K$.
For each $k\in K$, condition \eqref{basis univ coh} implies that there is a unique choice of elements $p_{ik}$ of $R$ satisfying, in the $k\th$ component, the conditions of Proposition~\ref{continuous linear} and of Definition~\ref{specialization}.
The elements $p_{ik}$, taken together for all $k\in K$, satisfy the conditions of Propositions~\ref{continuous linear} and~\ref{specialization}, and thus define the unique coordinate specialization from $\A_R(\tB)$ to $\A_R(\tB'')$.
We have verified that \eqref{basis univ coh} implies \eqref{basis univ univ}, so that the three conditions are equivalent.

But \eqref{basis univ coh} is equivalent to the assertion that, for each $\a\in R^n$, there exists a unique finite subset $S\subseteq I$ and unique nonzero elements $(c_i:i\in S)$ of $R$ such that $\a-\sum_{i\in S}c_i\b_i$ is a $B$-coherent linear relation.
This is equivalent to the assertion that $(\b_i:i\in I)$ is an $R$-basis for $B$.
\end{proof}

\begin{remark}\label{explicit spec}
Given a universal extended exchange matrix $\tB$ over $R$ and an extended exchange matrix $\tB'$ sharing the exchange matrix $B$ with $\tB$ and having entries in $R$, the proof of Theorem~\ref{basis univ} provides an explicit description of the unique coefficient specialization from $\A_R(\x,\tB)$ to $\A_R(\x',\tB')$.
It is the map sending each $x_j$ to $x'_j$ and acting on coefficient semifields as follows:
Let $(\b_i:i\in I)$ be the coefficient rows of $\tB$ and let $(\a_k:k\in K)$ be the coefficient rows of $\tB'$.
For each $k$, there is a unique choice $(p_{ik}:i\in I)$ of elements of $R$, finitely many nonzero, such that $\a_k-\sum_{i\in S}p_{ik}\b_i$ is a $B$-coherent linear relation. 
Then the restriction of $\varphi$ to coefficient semifields is the map described in Proposition~\ref{continuous linear}.
\end{remark}

\begin{prop}\label{basis exists}
Suppose the underlying ring $R$ is a field.
For any exchange matrix $B$, there exists an $R$-basis for $B$.
Given an $R$-spanning set $U$ for $B$, there exists an $R$-basis for $B$ contained in $U$. 
\end{prop}
\begin{proof}
Let $U$ be an $R$-spanning set for $B$.
Given any chain $U_1\subseteq U_2\subseteq\cdots$ of $R$-independent sets for $B$ contained in $U$, the union of the chain is an $R$-independent set for $B$.
Thus Zorn's Lemma says that among the $R$-independent subsets of $U$, there exists a maximal set $M$.
We show that if $R$ is a field, then $M$ is an $R$-spanning set for $B$.
If not, then there exists a vector $\a\in R^n$ such that no $B$-coherent linear relation $\a-\sum_{i\in S}c_i\b_i$ over $R$ exists with each $\b_i$ in $M$.
We can assume $\a\in U$ because $U$ is an $R$-spanning set for $B$:
If each element of $U$ is a $B$-coherent combination of elements of $M$, then every element of $R^n$ is a $B$-coherent combination of elements of $M$.
If $M\cup\set{\a}$ is not an $R$-independent set for $B$, then since $M$ is an $R$-independent set for $B$, there exists a $B$-coherent linear relation $c\a-\sum_{i\in S}c_i\b_i$ over $R$ with $c\neq 0$ and each $\b_i$ in $M$.
Since $R$ is a field, $\a-\sum_{i\in S}\frac{c_i}{c}\b_i$ is a $B$-coherent linear relation over $R$, and this contradiction shows that $M\cup\set{\a}$ is an $R$-independent set for $B$.
That contradicts the maximality of $M$, and we conclude that $M$ is an $R$-spanning set for $B$.
We have proved the second statement of the proposition.
The first statement follows because $R^n$ is an $R$-spanning set for~$B$.
\end{proof}

Theorem~\ref{basis univ} and Proposition~\ref{basis exists} combine to prove the following corollary.
\begin{cor}\label{univ exists}
Suppose the underlying ring $R$ is a field.
For any exchange matrix $B$, there exists an extended exchange matrix $\tB$ with exchange matrix $B$ that is universal over~$R$.
The cluster algebra $\A_R(\tB)$ has universal geometric coefficients over $R$.
\end{cor}

\begin{remark}\label{nonconstructive}
The proof of Proposition~\ref{basis exists} echoes the standard argument showing that a vector space has a (Hamel) basis.
As in the linear algebraic case, this proof provides no general way of constructing a basis, and thus Corollary~\ref{univ exists} provides no general way of constructing a universal extended exchange matrix.
\end{remark}

\begin{remark}\label{noninteger}
The definitions in Section~\ref{defs sec} are closest to the original definitions in~\cite{ca4} when $R=\integers$, and arguably it is most important to find $\integers$-bases for exchange matrices (and thus universal geometric coefficients over $\integers$).
Unfortunately, the second assertion of Proposition~\ref{basis exists} can fail when $R=\integers$, as shown in Example~\ref{rk1} below.
We have no proof of the assertion that every $B$ admits a $\integers$-basis, but also no counterexample.

In Section~\ref{rk2 sec}, we construct $R$-bases for any $B$ of rank 2 and any $R$.
In Section~\ref{Tits sec} we construct $R$-bases for any $B$ of finite type and any $R$.
We also conjecture a form for $R$-bases for $B$ of affine type and any $R$.
In~\cite{unisurface} and~\cite{unitorus}, we use laminations to construct an $R$-basis (with $R=\integers$ or $\rationals$) for certain exchange matrices arising from marked surfaces.
\end{remark}

\begin{example}\label{rk1}
Suppose $B=[0]$.
If $R$ is a field, then $\set{x,y}$ is an $R$-basis for $B$ if and only if $x$ and $y$ are elements of $R$ with strictly opposite signs.
The set $\set{\pm 1}$ is the unique $\integers$-basis for $B$.
In particular, the extended exchange matrix $\begin{bsmallmatrix*}[r]0\\1\\-1\end{bsmallmatrix*}$ is universal over any $R$.
The set $\set{-1,2,3}$ is a $\integers$-spanning set for $B$, but contains no $\integers$-basis for $B$.
In particular, the second assertion of Proposition~\ref{univ exists} may fail without the hypothesis that $R$ is a field.
\end{example}

The remainder of this section is devoted to further details on $B$-coherent linear relations and bases.
First, in the most important cases, condition~\eqref{piecewise eta} can be ignored when verifying that a linear relation is $B$-coherent.
\begin{prop}\label{no zero no piecewise}
Suppose $B$ has no column consisting entirely of zeros.
Then the formal expression $\sum_{i\in S}c_i\v_i$ is a $B$-coherent linear relation if and only if condition~\eqref{linear eta} holds for all sequences $\k$.
\end{prop}
\begin{proof}
Suppose condition~\eqref{linear eta} holds for all $\k$ for the formal expression $\sum_{i\in S}c_i\v_i$.
We need to show that condition~\eqref{piecewise eta} holds as well.
Given any sequence $\k$ of indices in $[n]$ and any additional integer $k\in[n]$, we show that~\eqref{piecewise eta} holds in the $k\th$ component.
Write $\w_i=\eta_\k^B(\v_i)$ for each $i\in S$, let $S_{>0}$ be the subset of $S$ consisting of indices $i$ such that the $k\th$ coordinate of $\w_i$ is positive, and let $S_{\le0}=S\setminus S_{>0}$.

Consider condition~\eqref{linear eta} for the sequences $\k$ and $k\k$.
These conditions are $\sum_{i\in S_{>0}}c_i\w_i=-\sum_{i\in S_{\le0}}c_i\w_i$ and $\sum_{i\in S_{>0}}c_i\eta_k^{\mu_\k(B)}(\w_i)=-\sum_{i\in S_{\le0}}c_i\eta_k^{\mu_\k(B)}(\w_i)$.
Let $\sum_{i\in S_{>0}}c_i\w_i=(a_1,\ldots,a_n)$, so that $\sum_{i\in S_{\le0}}c_i\w_i=(-a_1,\ldots,-a_n)$.
Since all of the vectors $\w_i$ for $i\in S_{>0}$ have positive $k\th$ coordinate, \eqref{mutation map def} says that $\sum_{i\in S_{>0}}c_i\eta_k^{\mu_\k(B)}(\w_i)$ is $(a_1',\ldots,a_n')$ given by:  
\[
a'_j=\left\lbrace\begin{array}{ll}
-a_k&\mbox{if }j=k;\\
a_j+a_kb_{kj}&\mbox{if $j\neq k$ and $b_{kj}\ge 0$};\\
a_j&\mbox{otherwise,}
\end{array}\right.
\]
where $b_{kj}$ is the $kj$-entry of $\mu_\k(B)$.
Similarly, $\sum_{i\in S_{\le0}}c_i\eta_k^{\mu_\k(B)}(\w_i)$ is $(a_1'',\ldots,a_n'')$, given by:
\[
a''_j=\left\lbrace\begin{array}{ll}
a_k&\mbox{if }j=k;\\
-a_j+a_kb_{kj}&\mbox{if $j\neq k$ and $b_{kj}\le 0$};\\
-a_j&\mbox{otherwise.}
\end{array}\right.
\]
The requirement that $\sum_{i\in S_{>0}}c_i\eta_k^{\mu_\k(B)}(\w_i)=-\sum_{i\in S_{\le0}}c_i\eta_k^{\mu_\k(B)}(\w_i)$ means that $(a_1',\ldots,a_n')=-(a_1'',\ldots,a_n'')$.
Therefore $a_kb_{kj}=0$ for all $j$.
The property of having a column consisting entirely of zeros is preserved under mutation, so $\mu_\k(B)$ has no column of zeros.
Since $\mu_\k(B)$ is skew-symmetrizable, it also has no row consisting entirely of zeros.
We conclude that $a_k=0$.
In particular, we have showed that the $k\th$ coordinate of $\sum_{i\in S_{\le0}}c_i\w_i$ is zero.
This is the $k\th$ component of \eqref{piecewise eta}.
\end{proof}

We record a simple but useful observation about $B$-coherent linear relations.
\begin{prop}\label{one positive or one negative}
Let $\sum_{i\in S}c_i\v_i$ be a $B$-coherent linear relation.
Suppose, for some $i\in S$, for some $j\in[n]$, and for some sequence $\k$ of indices in $[n]$, that the $j\th$ entry of $\eta_\k^B(\v_i)$ is strictly positive (resp. strictly negative) and that the $j\th$ entry of every vector $\eta_\k^B(\v_{i'})$ with $i'\in S\setminus\set{i}$ is nonpositive (resp. nonnegative).
Then $c_i=0$.
\end{prop}
\begin{proof}
Suppose the $j\th$ entry of $\eta_\k^B(\v_i)$ is strictly positive and the $j\th$ entries of the other vectors $\eta_\k^B(\v_{i'})$ are nonpositive.
Consider the $j\th$ coordinate of~\eqref{linear eta} and~\eqref{piecewise eta}, for the chosen sequence $\k$.
The difference between the two left-hand sides is the $j\th$ coordinate of $c_i\eta_\k^B(\v_i)$, which is therefore zero.
But $\eta_\k^B(\v_i)$ has a positive entry in its $j\th$ position, so $c_i=0$.

If the $j\th$ entry of $\eta_\k^B(\v_i)$ is strictly negative and the $j\th$ entries of the other vectors are nonnegative, then the $j\th$ entry of $\eta_{j\k}^B(\v_i)$ is strictly positive and the $j\th$ entries of the other vectors $\eta_{j\k}^B(\v_{i'})$ are nonpositive, and we conclude that $c_i=0$.
\end{proof}

To further simplify the task of finding universal geometric coefficients, we conclude this section with a brief discussion of reducibility of exchange matrices.

\begin{definition}[\emph{Reducible (extended) exchange matrices}]\label{def reducible}
Call an exchange matrix $B$ \newword{reducible} if there is some permutation $\pi$ of $[n]$ such that simultaneously permuting the rows and columns of $B$ results in a block-diagonal matrix.
Otherwise call $B$ \newword{irreducible}.
\end{definition}

The following proposition is immediate, and means that we need only construct bases for irreducible exchange matrices.

\begin{prop}\label{basis reducible}
Suppose $B$ is a $p\times p$ exchange matrix and $B'$ is a $q\times q$ exchange matrix.
If $(\b_i:i\in I)$ is an $R$-basis for $B_1$ and $(\b'_j:j\in J)$ is an $R$-basis for $B_2$, then $(\b_i\times\mathbf{0}_q:i\in I)\cup(\mathbf{0}_p\times \b'_j:j\in J)$ is an $R$-basis for $\begin{bsmallmatrix}B_1&\mathbf{0}\\\mathbf{0}&B_2\end{bsmallmatrix}$.
\end{prop}
Here $\b_i\times\mathbf{0}_q$ represents the vector $\b_i\in R^p$ included into $R^n$ by adding $q$ zeros at the end, and $\mathbf{0}_p\times \b'_j$ is interpreted similarly.

Proposition~\ref{basis reducible} also allows us to stay in the case where $B$ has no column of zeros, so that the definition of $B$-coherent linear relations simplifies as explained in Proposition~\ref{no zero no piecewise}.
This is because an exchange matrix with a column of zeros is reducible and has the $1\times1$ exchange matrix $[0]$ as a reducible component.
This component is easily dealt with as explained in Example~\ref{rk1}.

\section{\texorpdfstring{$B$}{B}-cones and the mutation fan}\label{B cone sec}
In this section, we use the mutation maps associated to an exchange matrix $B$ to define a collection of closed convex real cones called $B$-cones.
These define a fan called the mutation fan for $B$.

\begin{definition}[\emph{Cones}]\label{cone def}
A \newword{convex cone} is a subset of $\reals^n$ that is closed under positive scaling and under addition.
A convex cone, by this definition, is also convex in the usual sense.
A \newword{polyhedral} cone is a cone defined by finitely many weak linear inequalities, or equivalently it is the nonnegative $\reals$-linear span of finitely many vectors.
A \newword{rational} cone is a cone defined by finitely many weak linear inequalities with integer coefficients, or equivalently it is the nonnegative $\reals$-linear span of finitely many integer vectors.
A \newword{simplicial} cone is the nonnegative span of a set of linearly independent vectors.
\end{definition}

\begin{definition}[\emph{$B$-classes and $B$-cones}]\label{B cones}
Let $B$ be an $n\times n$ exchange matrix.
Define an equivalence relation $\equiv^B$ on $\reals^n$ by setting $\a_1\equiv^B\a_2$ if and only if $\vsgn(\eta^B_\k(\a_1))=\vsgn(\eta^B_\k(\a_2))$ for every finite sequence $\k$ of indices in~$[n]$.
Recall that $\vsgn(\a)$ denotes the vector of signs of the entries of $\a$.
Thus $\a_1\equiv^B\a_2$ means that $\eta_\k^B(\a_1)\in H$ if and only if $\eta_\k^B(\a_2)\in H$ for every open coordinate halfspace $H$ of $\reals^n$ and every sequence $\k$.
The equivalence classes of $\equiv^B$ are called \newword{$B$-classes}.
The closures of $B$-classes are called \newword{$B$-cones}.
The latter term is justified in Proposition~\ref{convex}, below.
\end{definition}

\begin{prop}\label{linear}
Every mutation map $\eta_\k^B$ is linear on every $B$-cone.
\end{prop}
\begin{proof}
Let $C$ be a $B$-class and let $\k=k_q,\ldots,k_1$.
We show by induction on $q$ that the map $\eta^B_\k$ is linear on $C$.
The base case $q=0$ is trivial.
If $q>0$ then let $\k'=k_{q-1},\ldots,k_1$.
By induction, the map $\eta^B_{\k'}$ is linear on $C$.
Since $C$ is a $B$-class, $\eta^B_{\k'}$ takes $C$ to a set $C'$ on which the function $\vsgn$ is constant.
In particular, $C'$ is contained in one of the domains of linearity of $\eta_{k_q}^{\mu_{\k'}(B)}$, so $\eta^B_\k=\eta_{k_q}^{\mu_{\k'}(B)}\circ\eta^B_{\k'}$ is linear on $C$.
Since $\eta^B_\k$ is also a continuous map, it is linear on the closure of $C$.
\end{proof}

\begin{prop}\label{convex}
Each $B$-class is a convex cone containing no line and each $B$-cone is a closed, convex cone containing no line.
\end{prop}
\begin{proof}
First, notice that each $B$-class is closed under positive scaling because each map $\eta_k^B$ commutes with positive scaling, and because the function $\vsgn$ is unaffected by positive scaling.  
Furthermore, if $\a_1\equiv^B \a_2$ then $\vsgn(\eta^B_\k(\a_1))=\vsgn(\eta^B_\k(\a_2))$, for any finite sequence $\k$ of indices in~$[n]$.
Thus by Proposition~\ref{linear}, $\vsgn(\eta^B_\k(\a_1+\a_2))=\vsgn(\eta^B_\k(\a_1)+\eta^B_\k(\a_2))=\vsgn(\eta^B_\k(\a_1))$, so $\a_1\equiv^B \a_1+\a_2$.
Since $B$-classes are closed under positive scaling and addition, they are convex cones.
The requirement that $\vsgn(\eta^B_\k(\,\cdot\,))$ is constant within $B$-classes implies in particular (taking $\k$ to be the empty sequence) that each $B$-class is contained in a coordinate orthant.
In particular, no $B$-class contains a line.
We have verified the assertion for $B$-classes, and the assertion for $B$-cones follows.
\end{proof}

\begin{prop}\label{cones preserved}
Let $\k$ be a finite sequence of indices in~$[n]$.
Then a set $C$ is a $B$-class if and only if $\eta_\k^B(C)$ is a $\mu_\k(B)$-class.
A set $C$ is a $B$-cone if and only if $\eta_\k^B(C)$ is a $\mu_\k(B)$-cone.
\end{prop}
\begin{proof}
Let $C$ be a $B$-class and let $\k'=k'_{q'},\ldots,k'_1$ be another finite sequence of indices in~$[n]$.
Then since $C$ is a $B$-class, the function $\vsgn$ is constant on $\eta^B_{\k'\k}(C)$, where $\k'\k$ is the concatenation of $\k'$ and $\k$.
Letting $\k'$ vary over all possible sequences, we see that $\eta_\k^B(C)$ is contained in some $\mu_\k(B)$-class $C'$.
Symmetrically, if $\k''$ is the reverse sequence of $\k$, $\eta_{\k''}^{\mu_\k(B)}(C')$ is contained in some $B$-class $C''$.
But  $\eta_{\k''}^{\mu_\k(B)}=(\eta_\k^B)^{-1}$, so $C\subseteq C''$.
By definition, distinct $B$-classes are disjoint, so $C''=C$, and thus $C'=\eta_\k^B(C)$.
The assertion for $B$-cones follows because $\eta_\k^B$ is a homeomorphism.
\end{proof}

The following is \cite[Definition~6.12]{ca4}.
\begin{definition}[\emph{Sign-coherent vectors}]\label{sign-coherent def}
A collection $X$ of vectors in $\reals^n$ is \newword{sign-coherent} if for any $k\in[n]$, the $k\th$ coordinates of the vectors in $X$ are either all nonnegative or all nonpositive.
\end{definition}

As pointed out in the proof of Proposition~\ref{convex}, each $B$-cone is contained in some coordinate orthant.
In other words:

\begin{prop}\label{sign-coherent}
Every $B$-cone is a sign-coherent set.
\end{prop}

Our understanding of $B$-cones allows us to mention the simplest kind of $B$-coherent linear relation.

\begin{definition}[\emph{$B$-local linear relations}]\label{B local}
Let $B$ be an $n\times n$ exchange matrix.
Let $S$ be a finite set, let $(\v_i:i\in S)$ be vectors in $\reals^n$ and let $(c_i:i\in S)$ be real numbers.
Then the formal expression $\sum_{i\in S}c_i\v_i$ is a \newword{$B$-local linear relation} if the equality $\sum_{i\in S}c_i\v_i=\mathbf{0}$ holds and if $\set{\v_i:i\in S}$ is contained in some $B$-cone.
\end{definition}

\begin{prop}\label{local coherent} 
A $B$-local linear relation is $B$-coherent.
\end{prop}
\begin{proof}
Let $\sum_{i\in S}c_i\v_i$ be a $B$-local linear relation.
Then by definition, \eqref{linear eta} holds for $\k$ empty.
Now Proposition~\ref{linear} implies that \eqref{linear eta} holds for all $\k$.
Propositions~\ref{cones preserved} and~\ref{sign-coherent} imply that, for any $\k$, each coordinate of \eqref{piecewise eta} is either the tautology $0=0$ or agrees with some coordinate of \eqref{linear eta}.
\end{proof}

In order to study the collection of all $B$-cones, we first recall some basic definitions from convex geometry.

\begin{definition}[\emph{Face}]\label{face def}
A subset $F$ of a convex set $C$ is a \newword{face} if $F$ is convex and if any line segment $L\subseteq C$ whose interior intersects $F$ has $L\subseteq F$.
In particular, the empty set is a face of $C$ and $C$ is a face of itself.
Also, if $H$ is any hyperplane such that $C$ is contained in one of the two closed halfspaces defined by $H$, then $H\cap C$ is a face of $C$.
The intersection of an arbitrary set of faces of $C$ is another face of $C$.
A face of a closed convex set is closed.
\end{definition}

\begin{definition}[\emph{Fan}]\label{fan def}
A \newword{fan} is a collection $\F$ of closed convex cones such that if $C\in\F$ and $F$ is a face of $C$, then $F\in\F$, and such that the intersection of any two cones in $\F$ is a face of each of the two.
In some contexts, a fan is required to have finitely many cones, but here we allow infinitely many cones.
A fan is \newword{complete} if the union of its cones is the entire ambient space.
A \newword{simplicial fan} is a fan all of whose cones are simplicial.
A \newword{subfan} of a fan $\F$ is a subset of $\F$ that is itself a fan.
If $\F_1$ and $\F_2$ are complete fans such that every cone in $\F_2$ is a union of cones in $\F_1$, then we say that $\F_1$ \newword{refines} $\F_2$ or equivalently that $\F_2$ \newword{coarsens} $\F_1$.
\end{definition}

\begin{definition}[\emph{The mutation fan for $B$}]\label{mut fan def}
Let $\F_B$ be the collection consisting of all $B$-cones, together with all faces of $B$-cones.
This collection is called the \newword{mutation fan for $B$}.
The name is justified by the following theorem.
\end{definition}

\begin{theorem}\label{fan}
The collection $\F_B$ is a complete fan.
\end{theorem}

To prove Theorem~\ref{fan}, we introduce some additional background on convex sets and prove several preliminary results.

\begin{definition}[\emph{Relative interior}]\label{relint def}
The \newword{affine hull} of a convex set $C$ is the union of all lines defined by two distinct points of $C$.
The \newword{relative interior} of $C$, written $\relint(C)$, is its interior as a subset of its affine hull, and $C$ is \newword{relatively open} if it equals its relative interior.
The relative interior of a convex set is nonempty.  
(See e.g.\ \cite[Theorem~2.3.1]{Webster}.)
\end{definition}

We need several basic facts about convex sets.
Proofs of the first four are found, for example, in \cite[Theorem~2.3.4]{Webster}, \cite[Theorem~2.3.8]{Webster}, and \cite[Corollary~2.4.11]{Webster}.

\begin{lemma}\label{Web lem}
Let $C$ be a convex set in $\reals^n$.
Let $\x$ be in the closure of $C$, let $\y$ be in the relative interior of $C$ and let $\ep\in[0,1)$.
Then $\ep\x+(1-\ep)\y$ is in the relative interior of $C$.
\end{lemma}

\begin{lemma}\label{relint clos}
The relative interior of a convex set $C$ in $\reals^n$ equals the relative interior of the closure of $C$.
\end{lemma}

\begin{lemma}\label{clos relint}
A closed convex set in $\reals^n$ is the closure of its relative interior.
\end{lemma}

\begin{lemma}\label{sep thm}
If $C$ and $D$ are nonempty convex sets in $\reals^n$ with disjoint relative interiors, then there exists a hyperplane $H$, defining halfspaces $H_+$ and $H_-$ such that $C\subseteq H_+$, $D\subseteq H_-$ and $C\cup D\not\subseteq H$.
\end{lemma}

\begin{lemma}\label{int faces}
Let~$F$, $G$, $M$, and~$N$ be convex sets in $\reals^n$ such that~$F$ is a face of~$M$ and $G$ is a face of $N$.
Suppose $M\cap N$ is a face of $M$ and of $N$.
Then $F\cap G$ is a face of $F$ and of $G$.
\end{lemma}
\begin{proof}
If $F\cap G$ is empty, then we are done.
Otherwise, let $L$ be any line segment contained in $F$ whose relative interior intersects $F\cap G$.
Then $L$ is in particular a line segment in $M$ whose relative interior intersects $M\cap N$.
By hypothesis, $M\cap N$ is a face of $M$, so $L$ is contained in $M\cap N$.
Since $G$ is a face of $N$, since $L$ is contained in $N$, and since the relative interior of $L$ intersects $G$, we see that $L$ is contained in $G$.
Thus $L$ is contained in $F\cap G$.
We have shown that $F\cap G$ is a face of $F$, and the symmetric argument shows that $F\cap G$ is a face of $G$.
\end{proof}

\begin{lemma}\label{relint sign-coherent}
If $C$ is a sign-coherent convex set then $\vsgn$ is constant on $\relint(C)$.
\end{lemma}
\begin{proof}
Let $\x,\y\in\relint(C)$ and suppose $\sgn(x_i)\neq\sgn(y_i)$ for some index $i\in[n]$.
Since $\x$ and $\y$ are in $\relint(C)$, the relative interior of $C$ contains an open interval about $\x$ in the line containing $\x$ and~$\y$.
If $\sgn(x_i)=0$, then $\sgn(y_i)\neq0$, and thus this interval about $\x$ contains points whose $i\th$ coordinate have all possible signs.
Arguing similarly if $\sgn(y_i)=0$, we see that in any case, $\relint(C)$ contains points whose $i\th$ coordinates have opposite signs.
This contradicts the sign-coherence of~$C$.
\end{proof}

We now prove some preliminary results about $B$-cones.

\begin{prop}\label{relint class easy}
Every $B$-cone $C$ is the closure of a unique $B$-class, and this $B$-class contains $\relint(C)$.
\end{prop}
\begin{proof}
The $B$-cone $C$ is the closure of some $B$-class $C'$.
In light of Proposition~\ref{convex}, Lemma~\ref{relint clos} says that $\relint(C')=\relint(C)$.
In particular, $\relint(C)\subseteq C'$.
If $C$ is the closure of some other $B$-class $C''$, then also $C''$ contains $\relint(C)$.
Since $\relint(C)$ is nonempty and since distinct $B$-classes are disjoint, $C'=C''$.
\end{proof}

To further describe $B$-cones, we will use a partial order on sign vectors that appears in the description of the (``big'') face lattice of an oriented matroid.
(See \cite[Section~4.1]{OrientedMatroids}.)
\begin{definition}[\emph{A partial order on sign vectors}]\label{vsgn poset def}
A \newword{sign vector} is an $n$-tuple whose entries are in $\set{-1,0,1}$, or in other words, a vector that arises by applying the operator $\vsgn$ to a vector in $\reals^n$.
For sign vectors $\x$ and $\y$, say $\x\preceq\y$ if $\x$ agrees with $\y$, except possibly that some entries $1$ or $-1$ in $\y$ become $0$ in $\x$.
The point of this definition is to capture a notion of limits of sign vectors:
If $\v$ is the limit of a sequence or continuum of vectors all having the same sign vector $\y$, then $\vsgn(\v)\preceq\y$.
\end{definition}

\begin{prop}\label{Bcone in out}
Let $C'$ be a $B$-class whose closure is the $B$-cone $C$.
Let $\y$ be any point in $C'$.
Then $C$ is the set of points $\x$ such that $\vsgn(\eta_\k^B(\x))\preceq\vsgn(\eta_\k^B(\y))$ for all sequences~$\k$ of indices in $[n]$.
\end{prop}
\begin{proof}
Let $\y'$ be any point in $\relint(C)$.
Proposition~\ref{relint class easy} says that $C'$ contains $\relint(C)$, so the vectors $\vsgn(\eta_\k^B(\y'))$ and $\vsgn(\eta_\k^B(\y))$ agree for any sequence $\k$.

Suppose $\x\in C$ and let $\k$ be a sequence of indices.
By Propositions~\ref{linear} and~\ref{cones preserved}, the point $\eta_\k^B(\y')$ is in the relative interior of $\eta_\k^B(C)$ and the point $\eta_\k^B(\x)$ is in $\eta^B_\k(C)$, which is the closure of $\eta_\k^B(C')$.
By Lemma~\ref{Web lem}, the entire line segment from $\eta_\k^B(\y')$ to $\eta_\k^B(\x)$, except possibly the point $\eta^B_\k(\x)$, is in the relative interior of $\eta_\k^B(C)$, and thus in the $\mu_\k(B)$-class $\eta_\k^B(C')$ by Proposition~\ref{relint class easy}.
Therefore $\vsgn(\,\cdot\,)$ is constant on the line segment, except possibly at $\eta^B_\k(\x)$, so $\vsgn(\eta_\k^B(\x))\preceq\vsgn(\eta_\k^B(\y'))=\vsgn(\eta_\k^B(\y))$.

On the other hand, suppose $\vsgn(\eta_\k^B(\x))\preceq\vsgn(\eta_\k^B(\y))$ for all sequences~$\k$ of indices in $[n]$.
Then the function $\vsgn$ is constant on the relative interior of line segment defined by $\eta_\k^B(\x)$ and $\eta_\k^B(\y')$, and this constant sign vector equals $\vsgn(\eta_\k^B(\y))$.
Let $L$ be the line segment defined by $\x$ and $\y'$.
We show by induction that $\eta_\k^B$ is linear on $L$, so that $\eta_\k^B(L)$ is the line segment defined by $\eta_{\k}^B(\x)$ and $\eta_{\k}^B(\y')$.
The base case where $\k$ is the empty sequence is easy because $\eta^B_\k$ is the identity map.
If $\k$ is not the empty sequence, then write $\k=j\k'$ for some index $j$ and some sequence $\k'$.
By induction on the length of the sequence $\k$, the map $\eta_{\k'}^B$ is linear on $L$, so $\eta_{\k'}^B(L)$ is the line segment defined by $\eta_{\k'}^B(\x)$ and $\eta_{\k'}^B(\y')$.
Since $\vsgn$ is constant on the relative interior of $\eta^B_{\k'}(L)$, the map $\eta_\k^B$ is also linear on $L$ and thus $\eta_\k^B(L)$ is the line segment defined by $\eta_{\k'}^B(\x)$ and $\eta_{\k'}^B(\y')$.

We have shown that $\vsgn(\eta^B_\k(\,\cdot\,))$ is constant and equals $\vsgn(\eta_\k^B(\y))$ on $\relint(L)$ for all $\k$.
Thus $\relint(L)$ is in $C'$.
All of $L$, including $\x$, is therefore in the closure of $C'$, which is the $B$-cone $C$.
\end{proof}

\begin{prop}\label{union of classes}
Every $B$-cone is a union of $B$-classes.
\end{prop}
\begin{proof}
For any $B$-class $C$, Proposition~\ref{Bcone in out} implies that a $B$-class $D$ is either completely contained in $C$ or disjoint from $C$.
\end{proof}

The following proposition generalizes Proposition~\ref{relint class easy}.

\begin{prop}\label{relint class}
Let $C$ be a $B$-cone and let $F$ be a nonempty convex set contained in $C$.
Then the relative interior of $F$ is contained in a $B$-class $D$ with $D\subseteq C$.
\end{prop}

\begin{proof}
Let $\k$ be some sequence of indices in $[n]$.
By Proposition~\ref{cones preserved}, $\eta^B_\k(C)$ is a $\mu_\k(B)$-cone.
By Proposition~\ref{linear}, $\eta^B_\k(C)$ is the image of $C$ under a linear map.
Thus $\eta^B_\k(F)$ is a nonempty convex set contained in $\eta^B_\k(C)$, and $\eta^B_\k$ maps the relative interior of $F$ to the relative interior of $\eta^B_\k(F)$.

Proposition~\ref{sign-coherent} says that $\eta^B_\k(F)$ is sign-coherent, so Lemma~\ref{relint sign-coherent} says that $\vsgn$ is constant on the relative interior of $\eta^B_\k(F)$.
Equivalently, $\vsgn\circ\eta^B_\k$ is constant on the relative interior of $F$.
Since this was true for any $\k$, the relative interior of $F$ is contained in some $B$-class $D$.
Since $D$ is a $B$-class which intersects~$C$, Proposition~\ref{union of classes} says that $D$ is contained in $C$.
\end{proof}

\begin{prop}\label{int of B-cones}
An arbitrary intersection of $B$-cones is a $B$-cone.
\end{prop}
\begin{proof}
Let $I$ be an arbitrary indexing set and, for each $i\in I$, let $C_i$ be a $B$-cone.
Then $C=\bigcap_{i\in I}C_i$ is non-empty because it contains the origin.
Furthermore, it is a closed convex cone because it is an intersection of closed convex cones.
Let $i\in I$.
Then Proposition~\ref{relint class} says that the relative interior of $C$ is contained in a $B$-class $D$ with $D\subseteq C_i$.
Since distinct $B$-classes are disjoint, applying Proposition~\ref{relint class} to each $i$, we obtain the same $D$, and we conclude that $D$ is contained in $C$.
By Lemma~\ref{clos relint}, $C$ is the closure of $\relint(C)$.
Since $\relint(C)\subseteq D\subseteq C$ and $C$ is closed, we conclude that $C$ is the closure of $D$.
Thus $C$ is a $B$-cone.
\end{proof}

\begin{lemma}\label{B-subcone face}
If $C$ and $D$ are $B$-cones with $C\subseteq D$, then $C$ is a face of $D$.
\end{lemma}
\begin{proof}
Let $F$ be the intersection of all faces of $D$ which contain $C$.
Since $D$ is itself a face of $D$, this intersection is well-defined.
Then $F$ is a face of $D$, and we claim that the relative interior of $C$ intersects the relative interior of $F$.

Suppose to the contrary that the relative interior of $C$ does not intersect the relative interior of $F$.
Then Lemma~\ref{sep thm} says that there exists a hyperplane $H$, defining halfspaces $H_+$ and $H_-$ such that $C\subseteq H_+$, $F\subseteq H_-$ and $C\cup F\not\subseteq H$.
Since $F$ contains $C$, we have $C\subseteq H_-$, so $C\subseteq H_+\cap H_-=H$.
Therefore $F\not\subseteq H$.
Now $H\cap F$ is a proper face of $F$ containing $C$, contradicting our choice of $F$ as the intersection of all faces of $D$ which contain $C$.
This contradiction proves the claim.

By Proposition~\ref{relint class}, there is some $B$-class $F'$ containing $\relint(F)$.
By Proposition~\ref{relint class easy}, $C$ is the closure of a $B$-class $C'$ containing $\relint(C)$.
By the claim, $\relint(C)\cap\relint(F)$ is nonempty, so $F'\cap C'$ is nonempty.
Since distinct $B$-classes are disjoint, we conclude that $F'=C'$.
Thus $\relint(F)\subseteq C'$.
Proposition~\ref{clos relint} says that $F$ is the closure of $\relint(F)$, so applying closures to the relation $\relint(F)\subseteq C'$, we see that $F\subseteq C$.
By construction $C\subseteq F$, so $C=F$, and thus $C$ is a face of~$D$.
\end{proof}

\begin{remark}\label{face B cone}
One can phrase a converse to Lemma~\ref{B-subcone face}:
If $D$ is a $B$-cone and $C$ is a face of $D$, then $C$ is a $B$-cone.
This statement is false; a counterexample appears in the proof of Proposition~\ref{rk2 wild basis}.
\end{remark}

We are now prepared to prove the main theorem of this section.

\begin{proof}[Proof of Theorem~\ref{fan}]
By Proposition~\ref{convex}, each element of $\F_B$ is a convex cone.
By the definition of $\F_B$, if $C\in \F_B$ then all faces of $C$ are also in $\F_B$.
Thus, to show that $\F_B$ is a fan, it remains to show that the intersection of any two cones in $\F_B$ is a face of each.
Since every cone in $\F_B$ is a face of some $B$-cone, Lemma~\ref{int faces} says that it is enough to consider the case where both cones are $B$-cones.
If $C$ and $D$ are distinct $B$-cones, then Proposition~\ref{int of B-cones} says that $C\cap D$ is a $B$-cone.
Lemma~\ref{B-subcone face} says that $C\cap D$ is a face of $C$ and that $C\cap D$ is a face of $D$.

Since the union of all $B$-classes is all of $\reals^n$, the union of the $B$-cones is also all of $\reals^n$.
Thus $\F_B$ is complete.
\end{proof}

We conclude the section with two more useful facts about $B$-cones.

\begin{prop}\label{smallest Bcone}
For any $\a\in\reals^n$, there is a unique smallest $B$-cone containing~$\a$.
This $B$-cone is the closure of the $B$-class of $\a$.
\end{prop}
\begin{proof}
By Proposition~\ref{int of B-cones}, the intersection of all $B$-cones containing $\a$ is a $B$-cone~$C$.
Proposition~\ref{union of classes} says that any $B$-cone containing $\a$ contains the entire $B$-class of~$\a$, so $C$ is the closure of the $B$-class of $\a$.
\end{proof}

\begin{prop}\label{contained Bcone}
A set $C\subseteq\reals^n$ is contained in some $B$-cone if and only if the set $\eta_\k^B(C)$ is sign-coherent for every sequence $\k$ of indices in $[n]$.
\end{prop}
\begin{proof}
The ``only if'' direction follows immediately from Proposition~\ref{Bcone in out}.
Suppose the set $\eta_\k^B(C)$ is sign-coherent for every sequence $\k$ of indices in $[n]$.
Then the set of finite nonnegative linear combinations of vectors in $\eta_\k^B(C)$ is sign-coherent for every~$\k$.
Let $D$ be the set of finite nonnegative linear combinations of vectors in~$C$.
Arguing by induction as in the second half of the proof of Proposition~\ref{Bcone in out}, we see that every map $\eta_\k^B$ is linear on $D$, so that $\eta_\k^B(D)$ is the set of finite nonnegative linear combinations of vectors in $\eta_\k^B(C)$.
Therefore, for all $\k$, the set $\eta_\k^B(D)$ is sign-coherent, so $\vsgn(\eta_\k^B(\,\cdot\,))$ is constant on $\relint(D)$ by Lemma~\ref{relint sign-coherent}.
Thus $\relint(D)$ is contained in some $B$-class and $D$ is contained in the closure of that $B$-class.
\end{proof}

Proposition~\ref{contained Bcone} suggests a method for computing approximations to the mutation fan $\F_B$.
The proposition implies that vectors $\x$ and $\y$ are \emph{not} contained in a common $B$-cone if and only if, for some $\k$ and $j$, they are strictly separated by the image, under $(\eta_\k^B)^{-1}$, of the $j\th$ coordinate hyperplane.
Thus we approximate $\F_B$ by computing these inverse images for all sequences $\k$ of length up to some $m$.
The inverse images define a decomposition of $\reals^n$ that may be coarser than $\F_B$ but that approaches $\F_B$ as $m\to\infty$.

\begin{example}\label{nonsimp example}
The approximation to $\F_B$, for $B=\begin{bsmallmatrix*}[r]0&-3&0\\2&0&2\\0&-3&\hspace{6pt}0\\\end{bsmallmatrix*}$ and $m=9$, is shown in Figure~\ref{nonsimp fig}.
\begin{figure}[ht]
\scalebox{1}{\includegraphics{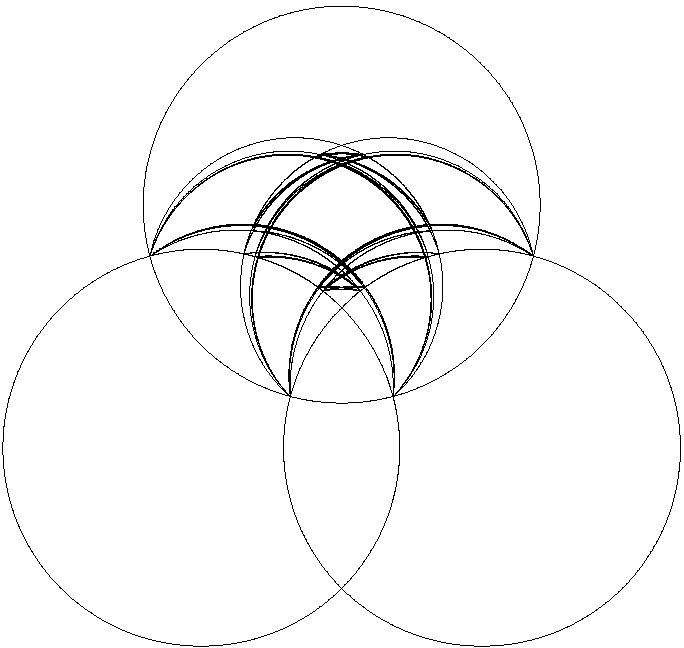}}
\begin{picture}(0,0)(164,-135)
\put(10,-12){$\e_1$}
\put(-7,19){$\e_2$}
\put(-25,-12){$\e_3$}
\end{picture}
\caption[The mutation fan $\F_B$]
{The mutation fan $\F_B$ for $B=\begin{bsmallmatrix*}[r]0&-3&0\\2&0&2\\0&-3&\hspace{6pt}0\\\end{bsmallmatrix*}$
}
\label{nonsimp fig}
\end{figure}
The picture is interpreted as follows: 
Intersecting each inverse image of a coordinate plane with a unit sphere about the origin, we obtain a collection of arcs of great circles.
These are depicted in the plane by stereographic projection.
The projections of the unit vectors $\e_1$, $\e_2$, and $\e_3$ are labeled.
We suspect that the differences between this approximation and $\F_B$ are unnoticeable at this resolution.
\end{example}

\section{Positive bases and cone bases}\label{pos cone sec}
In this section, we discuss two special properties that an $R$-basis for $B$ may have.
One of these properties is a notion of positivity.
It is not clear what consequences positivity has for cluster algebras, but it is a very natural notion for $R$-bases and for coefficient specializations.
Many of the bases that have been constructed, here and in \cite{unisurface,unitorus}, are positive.
The second special property, a condition on the interaction of the basis with the mutation fan, follows from the positivity property.

\begin{definition}[\emph{Positive basis and positive universal extended exchange matrix}]\label{positive univ tB}
An $R$-basis $(\b_i:i\in I)$ for $B$ is \newword{positive} if, for every $\a\in R^n$, there exists a finite subset $S\subseteq I$ and \emph{positive} elements $(c_i:i\in S)$ of $R$ such that $\a-\sum_{i\in S}c_i\b_i$ is a $B$-coherent linear relation.
The corresponding universal extended exchange matrix for $B$ is also called \newword{positive} in this case.
Looking back at Remark~\ref{explicit spec}, we see that a positive universal extended exchange matrix $\tB$ has a special property with respect to coefficient specializations:
Suppose $\tB'$ is another extended exchange matrix sharing the same exchange matrix with $\tB$.
When we describe the unique coefficient specialization $\varphi$ from $\A_R(\tB)$ to $\A_R(\tB')$ as in Proposition~\ref{continuous linear}, all of the $p_{ik}$ are nonnegative.
\end{definition}

\begin{prop}\label{positive unique}
For any fixed $R$, there is at most one positive $R$-basis for $B$, up to scaling each basis element by a positive unit in $R$.
\end{prop}
\begin{proof}
Suppose $(\a_i:i\in I)$ and $(\b_j:j\in J)$ are positive $R$-bases for $B$.
Given an index $i\in I$, there is a $B$-coherent linear relation $\a_i-\sum_{j\in S}c_j\b_j$ with positive coefficients $c_j\in R$.
For each $j\in S$, there is a $B$-coherent linear relation $\b_j-\sum_{k\in S_j}d_{jk}\a_k$ with positive coefficients $d_{jk}\in R$.
Combining these, we obtain a $B$-coherent relation $\a_i-\sum_{j\in S}\sum_{k\in S_j}c_jd_{jk}\a_k$.
Since $(\a_i:i\in I)$ is an $R$-basis for $B$, this $B$-coherent relation is trivial.
Thus the summation over $j\in S$ and $k\in S_j$ has only one term $\a_i$.
In particular, the set $S$ has exactly one element $j$, the set $S_j$ has only one element $i$, and $c_jd_{ji}=1$.
We see that $\a_i=c_j\b_j$ for some $j\in J$, where $c_j$ is a positive unit with inverse $d_{ji}$.
Thus every basis element $\a_i$ is obtained from a basis element $b_j$ by scaling by a positive unit.
\end{proof}

\begin{definition}[\emph{Cone basis for $B$}]\label{cone basis def}
Let $I$ be some indexing set and let $(\b_i:i\in I)$ be a collection of vectors in $R^n$.
Then $(\b_i:i\in I)$ is a \newword{cone $R$-basis} for $B$ if and only if the following two conditions hold.
\begin{enumerate}[(i)]
\item \label{cone basis span}
 If $C$ is a $B$-cone, then the $R$-linear span of $\set{\b_i:i\in I}\cap C$ contains $R^n\cap C$.
\item \label{cone basis indep}
 $(\b_i:i\in I)$ is an $R$-independent set for $B$.
\end{enumerate}
\end{definition}

\begin{prop}\label{cone basis basis}
If $(\b_i:i\in I)$ is a cone $R$-basis for $B$, then $(\b_i:i\in I)$ is an $R$-basis for $B$.
An $R$-basis is a cone $R$-basis if and only if its restriction to each $B$-cone $C$ is a basis (in the usual sense) for the $R$-linear span of the vectors in $R^n\cap C$.
\end{prop}
\begin{proof}
Condition \eqref{basis indep} of Definition~\ref{basis B def} is identical to condition \eqref{cone basis indep} of Definition~\ref{cone basis def}.
Suppose $(\b_i:i\in I)$ is a cone $R$-basis for $B$.
Let $\a\in R^n$ and let $C$ be the $B$-cone that is the closure of the $B$-class of $\a$.
By condition \eqref{cone basis span} of Definition~\ref{cone basis def}, $\a$ is an $R$-linear combination $\sum_{i\in S}c_i\b_i$ of vectors in $\set{\b_i:i\in I}\cap C$.
The linear relation $\a-\sum_{i\in S}c_i\b_i$ is $B$-local and thus $B$-coherent by Proposition~\ref{local coherent}.
We have established condition \eqref{basis span} of Definition~\ref{basis B def}, thus proving the first assertion of the proposition.

Still supposing $(\b_i:i\in I)$ to be a cone $R$-basis for $B$, let $C$ now be any $B$-cone.
If there is some non-trivial linear relation among the vectors in $\set{\b_i:i\in I}\cap C$, then that relation is $B$-local and thus $B$-coherent. 
This contradicts condition \eqref{cone basis indep} of Definition~\ref{cone basis def}, so $\set{\b_i:i\in I}\cap C$ is linearly independent.
Thus $\set{\b_i:i\in I}\cap C$ is a basis for its $R$-linear span, which equals the $R$-linear span of $C$ by condition \eqref{cone basis span} in Definition~\ref{cone basis def}.

On the other hand, suppose $(\b_i:i\in I)$ is an $R$-basis for $B$ whose restriction to each $B$-cone is a basis (in the usual sense) for the span of that $B$-cone.
Then condition \eqref{cone basis span} of Definition~\ref{cone basis def} holds. 
\end{proof}

\begin{remark}\label{explicit spec cone basis}
Suppose $\tB$ is a universal extended exchange matrix over $R$ whose coefficient rows constitute a \emph{cone} $R$-basis.
Then coefficient specializations from $\tB$ can be found more directly than in the case where we don't necessarily have a cone basis.
(See Remark~\ref{explicit spec}.) 
Each $\a_k$ is in some $B$-cone $C$.
Knowing that we have a cone basis, we can choose the elements $(p_{ik}:i\in I)$ by taking $p_{ik}=0$ for $\b_i\not\in C$ and choosing the remaining $p_{ik}$ (uniquely) so that $\a_k-\sum_{i\in S}p_{ik}\b_i$ is a linear relation in the usual sense.

For every exchange matrix $B$ for which the author has constructed an $R$-basis, the $R$-basis is in fact a cone $R$-basis.
The following question thus arises. 
\begin{question}\label{basis cone basis?}
If $(\b_i:i\in I)$ is an $R$-basis for $B$ is it necessarily a cone $R$-basis?
\end{question}
\end{remark}

We now relate the notion of a cone $R$-basis to the notion of a positive $R$-basis.

\begin{prop}\label{positive cone basis}
Given a collection $(\b_i:i\in I)$ of vectors in $R^n$, the following conditions are equivalent.
\begin{enumerate}[(i)]
\item \label{positive cone basis pos}
$(\b_i:i\in I)$ is a positive $R$-basis for $B$.
\item \label{positive cone basis cone}
$(\b_i:i\in I)$ is a positive cone $R$-basis for $B$.
\item \label{positive cone basis property}
$(\b_i:i\in I)$ is an $R$-independent set for $B$ with the following property:
If $C$ is a $B$-cone, then the nonnegative $R$-linear span of $\set{\b_i:i\in I}\cap C$ contains $R^n\cap C$.
\end{enumerate}
\end{prop}
\begin{proof}
Condition \eqref{positive cone basis cone} implies condition \eqref{positive cone basis pos} trivially.
Conversely, suppose \eqref{positive cone basis pos} holds, let $C$ be a $B$-cone and let $\a\in R^n\cap C$.
To show that \eqref{positive cone basis cone} holds, we need to show that the $R$-linear span of $\set{\b_i:i\in I}\cap C$ contains $\a$.

Since $(\b_i:i\in I)$ is a positive $R$-basis for $B$, there is a $B$-coherent linear relation $-\a+\sum_{i\in S}c_i\b_i$ with the $c_i$ positive.
Proposition~\ref{smallest Bcone} says that there is a unique smallest $B$-cone $D$ containing~$\a$, and that $D$ is the closure of the $B$-class of $\a$.
We claim that $\set{\b_i:i\in S}$ is contained in~$D$.
Otherwise, for some $i\in S$, Proposition~\ref{Bcone in out} says that there is a sequence $\k$ such that $\vsgn(\eta_\k^B(\b_i))\not\preceq\vsgn(\eta_\k^B(\a))$.
That is, there is some index $j\in[n]$ such that the $j\th$ coordinate of $\vsgn(\eta_\k^B(\b_i))$ is nonzero and different from the $j\th$ coordinate of $\vsgn(\eta_\k^B(\a))$. 
Possibly replacing $\k$ by $j\k$, we can assume that the $j\th$ coordinate of $\vsgn(\eta_\k^B(\b_i))$ is negative and the $j\th$ coordinate of $\vsgn(\eta_\k^B(\a))$ is nonnegative.
Thus we write $S$ as a disjoint union $S_1\cup S_2$, with $S_1$ nonempty, such that $\eta^B_\k(\b_i)$ has negative $j\th$ coordinate for $i\in S_1$, such that $\eta^B_\k(\b_i)$ has nonnegative $j\th$ coordinate for $i\in S_2$, and such that $\a$ has nonnegative $j\th$ coordinate.  
Since $-\a+\sum_{i\in S}c_i\b_i$ is $B$-coherent, we appeal to \eqref{piecewise eta} to conclude that the $j\th$ coordinate of $\sum_{i\in S_1}c_i\eta^B_\k(\b_i)$ is zero.
But since the $c_i$ are all positive, we have reached a contradiction, thus proving the claim.
The claim shows that the $B$-coherent linear relation $-\a+\sum_{i\in S}c_i\b_i$ in particular writes $\a$ as a positive linear combination of elements of $D$.
Since $D\subseteq C$, we have established \eqref{positive cone basis cone}.
Since the relation is positive, we have also shown that \eqref{positive cone basis pos} implies \eqref{positive cone basis property}.

If \eqref{positive cone basis property} holds, then for any $\a$ in $R^n$, there exists a finite subset $S\subseteq I$ and positive elements $(c_i:i\in S)$ of $R$ such that $\a-\sum_{i\in S}c_i\b_i$ is a $B$-local linear relation.
Proposition~\ref{local coherent} implies that $(\b_i:i\in I)$ is a positive $R$-basis for $B$, so \eqref{positive cone basis pos} holds.
\end{proof}

\begin{remark}\label{cone and pos}
A positive basis is necessarily a cone basis, by Proposition~\ref{positive cone basis}, so explicit coefficient specializations from the corresponding universal extended exchange matrix are found as described in Remark~\ref{explicit spec cone basis}.
In Section~\ref{rk2 sec}, there are examples of exchange matrices $B$ having a cone $R$-basis but not a positive $R$-basis for $R=\integers$ or $R=\rationals$.
\end{remark}

The existence of a positive $R$-basis has implications for the structure of the mutation fan.
We describe these implications by constructing another fan closely related to $\F_B$.

\begin{definition}[\emph{$R$-part of a fan}]\label{R-part def}
Suppose $\F$ is a fan and $R$ is an underlying ring.
Suppose $\F'$ is a fan satisfying the following conditions:
\begin{enumerate}[(i)]
\item \label{R cones R part}
Each cone in $\F'$ is the nonnegative $\reals$-linear span of finitely many vectors in~$R^n$.
(For example, if $R=\integers$ or $\rationals$, then these are rational cones.)
\item \label{cone in cone R part}
Each cone in $\F'$ is contained in a cone of $\F$.
\item \label{max cone R part}
For each cone $C$ of $\F$, there is a unique largest cone (under containment) among cones of $\F'$ contained in $C$.
This largest cone contains $R^n\cap C$.
\end{enumerate}
Then $\F'$ is called the \newword{$R$-part} of $\F$.
The $R$-part of $\F$ might not exist.
For example, the fans discussed later in Proposition~\ref{rk2 wild basis} have no $\rationals$-part.
However, if the $R$-part of $\F$ exists, then it is unique:  
For each cone $C$ of $\F$, condition \eqref{max cone R part} implies that the largest cone of $\F'$ contained in $C$ is the nonnegative $\reals$-linear span of $R^n\cap C$.
Condition~\eqref{max cone R part} implies that every cone in $\F'$ is a face of one of these largest cones.
\end{definition}

\begin{prop}\label{R part prop}
Suppose $\F$ is a fan, $R$ is an underlying ring, and $\F'$ is the $R$-part of $\F$.
Every cone of $\F$ spanned by vectors in $R^n$ is a cone in $\F'$.
The full-dimensional cones in $\F$ are exactly the full-dimensional cones in $\F'$.
\end{prop}
\begin{proof}
Suppose $C$ is a cone of $\F$ spanned by vectors in $R^n$.
Condition \eqref{max cone R part} of Definition~\ref{R-part def} says that there is a cone $D$ of $\F'$ contained in $C$ and containing $R^n\cap C$.
But then $D$ contains the vectors spanning $C$, so $C=D$.

Suppose $C$ is a full-dimensional cone of $\F$.
Since $\rationals^n$ is dense in $\reals^n$, if $C$ is strictly larger than the largest cone $D$ of $\F'$ contained in $C$, then there are rational vectors (and thus integer vectors and thus vectors in $R^n$) in $C\setminus D$.
This is a contradiction to condition~\eqref{max cone R part} of Definition~\ref{R-part def}.
Thus $C$ equals the cone $D$ of $\F'$.

Conversely, suppose $D$ is a full-dimensional cone of $\F'$.
Then condition~\eqref{cone in cone R part} of Definition~\ref{R-part def} says that $D$ is contained in some full-dimensional cone $C$ of $\F$.
Since $\F'$ is a fan and $D$ is full dimensional, $D$ is the largest cone of $\F'$ contained in $C$, and as argued above, $D=C$.
Thus $D$ is a cone of $\F$.
\end{proof}

Now suppose a positive $R$-basis $(\b_i:i\in I)$ exists for $B$.
Let $\F^R_B$ be the collection of all cones spanned by sets $\set{\b_i:i\in I}\cap C$, where $C$ ranges over all $B$-cones, together with all faces of such cones.
In light of Proposition~\ref{positive unique}, the collection $\F^R_B$ does not depend on the choice of positive $R$-basis.

\begin{prop}\label{FRB}
If a positive $R$-basis $(\b_i:i\in I)$ exists for $B$, then $\F^R_B$ is a simplicial fan and is the $R$-part of $\F_B$.
\end{prop}

\begin{proof}
If some cone $C$ in $\F^R_B$ is not simplicial, then $C$ is generated by a set $U$ of vectors in $R^n$ that is linearly dependent over $\reals$.
We conclude that $U$ is also linearly dependent over $R$.
(Suppose $U$ is a set of vectors in $R^n$ that is linearly independent over $R$.
Write a matrix whose rows are $U$.
Use row operations over $R$ to put the matrix into echelon form.
Since there are no nontrivial $R$-linear relations on the rows, the echelon form has full rank, so $U$ is linearly independent over $\reals$.)
The relation expressing this linear dependence is $B$-local by the definition of $\F^R_B$, so by Proposition~\ref{local coherent} it is a $B$-coherent linear relation over $R$ among the basis vectors.

Now suppose $C_1$ and $C_2$ are maximal cones in $\F^R_B$, spanned respectively by subsets $U_1$ and $U_2$ of the positive basis.
Since $C_1$ and $C_2$ are spanned by vectors in $R^n$, if $R$ is a field, then $C_1$ and $C_2$ are each intersections of halfspaces with normal vectors in $R^n$.
Thus $C_1\cap C_2$ is an intersection of halfspaces with normal vectors in $R^n$, and thus is spanned by vectors in $R^n$.
In particular, there exists a vector $\x\in R^n$ contained in the relative interior of $C_1\cap C_2$.
The vector $\x$ can be expressed both as a nonnegative $R$-linear combination of $U_1$ and as a nonnegative $R$-linear combination of $U_2$.
Both of these expressions are $B$-local and thus $B$-coherent by Proposition~\ref{local coherent}, so their difference is a $B$-coherent $R$-linear relation.
But since $U_1$ and $U_2$ are part of an $R$-basis (and in particular an $R$-independent set) for $B$, the two expressions must coincide, so that each writes $\x$ as a nonnegative $R$-linear combination of $U_1\cap U_2$.
We conclude that $C_1\cap C_2$ is contained in the cone spanned by $U_1\cap U_2$, and the opposite containment is immediate.
Thus, since $C_1$ and $C_2$ are simplicial, $C_1\cap C_2$ is a face of both.
If instead $R=\integers$, then the usual arguments by clearing denominators show that $(\b_i:i\in I)$ is also a positive $\rationals$-basis for $B$, so that $\F^\rationals_B=\F^\integers_B$ and thus $\F^\integers_B$ is a simplicial fan.

We now verify the conditions of Definition~\ref{R-part def}.
Conditions~\eqref{R cones R part} and~\eqref{cone in cone R part} hold by construction.
The first part of condition~\eqref{max cone R part} holds by construction and the second part follows from the implication \eqref{positive cone basis pos}$\implies$\eqref{positive cone basis property} in Proposition~\ref{positive cone basis}.
\end{proof}

The following corollary is immediate by Propositions~\ref{positive cone basis} and~\ref{FRB}.

\begin{cor}\label{FB basis ray}
If a positive $\integers$-basis exists for $B$, then the unique positive $\integers$-basis for $B$ consists of the smallest nonzero integer vector in each ray of the $\integers$-part of $\F_B$.
If $R$ is a field and a positive $R$-basis exists for $B$, then a collection of vectors is a positive $R$-basis for $B$ if and only if it consists of exactly one nonzero vector in each ray of the $R$-part of $\F_B$.
\end{cor}

Propositions~\ref{R part prop} and~\ref{FRB} imply that $\F_B^\reals=\F_B$.
Thus we have the following corollary, which we emphasize is specific to the case $R=\reals$.

\begin{cor}\label{pos reals FB}
If a positive $\reals$-basis for $B$ exists, then $\F_B$ is simplicial.  
The basis consists of exactly one vector in each ray of $\F_B$.
\end{cor}

For any exchange matrix $B$, Proposition~\ref{local coherent} implies that a collection consisting of exactly one nonzero vector in each ray of $\F_B$ is an $\reals$-spanning set for $B$ and thus contains an $\reals$-basis for $B$ by Proposition~\ref{basis exists}.
In particular, the following proposition holds.
\begin{prop}\label{pos reals FB converse}
A collection consisting of exactly one nonzero vector in each ray of $\F_B$ is a positive $\reals$-basis for $B$ if and only if it is an $\reals$-independent set for~$B$.
\end{prop}
In this case, $\F_B$ is simplicial by Corollary~\ref{pos reals FB}.
On the other hand, if $\F_B$ is simplicial, then a collection consisting of exactly one nonzero vector in each ray of $\F_B$ could conceivably fail to be an $\reals$-independent set for $B$, in which case no positive $\reals$-basis exists for $B$.
However, we know of no exchange matrix $B$ for which this happens.

\begin{example}\label{nonsimp example cont'd}
In Sections~\ref{rk2 sec} and~\ref{Tits sec} and in~\cite{unisurface,unitorus}, we encounter examples of mutation fans that are simplicial.
There also appear to exist exchange matrices $B$ such that $\F_B$ is not simplicial, and one such example appears to be $B=\begin{bsmallmatrix*}[r]0&-3&0\\2&0&2\\0&-3&\hspace{6pt}0\\\end{bsmallmatrix*}$.
This ``appearance'' is based on computing approximations to $\F_B$ as explained in the paragraph after Proposition~\ref{contained Bcone}.
See Figure~\ref{nonsimp fig}.
It appears that as $m\to\infty$, the quadrilateral region at the middle of the picture near the top will decrease in size but neither disappear nor lose its quadrilateral shape.
(Compare with the example $B=\begin{bsmallmatrix*}[r]0&2\\-3&0\end{bsmallmatrix*}$, depicted in Figure~\ref{wild mutation fans}.)
If $\F_B$ is indeed not simplicial, then Corollary~\ref{pos reals FB} says that no positive $\reals$-basis exists for $B$.
\end{example}

\section{Properties of the mutation fan}\label{fan prop sec}
In this section, we prove some properties of mutation fans that are useful in constructing $\F_B$ is some cases.
We begin by pointing out several symmetries.

It is apparent from Definition~\ref{map} that 
\begin{equation}\label{eta antipodal}
\eta_\k^B(\a)=-\eta_\k^{-B}(-\a)\quad\mbox{for any sequence $\k$.}
\end{equation}
Since also the antipodal map $\a\mapsto-\a$ commutes with the map $\vsgn$, we see that $\a_1\equiv^B\a_2$ if and only if $(-\a_1)\equiv^{-B}(-\a_2)$.
Thus we have the following proposition, in which $-\F_{B}$ denotes the collection of cones $-C=\set{-\a:\a\in C}$ such that $C$ is a cone in $\F_{B}$.
\begin{prop}\label{antipodal FB}
$\F_{-B}=-\F_B$.
\end{prop}

Given any permutation $\pi$ of $[n]$, let $\pi(B)$ be the exchange matrix whose $ij$-entry is $b_{\pi(i)\pi(j)}$, where the entries of $B$ are $b_{ij}$.
Then $\mu_{\pi(k)}(\pi B)=\pi(\mu_k(B))$.
Let $\pi$ also denote the linear map sending $\e_{\pi(i)}$ to $\e_i$.
Then $\eta_{\pi(k)}^{\pi B}\circ\pi=\pi\circ\eta_k^B$.
Thus we have the following proposition, in which $\pi\F_B$ denotes the collection of cones $\pi C=\set{\pi\a:\a\in C}$ such that $C$ is a cone in $\F_B$.
\begin{prop}\label{pi FB}
$\F_{\pi B}=\pi\F_B$.
\end{prop}

Proposition~\ref{cones preserved} implies the following proposition, in which $\eta_\k^B\F_B$ denotes the collection of cones $\eta_\k^B C=\set{\eta_\k^B(\a):\a\in C}$ such that $C$ is a cone in $\F_B$.
\begin{prop}\label{eta FB}
$\F_{\mu_\k(B)}=\eta_\k^B\F_B$.
\end{prop}

A less obvious symmetry is a relationship called rescaling.

\begin{definition}[\emph{Rescaling of exchange matrices}]\label{rescale def}
Let $B$ and $B'$ be exchange matrices.
Then $B'$ is a \newword{rescaling} of $B$ if there exists a diagonal matrix $\Sigma$ with positive entries such that $B'=\Sigma^{-1} B\Sigma$.
In this case, if $\Sigma=\diag(\sigma_1,\ldots,\sigma_n)$, then the $ij$-entry of $B'$ is $\frac{\sigma_j}{\sigma_i}$ times the $ij$-entry of $B$.

Every matrix is a rescaling of itself, taking $\Sigma=cI$ for some positive $c$.
If $B'=\Sigma^{-1} B\Sigma$ for some $\Sigma$ not of the form $cI$, then $B'$ is a \newword{nontrivial} rescaling of~$B$.
A given exchange matrix may or may not admit any nontrivial rescalings.
\end{definition}

\begin{prop}\label{rescale iff}
If $B$ and $B'$ are exchange matrices, then $B'$ is a rescaling of $B$ if and only if $\sgn(b_{ij})=\sgn(b'_{ij})$ and $b_{ij}b_{ji}=b'_{ij}b'_{ji}$ for all $i,j\in[n]$.
\end{prop}
\begin{proof}
If $\Sigma$ is a diagonal matrix such that $B'=\Sigma^{-1} B\Sigma$, then $b'_{ij}b'_{ji}=\frac{\sigma_j}{\sigma_i}b_{ij}\frac{\sigma_i}{\sigma_j}b_{ji}=b_{ij}b_{ji}$.
Conversely, suppose $\sgn(b_{ij})=\sgn(b'_{ij})$ and $b_{ij}b_{ji}=b'_{ij}b'_{ji}$ for all $i,j\in[n]$.
Let $d_1,\ldots d_n$ be the skew-symmetrizing constants for $B$, so that $d_ib_{ij}=-d_jb_{ji}$.
Let $D$ be the diagonal matrix $\diag(\sqrt{d^{-1}},\ldots,\sqrt{d_n^{-1}})$, and define $\bB=D^{-1}BD$.
The $ij$-entry of $\bB$ is $\sqrt{\frac{d_i}{d_j}}b_{ij}$.
Since the $d_i$ are skew-symmetrizing constants, we calculate $\sqrt{\frac{d_i}{d_j}}b_{ij}=\sqrt{\frac{1}{d_id_j}}d_ib_{ij}=\sqrt{\frac{1}{d_id_j}}d_jb_{ji}=\sqrt{\frac{d_j}{d_i}}d_jb_{ji}$.
Since $\sqrt{\frac{d_i}{d_j}}b_{ij}\sqrt{\frac{d_j}{d_i}}b_{ji}=b_{ij}b_{ji}$, the $ij$-entry of $\bB$ is $\sgn(b_{ij})\sqrt{-b_{ij}b_{ji}}$.  
Define matrices $D'$ and $\widebar{B'}$ similarly in terms of the skew-symmetrizing constants for $B'$.
Arguing similarly, we see that the $ij$-entry of $\bB$ is $\sgn(b'_{ij})\sqrt{-b'_{ij}b'_{ji}}$.  
Thus $\widebar{B'}=\bB$, so $B'=D'D^{-1}BD(D')^{-1}$.
\end{proof}

Proposition~\ref{rescale iff} relies on the assumption that $B$ and $B'$ are exchange matrices in the sense of Definition~\ref{seed}.
That is, they are skew-symmetrizable integer matrices.

\begin{prop}\label{rescale int}
If $B$ and $B'$ are exchange matrices such that $B'$ is a rescaling of $B$, then the diagonal matrix $\Sigma$ with $B'=\Sigma^{-1} B\Sigma$ can be taken to have integer entries.
\end{prop}
\begin{proof}
Fixing $B$ and $B'$, the matrix $\Sigma=\diag(\sigma_1,\ldots,\sigma_n)$ satisfies $B'=\Sigma^{-1} B\Sigma$ if and only if $\sigma_ib'_{ij}=b_{ij}\sigma_j$ for all $i$ and $j$.
If these equations can be solved for the $\sigma_i$, then there is a rational solution and therefore an integer solution.
\end{proof}

\begin{remark}\label{root rescaling}
The definition of rescaling is motivated by root system considerations.
The Cartan companion $A$ of $B$ (see Definition~\ref{Cartan comp def}) defines a root system, and in particular simple roots and simple co-roots, as well as a symmetric bilinear form $K$.
The matrix $A$ expresses $K$ in terms of the simple co-root basis (on the left) and the simple root basis (on the right).
If $\Sigma=\diag(\sigma_1,\ldots,\sigma_n)$, then $\Sigma^{-1}A\Sigma$ is the matrix expressing $K$ in the basis $\set{\sigma_i^{-1}\alpha_i\ck:i\in[n]}$ (on the left) and $\set{\sigma_i\alpha_i:i\in[n]}$ (on the right).
In other words, $\Sigma^{-1}A\Sigma$ is a Cartan matrix with simple roots $\set{\sigma_i^{-1}\alpha_i\ck:i\in[n]}$ and simple co-roots $\set{\sigma_i\alpha_i:i\in[n]}$.
In this situation, every root in the root system for the Cartan matrix $\Sigma^{-1}A\Sigma$ is a scaling of a corresponding root for $A$, and vice versa.
\end{remark}

\begin{prop}\label{rescale props}
Suppose $B'$ is a rescaling of $B$, specifically with $B'=\Sigma^{-1} B\Sigma$.
\begin{enumerate}
\item \label{Sigma mu}
$\mu_\k(B')$ is a rescaling of $\mu_\k(B)$ for any sequence $\k$ of indices in $[n]$.
Specifically, $\mu_\k(B')=\Sigma^{-1}\mu_\k(B)\Sigma$.
\item \label{Sigma eta}
If $\a\in\reals^n$ is a row vector, then $\eta_\k^{B'}(\a\Sigma)=\eta_\k^B(\a)\Sigma$ for any sequence $\k$ of indices in $[n]$.
\item \label{Sigma FB}
The mutation fan $\F_{B'}$ is the collection of all cones $C\Sigma=\set{\a\Sigma:\a\in C}$, where $C$ ranges over cones in $\F_B$.
\item \label{Sigma B coherent}
The expression $\sum_{i\in S}c_i\v_i$ is a $B$-coherent linear relation if and only if $\sum_{i\in S}c_i(\v_i\Sigma)$ is a $B'$-coherent linear relation.
\item \label{Sigma indep}
If $(\b_i:i\in I)$ is a $\integers$-independent set for $B$ and $\Sigma$ is taken to have integer entries, then $(\b_i\Sigma:i\in I)$ is a $\integers$-independent set for $B'$.
\item \label{Sigma basis}
Suppose the underlying ring $R$ is a field.
Then $(\b_i:i\in I)$ is an $R$-independent set, $R$-spanning set, $R$-basis, cone $R$-basis and/or positive $R$-basis for $B$ if and only if $(\b_i\Sigma:i\in I)$ is the same for~$B'$.
\end{enumerate}
\end{prop}
\begin{proof}
For \eqref{Sigma mu} and \eqref{Sigma eta}, it is enough to replace the sequence $\k$ with a single index $k\in[n]$.
For \eqref{Sigma mu}, we use~\eqref{b mut} to verify that the $ij$-entry of $\mu_k(B')$ is $\frac{\sigma_j}{\sigma_i}$ times the $ij$-entry of $\mu_k(B)$.
For \eqref{Sigma eta}, we use~\eqref{mutation map def} to verify that the $j\th$ entry of $\eta_k^{B'}(\a\Sigma)$ is $\sigma_j$ times the $j\th$ entry of $\eta_k^B(\a)$.
To prove \eqref{Sigma FB}, we show that the map $\a\mapsto\a\Sigma$ takes $B$-classes to $B'$ classes.
Indeed, by \eqref{Sigma eta}, we have $\vsgn(\eta^{B'}_\k(\a_1\Sigma))=\vsgn(\eta^{B'}_\k(\a_2\Sigma))$ if and only if $\vsgn(\eta_\k^B(\a_1)\Sigma)=\vsgn(\eta_\k^B(\a_2)\Sigma)$, but since multiplication by $\Sigma$ preserves $\vsgn$, this is if and only if $\vsgn(\eta^B_\k(\a_1))=\vsgn(\eta^B_\k(\a_2))$.

For any $\k$, \eqref{Sigma eta} says that the sum $\sum_{i\in S}c_i\eta^{B'}_\k(\v_i\Sigma)$ equals $\left(\sum_{i\in S}c_i\eta^B_\k(\v_i)\right)\Sigma$ and that the sum $\sum_{i\in S}c_i\mathbf{min}(\eta^{B'}_\k(\v_i\Sigma),\mathbf{0})$ equals $\left(\sum_{i\in S}c_i\mathbf{min}(\eta^B_\k(\v_i),\mathbf{0})\right)\Sigma$.
We conclude that~\eqref{Sigma B coherent} holds.
Now~\eqref{Sigma indep} and \eqref{Sigma basis} follow.
\end{proof}

\begin{remark}\label{rescale integer basis}
If $R$ is not a field (that is, if $R=\integers$), then the conclusion of Proposition~\ref{rescale props}\eqref{Sigma basis} can fail.
For an example, we look ahead to Section~\ref{rk2 sec}.
Let $\tB$ be the third matrix in~\eqref{rk2 finite universal}, with exchange matrix $B=\begin{bsmallmatrix*}[r]0&1\\-2&0\end{bsmallmatrix*}$.
Then $B'=\begin{bsmallmatrix*}[r]0&2\\-1&0\end{bsmallmatrix*}$ is a rescaling of $B$ for $\Sigma$ any matrix of the form $\begin{bsmallmatrix*}[r]a&0\\0&2a\end{bsmallmatrix*}$ for positive integer $a$.
The coefficient rows $\b_i$ of $\tB$ are a positive $\integers$-basis for $B$, but there is no choice of $a$ such that the vectors $\b_i\Sigma$ are a $\integers$-basis for $B'$.
There is, however, a positive $\integers$-basis for $B'$ whose elements are positive rational multiples of the vectors $\b_i\Sigma$.
\end{remark}

The final proposition of this section concerns limits of $B$-cones.

\begin{definition}[\emph{Limits of rays and of cones}]\label{limit def}
Given a sequence $(\rho_m:m=1,2,\ldots)$ of rays in $\reals^n$, we say that the sequence converges if the sequence $(\v_m:m=1,2,\ldots)$ consisting of a unit vector $\v_m$ in each ray $\rho_m$ converges in the usual topology on $\reals^n$.
The limit of the sequence is the ray spanned by the limit of the vectors $\v_m$.
A sequence $(C_m:m=1,2,\ldots)$ of closed cones in $\reals^n$ \newword{converges} if, for some fixed $p$, each cone $C_m$ is the nonnegative linear span of some rays $\set{\rho_{m;i}:i=1,2,\ldots,p}$ and the sequence $(\rho_{m;i}:m=1,2,\ldots)$ converges for each $i\in[p]$.
The limit of the sequence is the nonnegative linear span of the limit rays.
\end{definition}

\begin{prop}\label{lim B cone}
Given a sequence of $B$-cones that converges in the sense of Definition~\ref{limit def}, the limit of the sequence is contained in a $B$-cone.
\end{prop}
\begin{proof}
Suppose $(C_m:m=1,2,\ldots)$ is a sequence of cones, each the nonnegative linear span of rays $\set{\rho_{m;i}:i=1,2,\ldots,p}$ such that $(\rho_{m;i}:m=1,2,\ldots)$ converges for each $i\in[p]$.
Write $C$ for the limiting cone.
For each $m\ge1$ and $i\in[p]$, let $\v_{m;i}$ be the unit vector in $\rho_{m;i}$, and write $\v_i$ for the limit of the sequence $(\v_{m;i}:m=1,2,\ldots)$ for each $i$.

If $\eta_\k^B(\v_i)$ and $\eta_\k^B(\v_j)$ have strictly opposite signs in some coordinate, for some sequence $\k$, then by the continuity of $\eta_\k^B$, for large enough $m$, the vectors $\eta_\k^B(\v_{m;i})$ and $\eta_\k^B(\v_{m;j})$ have strictly opposite signs in that coordinate.
This contradicts Proposition~\ref{contained Bcone} because $\v_{m;i}$ an $\v_{m;j}$ are both in $C_m$, so $\set{\eta_\k^B(\v_i):i=1,2,\ldots,p}$ is sign-coherent for each $\k$.
Proposition~\ref{contained Bcone} says that the set $\set{\v_i:i=1,2,\ldots,p}$ is contained in some $B$-cone.
Now Proposition~\ref{convex} implies that the set $C$ of nonnegative linear combinations of $\set{\v_i:i=1,2,\ldots,p}$ is contained in that $B$-cone.
\end{proof}

\section{\texorpdfstring{$\g$}{g}-Vectors and the mutation fan}\label{g vec sec}
In this section, we show that the mutation fan $\F_B$ contains an embedded copy of a fan defined by $\g$-vectors for $B^T$.
The result is conditional on the following conjecture, which shown in Proposition~\ref{H B} to be equivalent to a conjecture from~\cite{ca4}.
\begin{conj}\label{pos orth B}
The nonnegative orthant $(\reals_{\ge0})^n$ is a $B$-cone.
\end{conj}

At the time of writing, the conjecture is known for some exchange matrices $B$ and not for others (Remark~\ref{when sign-coherence known}), so we need to be precise about what we require.

\begin{definition}[\emph{Standard Hypotheses on $B$}]\label{standard hyp}  
The \newword{Standard Hypotheses} on $B$ are that Conjecture~\ref{pos orth B} holds for every exchange matrix in the mutation class of $B$ and for every exchange matrix in the mutation class of $-B$.
We use the symbol~$O$ for the nonnegative orthant $(\reals_{\ge0})^n$.
In light of Proposition~\ref{antipodal FB}, we restate the Standard Hypotheses as the requirement that both $O$ and the nonpositive orthant $-O$ are $B'$-cones for every exchange matrix $B'$ in the mutation class of $B$. 
\end{definition}

The Standard Hypotheses allow us to relate the mutation fan for $B$ to the $\g$-vectors associated to cluster variables in a cluster algebra with principal coefficients.

\begin{definition}[\emph{Principal coefficients}]\label{prin def}
Given an exchange matrix $B$, consider the $2n\times n$ extended exchange matrix $\tB$ whose top $n$ rows are $B$ and whose bottom $n$ rows are the $n\times n$ identity matrix.
A cluster pattern or $Y$-pattern with $\tB$ in its initial seed is said to have \newword{principal coefficients} at the initial seed.
We write $\x^{B;t_0}_t$ for the cluster indexed by $t$ in the cluster pattern with exchange matrix $B$ and principal coefficients at the vertex $t_0$ of $\T_n$.
The notation $x^{B;t_0}_{i;t}$ denotes the $i\th$ cluster variable in the cluster $\x^{B;t_0}_t$.
\end{definition}

The $\g$-vectors are most naturally defined as a $\integers^n$-grading on the cluster algebra.
(See \cite[Section~6]{ca4}.)
For convenience, however, we take as the definition a recursion (in fact two recursions) on $\g$-vectors established in \cite[Proposition~6.6]{ca4}.

\begin{definition}[\emph{$\g$-Vectors}]\label{gvec def}
We define a \newword{$\g$-vector} $\g^{B;t_0}_{i;t}$ for each cluster variable $x^{B;t_0}_{i;t}$.
The $\g$-vector associated to the cluster variable $x^{B;t_0}_{\ell,t_0}$ in the initial seed is the standard unit basis vector $\e_\ell\in\reals^n$.
The remaining $\g$-vectors are defined by the following recurrence relation, in which we have suppressed the superscripts ${B;t_0}$.
\begin{equation}
\label{g recur}
\g_{\ell,t'}=\begin{cases}
\g_{\ell,t}&\text{if $\ell\neq k$,}\\
-\g_{k,t}+\sum_{i=1}^n[b_{ik}^t]_+\g_{i;t}-\sum_{i=1}^n[b^t_{n+i,k}]_+\col(i,B)&\text{if $\ell=k$.} 
\end{cases}
\end{equation}
for each edge $t\dashname{k}t'$ in $\T_n$.
The entries $b^t$ refer to the $Y$-pattern of geometric type with exchange matrix $B$ at $t_0$ and principal coefficients.
The notation $\col(j,B)$ refers to the $j\th$ column of the initial exchange matrix $B$.

The $\g$-vectors are also defined by the recurrence relation
\begin{equation}
\label{g recur alt}
\g_{\ell,t'}=\begin{cases}
\g_{\ell,t}&\text{if $\ell\neq k$,}\\
-\g_{k,t}+\sum_{i=1}^n[-b_{ik}^t]_+\g_{i;t}-\sum_{i=1}^n[-b^t_{n+i,k}]_+\col(i,B)&\text{if $\ell=k$.} 
\end{cases}
\end{equation}
The equivalence of~\eqref{g recur} and~\eqref{g recur alt} is not obvious in this context, but the non-recursive definition given in \cite[Section~6]{ca4} validates the equivalence and the well-definition of the recursive definitions.
See the discussion in the proof of \cite[Proposition~6.6]{ca4}.
\end{definition}

\begin{definition}[\emph{$\g$-Vector cone}]\label{gvec cone def}
For each cluster $\x^{B;t_0}_t$, the associated \newword{$\g$-vector cone} is the nonnegative linear span of the vectors $\set{\g^{B;t_0}_{i;t}:i\in[n]}$.
We write $\Cone^{B;t_0}_t$ for the $\g$-vector cone associated to $\x^{B;t_0}_t$.
\end{definition}

\begin{definition}[\emph{Transitive adjacency and the subfan $\F^\circ_B$}]\label{adj def}
Two full-dimensional cones are \newword{adjacent} if they have disjoint interiors but share a face of codimension~$1$.
Two full-dimensional cones $C$ and $D$ in a fan $\F$ are \newword{transitively adjacent in $\F$} if there is a sequence $C=C_0,C_1,\ldots,C_k=D$ (possibly with $k=0$) of full-dimensional cones in $\F$ such that $C_{i-1}$ and $C_i$ are adjacent for each $i\in[k]$.
If the nonnegative orthant $O$ is a $B$-cone, then the full-dimensional cones in $\F_B$ that are transitively adjacent to $O$ in $\F_B$ are the maximal cones of a subfan $\F^\circ_B$ of $\F_B$.
\end{definition}

The following is the main result of this section.
\begin{theorem}\label{g subfan}
Assume the Standard Hypotheses on $B$.
Then the subfan $\F_B^\circ$ of $\F_B$ is the set of $\g$-vector cones $\set{\Cone^{B^T;t_0}_t:t\in\T_n}$, together with their faces.  
\end{theorem}

Before proving the theorem, we discuss the Standard Hypotheses further.
The following assertion appears in the proof of \cite[Proposition~5.6]{ca4} as condition (ii'), which is shown there to be equivalent to \cite[Conjecture~5.4]{ca4}.

\begin{conj}\label{H coherent}
For each extended exchange matrix in a $Y$-pattern with initial exchange matrix $B$ and principal coefficients, rows $n+1$ through $2n$ are a set of sign-coherent vectors.
\end{conj}

Many cases of Conjecture~\ref{H coherent} are known.  See Remark~\ref{when sign-coherence known}.

\begin{prop}\label{H B}
Conjecture~\ref{H coherent} holds for a given $B$ if and only if Conjecture~\ref{pos orth B} holds for $B$.
\end{prop}
\begin{proof}
Conjecture~\ref{H coherent} is equivalent to the following assertion: 
For each sequence~$\k$, the set $\set{\eta_{\k}(\e_i):i\in[n]}$ is sign-coherent.
Proposition~\ref{contained Bcone} says that the latter assertion is equivalent to the assertion that $\set{\e_i:i\in[n]}$ is contained in some $B$-cone.
Any set $O'$ strictly larger than $O$ contains a vector with a strictly negative entry and also a vector with all entries positive.
Thus Proposition~\ref{contained Bcone} (with $\k$ the empty sequence) implies that $O'$ is not a $B$-cone.
Since a $B$-cone containing $\set{\e_i:i\in[n]}$ is no larger than $O$, the set $\set{\e_i:i\in[n]}$ is contained in some $B$-cone if and only if $O$ is a $B$-cone.
\end{proof}

A crucial consequence of the Standard Hypotheses is another conjecture, \cite[Conjecture~7.12]{ca4}.
We quote the following weak version.

\begin{conj}\label{7.12}
Let $t_0\dashname{k} t_1$ be an edge in $\T_n$ and let $B_0$ and $B_1$ be exchange matrices such that $B_1=\mu_k(B_0)$.
Then, for any $t\in\T_n$ and $i\in[n]$, the $\g$-vectors $\g_{i;t}^{B_0;t_0}=(g_1,\ldots,g_n)$ and $\g_{i;t}^{B_1;t_1}=(g_1',\ldots,g_n')$ are related by
\begin{equation}\label{7.12 eq}
g'_j=\left\lbrace\!\!\!\begin{array}{ll}
-g_k&\mbox{if }j=k;\\[3 pt]
g_j+[b_{jk}]_+g_k-b_{jk}\min(g_k,0)&\mbox{if }j\neq k,
\end{array}\right.
\end{equation}
where the quantities $b_{jk}$ are the entries of the matrix $B_0$.
\end{conj}

More to the point is a restatement of Conjecture~\ref{7.12} in the language of mutation maps.
First, we rewrite \eqref{7.12 eq} as 
\begin{equation}\label{7.12 eq rewrite}
g'_j=\left\lbrace\begin{array}{ll}
-g_k&\mbox{if }j=k;\\
g_j+b_{jk}g_k&\mbox{if $j\neq k$, $g_k\ge 0$ and $b_{jk}\ge 0$};\\
g_j-b_{jk}g_k&\mbox{if $j\neq k$, $g_k\le 0$ and $b_{jk}\le 0$};\\
g_j&\mbox{otherwise.}
\end{array}\right.
\end{equation}
Then, comparing with~\eqref{mutation map def}, we see that the following conjecture is equivalent to Conjecture~\ref{7.12}.
(See \cite[Remark~7.15]{ca4} for a related restatement of \cite[Conjecture~7.12]{ca4}.)

\begin{conj}\label{gvec mu}
Let $t_0\dashname{k} t_1$ be an edge in $\T_n$ and let $B_0$ and $B_1$ be exchange matrices such that $B_1=\mu_k(B_0)$.
Then, for any $t\in\T_n$ and $i\in[n]$, the $\g$-vectors $\g_{i;t}^{B_0;t_0}$ and $\g_{i;t}^{B_1;t_1}$ are related by $\g_{i;t}^{B_1;t_1}=\eta^{B_0^T}_k(\g_{i;t}^{B_0;t_0})$, where $B_0^T$ is the transpose of~$B_0$.
\end{conj}

Nakanishi and Zelevinsky proved \cite[Proposition~4.2(v)]{NZ} that if Conjecture~\ref{H coherent} is true for all $B$, then \cite[Conjecture~7.12]{ca4} (the strong form of Conjecture~\ref{7.12}) is true for all $B$.
Their argument also proves the following statement, with weaker hypotheses and weaker conclusions, and with Conjecture~\ref{7.12} replaced by the equivalent Conjecture~\ref{gvec mu}.

\begin{theorem}\label{NZ 4.1(v)}
If the Standard Hypotheses hold for $B$, then Conjecture~\ref{gvec mu} holds for all exchange matrices $B_0$ and $B_1$ mutation equivalent to $B$.
\end{theorem}

The recursive definition of $\g$-vectors implies that if $t$ and $t'$ are connected by an edge in $\T_n$, then $\Cone^{B^T;t_0}_t$ and $\Cone^{B^T;t_0}_{t'}$ are adjacent.
Thus Theorem~\ref{g subfan} is an immediate corollary to the following proposition.

\begin{prop}\label{Cone B}
Assume the Standard Hypotheses on $B$.
If $t_0,t_1,\ldots,t_q$ is a path in $\T_n$, with the edge from $t_{i-1}$ to $t_i$ labeled $k_i$, then $\Cone^{B^T;t_0}_{t_q}$ is the full-dimensional $B$-cone $\eta_{k_1,\ldots,k_q}^{\mu_{k_q,\ldots,k_1}(B)}(O)$, where $O$ is the nonnegative orthant $(\reals_{\ge0})^n$.
\end{prop}
\begin{proof}
We argue by induction on $q$. 
In the course of the induction, we freely replace $B$ by elements of its mutation class. 
If $q=0$ then the proposition follows by the base of the inductive definition of $\g$-vectors and by Conjecture~\ref{pos orth B} for $B$.
If $q>0$, then take $B_0=B^T$ and $B_1=\mu_{k_1}(B^T)$.
Matrix mutation commutes with matrix transpose, so $B_1=(\mu_{k_1}(B))^T$.
By Theorem~\ref{NZ 4.1(v)}, Conjecture~\ref{gvec mu} holds for $B_0$ and $B_1$, so $\eta_{k_1}^B$ takes the extreme rays of $\Cone_{t_q}^{B^T;t_0}$ to the extreme rays of $\Cone_{t_q}^{\mu_{k_1}\!(B^T);t_1}$.
Recalling that $\eta_{k_1}^{\mu_{k_1}(B)}$ is the inverse of $\eta_{k_1}^B$, we see that $\eta_{k_1}^{\mu_{k_1}(B)}$ takes the extreme rays of $\Cone_{t_q}^{\mu_{k_1}\!(B^T);t_1}$ to the extreme rays of $\Cone_{t_q}^{B^T;t_0}$.
By induction (because the path from $t_1$ to $t_q$ has length $q-1$), the $\g$-vector cone $\Cone_{t_q}^{\mu_{k_1}\!(B^T);t_1}$ equals $\eta_{k_2,\ldots,k_q}^{\mu_{k_q,\ldots,k_2}(\mu_{k_1}(B))}(O)$, and this cone is a $\mu_{k_1}(B)$-cone.
Since $\Cone_{t_q}^{\mu_{k_1}\!(B^T);t_1}$ is a $\mu_{k_1}(B)$-cone, the map $\eta_{k_1}^{\mu_{k_1}(B)}$ is linear on it by Proposition~\ref{linear}, so this map takes the entire cone $\Cone_{t_q}^{\mu_{k_1}\!(B^T);t_1}$ to the cone $\Cone_{t_q}^{B^T;t_0}$ (rather than only mapping extreme rays to extreme rays).
Thus 
\[\Cone_{t_q}^{B^T;t_0}=\eta_{k_1}^{\mu_{k_1}(B)}\eta_{k_2,\ldots,k_q}^{\mu_{k_q,\ldots,k_2}(\mu_{k_1}(B))}(O).\]
Referring to \eqref{eta def}, we see that 
\[\eta_{k_2,\ldots,k_q}^{\mu_{k_q,\ldots,k_2}(\mu_{k_1}(B))}(O)=\eta_{k_2}^{\mu_{k_2k_1}(B))}\circ\eta_{k_3}^{\mu_{k_3k_2k_1}(B))}\circ\cdots\circ\eta_{k_q}^{\mu_{k_qk_{q-1}\cdots k_1}(B))}(O),\]
so that 
\[\Cone_{t_q}^{B^T;t_0}=\eta_{k_1}^{\mu_{k_1}(B)}\circ\eta_{k_2}^{\mu_{k_2k_1}(B))}\circ\cdots\circ\eta_{k_q}^{\mu_{k_qk_{q-1}\cdots k_1}(B))}(O)=\eta_{k_1,\ldots,k_q}^{\mu_{k_q,\ldots,k_1}(B)}(O).\]
Since $\eta_{k_2,\ldots,k_q}^{\mu_{k_q,\ldots,k_2}(\mu_{k_1}(B))}(O)$ is a $\mu_{k_1}(B)$-cone, Proposition~\ref{cones preserved} says that $\Cone_{t_q}^{B^T;t_0}$ is a $\mu_{k_1}(\mu_{k_1}(B))$-cone, or in other words, a $B$-cone.
\end{proof}

\smallskip

\begin{remark}~\label{when sign-coherence known}
Conjecture~\ref{H coherent} (and thus Conjecture~\ref{pos orth B}) is known in many cases, but currently not in full generality.
In particular, it was proved in \cite{QP2} for skew-symmetric exchange matrices.
(See also \cite{Nagao,Plamondon}.)
(In particular, since skew-symmetry is preserved under matrix mutation, Conjecture~\ref{gvec mu} is true whenever $B_0$ and $B_1$ are skew-symmetric.)
Conjecture~\ref{H coherent} is not specifically mentioned in \cite{QP2}, but \cite[Theorem~1.7]{QP2} establishes \cite[Conjecture~5.4]{ca4}, which is shown in the proof of \cite[Proposition~5.6]{ca4} to be equivalent to Conjecture~\ref{H coherent}.
In \cite{Demonet}, the construction of~\cite{QP2} is extended to some non-skew-symmetric exchange matrices.
In particular, \cite[Proposition~11.1]{Demonet} establishes Conjecture~\ref{H coherent} for a class of exchange matrices including all matrices of the form $DS$, where $D$ is an integer diagonal matrix and $S$ is an integer skew-symmetric matrix.
(This is a strictly smaller class than the class of all skew-symmetrizable matrices, which are the integer matrices of the form $D^{-1}S$ where $D$ is an integer diagonal matrix and $S$ is a skew-symmetric matrix.)
In a personal communication to the author \cite{Demonet personal}, Demonet indicates that his results can be extended to prove Conjecture~\ref{H coherent} for all exchange matrices that are mutation equivalent to an acyclic exchange matrix.  
See also \cite[Remark~7.2]{Demonet} and the beginning of \cite[Section~11]{Demonet}.
\end{remark}

As a corollary to Theorem~\ref{g subfan}, we see that any positive $R$-basis for $B$ must involve the $\g$-vectors for $B^T$.
\begin{cor}\label{pos basis g cor}
Assuming the Standard Hypotheses on $B$, if a positive $R$-basis exists for $B$, then the positive basis includes a positive scaling of the $\g$-vector associated to each cluster variable for $B^T$.
If $R=\integers$, then the positive basis includes the $\g$-vector associated to each cluster variable for $B^T$.
\end{cor}

To prove Corollary~\ref{pos basis g cor} (in particular the second statement), we consider the following conjecture, which is \cite[Conjecture~7.10(2)]{ca4}.
\begin{conj}\label{g Z basis}
For each $t\in\T_n$, the vectors $(\g_{i;t}^{B;t_0}:i\in[n])$ are a $\integers$-basis for~$\integers^n$.
\end{conj}
In \cite[Proposition~4.2(iv)]{NZ}, it is shown that if Conjecture~\ref{H coherent} is true for all $B$, then Conjecture~\ref{g Z basis} is true for all $B$.
Again, the argument from \cite{NZ} also proves a statement with weaker hypothesis and weaker conclusion:
\begin{theorem}\label{NZ 4.1(iv)}
If Conjecture~\ref{H coherent} is true for some $B$, then Conjecture~\ref{g Z basis} is also true for $B$.
\end{theorem}

\begin{proof}[Proof of Corollary~\ref{pos basis g cor}]
Corollary~\ref{FB basis ray} says that the positive basis consists of one nonzero vector in each ray in the $R$-part of $\F_B$.
Thus by Proposition~\ref{R part prop}, the positive basis contains one nonzero vector in each ray of $\F^\circ_B$.
Theorem~\ref{g subfan} implies that the basis contains a positive scalar multiple of the $\g$-vector of each cluster variable for $B^T$.
If $R=\integers$, then in particular each vector in the positive basis is the shortest integer vector in the ray it spans.
By Theorem~\ref{NZ 4.1(iv)}, the same is true of each $\g$-vector for $B^T$.
\end{proof}

We now present some results about rescalings of exchange matrices as they relate to the Standard Hypotheses and to $\g$-vectors.
The first result is immediate from Proposition~\ref{rescale props}\eqref{Sigma mu} and~\eqref{Sigma FB}.
\begin{prop}\label{standard rescale}
Suppose $B'$ is a rescaling of $B$.
The Standard Hypotheses hold for $B'$ if and only if they hold for $B$.
\end{prop}

\begin{prop}\label{Stand Hyp pmBT}
The Standard Hypotheses for $B$, for $-B$, for $B^T$, and for $-B^T$ are all equivalent. 
\end{prop}
\begin{proof}
The Standard Hypotheses for $B$ and $-B$ are syntactically equivalent.
Proposition~\ref{rescale iff} implies that $-B^T$ is a rescaling of $B$.
Proposition~\ref{standard rescale} completes the argument.
\end{proof}

\begin{prop}\label{g rescale}
Suppose $B'$ is a rescaling of $B$, specifically with $B'=\Sigma^{-1}B\Sigma$.
If the Standard Hypotheses hold for $B$, then $\g^{B';t_0}_{i;t}$ is the smallest nonzero integer vector in the ray spanned by $\g^{B;t_0}_{i,t}\Sigma$.
\end{prop}

\begin{proof}
Using Propositions~\ref{standard rescale} and~\ref{Stand Hyp pmBT} to obtain the Standard Hypotheses for all the relevant matrices, we appeal to Propositions~\ref{Cone B} and~\ref{rescale props}\eqref{Sigma FB} to conclude that $\g^{B';t_0}_{i;t}$ and $\g^{B;t_0}_{i,t}\Sigma$ span the same ray.
Theorem~\ref{NZ 4.1(iv)} implies that $\g^{B';t_0}_{i;t}$ is the smallest nonzero integer vector in that ray.
\end{proof}

We next discuss two conjectures suggested by Corollary~\ref{pos basis g cor}.
The first is a separation property for cluster variables.
\begin{conj}\label{g compat conj}
Suppose $x^{B;t_0}_{i;t'}$ and $x^{B;t_0}_{j;t''}$ are cluster variables in the cluster algebra with principal coefficients.
Then the following are equivalent.
\begin{enumerate}[(i)]
\item \label{not contained}
 $x^{B;t_0}_{i;t'}$ and $x^{B;t_0}_{j;t''}$ are \emph{not} contained in any common cluster.
\item \label{strictly opp}
There exists $t\in \T_n$ and $k\in[n]$ such that $\g^{B_t;t}_{i;t'}$ and $\g^{B_t;t}_{j,t''}$ have strictly opposite signs in their $k\th$ entry.
\end{enumerate}
\end{conj}

\begin{prop}\label{g compat one direction}
Assuming the Standard Hypotheses on $B$, condition \eqref{strictly opp} of Conjecture~\ref{g compat conj} implies condition \eqref{not contained}.
\end{prop}
\begin{proof}
By Proposition~\ref{contained Bcone} and Theorem~\ref{NZ 4.1(v)}, condition \eqref{strictly opp} implies that $\g^{B;t_0}_{i;t'}$ and $\g^{B;t_0}_{j;t''}$ are not contained in any common $B^T$-cone.
By Proposition~\ref{Stand Hyp pmBT}, the Standard Hypotheses hold for $B^T$, so all of the $\g$-vector cones for $B$ are $B^T$-cones by Theorem~\ref{g subfan}.
Thus $\g^{B_t;t}_{i;t'}$ and $\g^{B_t;t}_{j,t''}$ are not contained in any common $\g$-vector cone, and condition \eqref{not contained} follows.
\end{proof}
We cannot reverse the argument for Proposition~\ref{g compat one direction} because conceivably the two $\g$-vectors are contained in some $B$-cone that is not in $\F^\circ_B$.

The second conjecture suggested by Theorem~\ref{g subfan} is an independence property of $\g$-vectors.
\begin{conj}\label{g B coher conj}
Suppose $t_1,\ldots,t_m$ are vertices of $\T_n$, suppose $i_i,\ldots,i_m$ are indices in $[n]$, and suppose $c_1,\ldots,c_m$ are integers.
If $\sum_{j=1}^mc_j\g^{B_t;t}_{i_j;t_j}=\mathbf{0}$ for all vertices $t$ of $\T_n$, then $c_j=0$ for all $j=1,\ldots,m$.
\end{conj}

\begin{prop}\label{g B coher pos basis}
Assuming the Standard Hypotheses on $B$, if there exists an underlying ring $R$ such that a positive $R$-basis exists for $B^T$, then Conjecture~\ref{g B coher conj} holds for $B$.
\end{prop}
\begin{proof}
Given the Standard Hypotheses on $B$, Proposition~\ref{Stand Hyp pmBT} and Theorem~\ref{NZ 4.1(v)} allow us to restate Conjecture~\ref{g B coher conj}:
There is no nontrivial $B^T$-coherent linear relation over $\integers$ among the $\g$-vectors for cluster variables associated to $B$.

Suppose there exists an underlying ring $R$ such that a positive $R$-basis exists for~$B^T$.
In light of Proposition~\ref{Stand Hyp pmBT}, Corollary~\ref{pos basis g cor} says that the positive basis contains a positive scalar multiple of the $\g$-vector of each cluster variable for $B$.
In particular, a $B^T$-coherent linear relation over $\integers$ among these $\g$-vectors can be rewritten as a $B^T$-coherent linear relation among basis elements.
Thus any $B^T$-coherent linear relation over $\integers$ among these $\g$-vectors is trivial.
\end{proof}

\section{The rank-2 case}\label{rk2 sec}
In this section, we construct $R$-bases for exchange matrices of rank $2$.
We prove the following theorem and give explicit descriptions of the additional vectors mentioned.
The details are given in Propositions~\ref{rk2 finite basis}, \ref{rk2 affine basis} and~\ref{rk2 wild basis}.
\begin{theorem}\label{rk 2 thm}
Let $B$ be an exchange matrix of rank $2$.
Then an $R$-basis for $B$ can be constructed by taking the $\g$-vectors associated to $B^T$ and possibly adding one or two additional vectors.
\end{theorem}

The exchange matrices of rank $2$ are the matrices $B$ of the form $\begin{bsmallmatrix}0&a\\b&0\end{bsmallmatrix}$ where $a$ and $b$ are integers with $\sgn(b)=-\sgn(a)$.
(This is skew-symmetrizable with $d_1=|b|$ and $d_2=|a|$.)  
We label the vertices of the infinite $2$-regular tree $\T_2$ as $\ldots,t_{-1},t_0,t_1,\ldots$ with $t_i$ adjacent to $t_{i-1}$ by an edge labeled $k\in\set{1,2}$ with $k\equiv i\mod 2$.
The Standard Hypotheses (Definition~\ref{standard hyp}) are known to hold for exchange matrices of rank $2$.
(For example, one may apply \cite[Proposition~11.1]{Demonet} as explained in Remark~\ref{when sign-coherence known}.)
Since $B$ is of rank $2$, the fan $\F_B$ is a fan in $\reals^2$, and since no $B$-cone contains a line, $\F_B$ is therefore simplicial.

The following proposition accomplishes most of the work towards proving Theorem~\ref{rk 2 thm}.
We again write $\R(B)$ for the set of rays of the mutation fan $\F_B$.
We write $\R^\circ(B)$ for the subset of $\R(B)$ consisting of rays of $\F_B^\circ$.
(For the definition of $\F_B^\circ$, see the paragraph before Theorem~\ref{g subfan}.)

\begin{prop}\label{rk2 gvec b-coherent}
Suppose $B$ is of rank $2$ and, for each ray $\rho\in\R(B)$, choose a nonzero vector $\v_\rho$ contained in $\rho$.
Then any $B$-coherent linear relation supported on the set $\set{\v_\rho:\rho\in\R(B)}$ is in fact supported on $\set{\v_\rho:\rho\in\R(B)\setminus\R^\circ(B)}$.
\end{prop}
\begin{proof}
Suppose $\rho$ is in $\R^\circ(B)$, so that $\rho$ is an extreme ray of a $\g$-vector cone for $B^T$ by Theorem~\ref{g subfan}.
In fact, $\rho$ is an extreme ray of exactly two $\g$-vector cones for $B^T$, and these cones are $\Cone_{t_q}^{B^T;t_0}$ and $\Cone_{t_{q-1}}^{B^T;t_0}$ for some $q\in\integers$.
We argue the case $q\ge0$; the other case is the same except for notation.
Proposition~\ref{Cone B} says that $\Cone^{B^T;t_0}_{t_{q-1}}=\eta_{k_1,\ldots,k_{q-1}}^{\mu_{k_{q-1},\ldots,k_1}(B)}(O)$ and that $\Cone^{B^T;t_0}_{t_q}=\eta_{k_1,\ldots,k_q}^{\mu_{k_q,\ldots,k_1}(B)}(O)$.
Inverting the maps $\eta$ in these two expressions, we see that the nonnegative orthant $O$ equals $\eta^B_{k_{q-1},\ldots,k_1}\left(\Cone^{B^T;t_0}_{t_{q-1}}\right)$ and also equals 
\[\eta^B_{k_q,\ldots,k_1}\left(\Cone^{B^T;t_0}_{t_q}\right)=\eta_{k_q}^{\mu_{k_{q-1},\ldots,k_1}(B)}\eta^B_{k_{q-1},\ldots,k_1}\left(\Cone^{B^T;t_0}_{t_q}\right).\]
By Proposition~\ref{eta FB}, the cones $\eta^B_{k_{q-1},\ldots,k_1}\left(\Cone^{B^T;t_0}_{t_{q-1}}\right)$ and $\eta^B_{k_{q-1},\ldots,k_1}\left(\Cone^{B^T;t_0}_{t_q}\right)$ intersect in a ray of $\F_{\mu_{k_{q-1},\ldots,k_1}(B)}$, namely the ray $\eta^B_{k_{q-1},\ldots,k_1}(\rho)=\eta^B_{k_q,\ldots,k_1}(\rho)$.
But the map $\eta_{k_q}^{\mu_{k_{q-1},\ldots,k_1}(B)}$ fixes one extreme ray of $O$ and sends the other outside of~$O$.
We conclude that $\eta^B_{k_{q-1},\ldots,k_1}(\rho)$ and $\eta^B_{k_q,\ldots,k_1}(\rho)$ both equal $\reals_{\ge0}\e_j$, where $j$ is the unique element of $\set{1,2}\setminus\set{k_q}$.
In rank $2$, every single mutation operation $\mu_k$ is simply negation of the matrix.
Thus, either $\mu_{k_{q-1},\ldots,k_1}(B)$ or  $\mu_{k_q,\ldots,k_1}(B)$ has a nonnegative $ij$-entry.
If $\mu_{k_{q-1},\ldots,k_1}(B)$ has a nonnegative $ij$-entry, then write $\k=k_{q-1},\ldots,k_1$ and $t=t_{q-1}$ and $t'=t_q$.
Otherwise, write $\k=k_q,\ldots,k_1$ and $t=t_q$ and $t'=t_{q-1}$.
Write $i=k_q$.
The edge connecting $t$ to $t'$ is labeled $t\dashname{i}t'$.

The $\g$-vectors $\g^{\mu_\k(B^T);t}_{i;t}=\e_i$ and $\g^{\mu_\k(B^T);t}_{j;t}=\e_j$ span the cone $O=\Cone_t^{\mu_\k(B^T);t}$.
The matrix $\mu_\k(B^T)$ has a nonnegative $ji$-entry, so the recursion \eqref{g recur alt} says that $\g^{\mu_\k(B^T);t}_{i;t'}=-\e_i$ and $\g^{\mu_\k(B^T);t}_{j;t'}=\e_j$.
Thus the cones $\Cone_t^{\mu_\k(B^T);t}$ and $\Cone_{t'}^{\mu_\k(B^T);t}$ cover the closed halfspace consisting of vectors with non-negative $j\th$ entry.
These are cones in the fan $\F_{\mu_\k(B)}$ by Theorem~\ref{g subfan}, so their common ray $\eta^B_\k(\rho)$ is the unique ray of $\F_{\mu_\k(B)}$ that intersects the open halfspace consisting of vectors with positive $j\th$ entry.

Now suppose $\sum_{i\in S}c_i\v_i$ is a $B$-coherent linear relation with $S$ a finite subset of $\R(B)$.
By Proposition~\ref{eta FB}, for each $\rho\in S$, the vector $\eta_\k(\v_{\rho})$ spans a ray of $\F_{\mu_\k(B)}$.
By the argument above, if $\rho\in S\cap\R^\circ(B)$, then the $j\th$ entry of $\eta_\k^B(\v_\rho)$ is strictly positive and every vector $\v_{\rho'}$ with $\rho'\neq\rho$ has non-positive $j\th$ entry.
We conclude from Proposition~\ref{one positive or one negative} that $c_\rho=0$.
\end{proof}

To complete the proof of Theorem~\ref{rk 2 thm}, we explicitly construct the mutation fan for $B$ and to prove that the remaining vectors $\set{\v_\rho:\rho\in\R(B)}$ beyond the rays in $\g$-vector cones must also appear with coefficient zero in any $B$-coherent linear relation.
We consider several cases separately.

\begin{definition}[\emph{Cluster algebra/exchange matrix of finite type}]\label{fin def}
A cluster algebra is of \newword{finite type} if and only if its associate cluster pattern contains only finitely many distinct seeds.
Otherwise it is of \newword{infinite type}.
\end{definition}

The classification \cite{ga,ca2} of cluster algebras of finite type says in particular that $B=\begin{bsmallmatrix}0&a\\b&0\end{bsmallmatrix}$ is of finite type if and only if $ab\ge-3$.
The possibilities for $(a,b)$ are $(0,0)$, $(\pm 1,\mp 1)$, $(\pm 1,\mp 2)$, $(\pm 2,\mp 1)$, $(\pm 1,\mp 3)$, and $(\pm 3,\mp 1)$.

\begin{prop}\label{rk2 finite basis}
If $B$ is a rank-$2$ exchange matrix of finite type, then for any $R$, the $\g$-vectors for $B^T$ constitute a positive $R$-basis for $B$.
\end{prop}
\begin{proof}
In the case where $B$ is of finite type, it is known that the $\g$-vector cones are the maximal cones of a complete fan.
For example, in Section~\ref{Tits sec}, we explain how this fact (Theorem~\ref{g complete}), for arbitrary rank and finite type, follows from results of \cite{camb_fan,framework,YZ}.  
(See Remark~\ref{me and ca4}.)
This can also be verified directly in all rank-$2$ finite-type cases.
In particular, $\F_B$ consists of the $\g$-vector cones associated to $B^T$, and their faces.
The collection of all $\g$-vectors is a positive $R$-basis for $B$ by Theorem~\ref{NZ 4.1(iv)} and Proposition~\ref{rk2 gvec b-coherent}.
\end{proof}

One can calculate these bases explicitly in all cases.
Some of the resulting universal extended exchange matrices are shown below in~\eqref{rk2 finite universal}.
\begin{equation}
\label{rk2 finite universal}
\begin{bsmallmatrix*}[r]
0&0\\
0&0\\
1&0\\
0&1\\
-1&0\\
0&-1\\
\end{bsmallmatrix*}
\qquad
\begin{bsmallmatrix*}[r]
0&1\\
-1&0\\
1&0\\
0&1\\
-1&0\\
0&-1\\
1&-1\\
\end{bsmallmatrix*}
\qquad
\begin{bsmallmatrix*}[r]
0&1\\
-2&0\\
1&0\\
0&1\\
-1&0\\
0&-1\\
1&-1\\
2&-1\\
\end{bsmallmatrix*}
\qquad
\begin{bsmallmatrix*}[r]
0&1\\
-3&0\\
1&0\\
0&1\\
-1&0\\
0&-1\\
1&-1\\
3&-2\\
2&-1\\
3&-1\\
\end{bsmallmatrix*}
\end{equation}
In light of Propositions~\ref{antipodal FB} and~\ref{pi FB}, the remaining cases are obtained from the cases shown by applying one or both of the following operations: (a) negating all entries and/or (b) swapping the columns and then swapping the first two rows.
The mutation fans corresponding to the four cases in~\eqref{rk2 finite universal} are shown in Figure~\ref{finite mutation fans}.
The standard unit basis vector $\e_1$ points to the right in the pictures and $\e_2$ points upwards.
Black lines indicate the rays.
\begin{figure}[ht]
\begin{tabular}{ccccccc}
\includegraphics{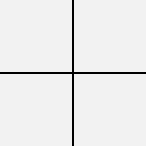}&&
\includegraphics{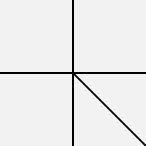}&&
\includegraphics{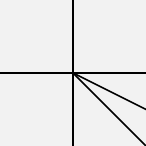}&&
\includegraphics{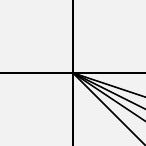}
\\[2 pt]
$B=\begin{bsmallmatrix*}[r]0&0\\0&0\end{bsmallmatrix*}$&&
$B=\begin{bsmallmatrix*}[r]0&1\\-1&0\end{bsmallmatrix*}$&&
$B=\begin{bsmallmatrix*}[r]0&1\\-2&0\end{bsmallmatrix*}$&&
$B=\begin{bsmallmatrix*}[r]0&1\\-3&0\end{bsmallmatrix*}$
\end{tabular}
\caption{Mutation fans $\F_B$ for the finite examples in~\eqref{rk2 finite universal}}
\label{finite mutation fans}
\end{figure}

\begin{example}\label{B2 in detail}
We now consider a detailed example based on the universal extended exchange matrix $\tB$ in~\eqref{rk2 finite universal} whose underlying exchange matrix is $B=\begin{bsmallmatrix*}[r]0&1\\-2&0\end{bsmallmatrix*}$.
We write the indexing set $I$ as $\set{a,b,c,d,e,f}$, so that the matrix is indexed as 
\begin{equation}
\begin{array}{c}
\mbox{\hspace{14 pt}\tiny 1 \hspace{4 pt} 2}\\[1 pt]
\begin{smallmatrix}1\\2\\a\\b\\c\\d\\e\\f\end{smallmatrix}\begin{bsmallmatrix*}[r]
0&1\\
-2&0\\
1&0\\
0&1\\
-1&0\\
0&-1\\
1&-1\\
2&-1\\
\end{bsmallmatrix*}
\end{array}
\end{equation}
We continue to label the vertices of $\T_2$ as $\ldots,t_{-1},t_0,t_1,\ldots$ with $t_i$ adjacent to $t_{i-1}$ by an edge labeled $k\in\set{1,2}$ such that $k\equiv i\mod 2$.
In this case, the labeled seed at $t_i$ depends only on $i\mod 6$.
The extended exchange matrices, coefficients, and cluster variables in this cluster algebra are shown in Tables~\ref{univ tB table} and~\ref{univ x table}.
\begin{table}[p]
\begin{tabular}{|c||c|c|c|c|c|c|}
\hline
$t$&$t_0$&$t_1$&$t_2$&$t_3$&$t_4$&$t_5$\\
\hline
&&&&&&\\[-10pt]
\hline
&&&&&&\\[-10pt]
$\tB_t$&
$\begin{bsmallmatrix*}[r]
0&1\\
-2&0\\
1&0\\
0&1\\
-1&0\\
0&-1\\
1&-1\\
2&-1\\
\end{bsmallmatrix*}$
&
$\begin{bsmallmatrix*}[r]
0&-1\\
2&0\\
-1&1\\
0&1\\
1&0\\
0&-1\\
-1&0\\
-2&1\\
\end{bsmallmatrix*}$
&
$\begin{bsmallmatrix*}[r]
0&1\\
-2&0\\
1&-1\\
2&-1\\
1&0\\
0&1\\
-1&0\\
0&-1\\
\end{bsmallmatrix*}$
&
$\begin{bsmallmatrix*}[r]
0&-1\\
2&0\\
-1&0\\
-2&1\\
-1&1\\
0&1\\
1&0\\
0&-1\\
\end{bsmallmatrix*}$
&
$\begin{bsmallmatrix*}[r]
0&1\\
-2&0\\
-1&0\\
0&-1\\
1&-1\\
2&-1\\
1&0\\
0&1\\
\end{bsmallmatrix*}$
&
$\begin{bsmallmatrix*}[r]
0&-1\\
2&0\\
1&0\\
0&-1\\
-1&0\\
-2&1\\
-1&1\\
0&1\\
\end{bsmallmatrix*}$\\[22pt]
\hline
&&&&&&\\[-10pt]
$y_{1,t}$&$\frac{u_au_eu_f^2}{u_c}$&$\frac{u_c}{u_1u_eu_f^2}$&$\frac{u_au_b^2u_c}{u_e}$&$\frac{u_e}{u_au_b^2u_c}$&$\frac{u_cu_d^2u_e}{u_a}$&$\frac{u_a}{u_cu_d^2u_e}$ \\[5 pt]\hline
&&&&&&\\[-10pt]
$y_{2,t}$&$\frac{u_b}{u_du_eu_f}$&$\frac{u_au_bu_f}{u_d}$&$\frac{u_d}{u_au_bu_f}$&$\frac{u_bu_cu_d}{u_f}$&$\frac{u_f}{u_bu_cu_d}$&$\frac{u_du_eu_f}{u_b}$ \\[5 pt]\hline
\end{tabular}\\
\captionsetup{width=\textwidth}
\caption{Extended exchange matrices and coefficients for Example~\ref{B2 in detail} (universal)}
\label{univ tB table}
\vspace{-1.55pt}
\begin{tabular}{|c||c|c|}
\hline
$t$&$x_{1;t}$&$x_{2;t}$\\ \hline
&&\\[-10pt]
\hline
$t_0$&$x_1$&$x_2$	\\ \hline
&&\\[-10pt]
$t_1$&$\frac{x_2^2u_c+u_au_eu_f^2}{x_1}$&$x_2$	\\[4pt] \hline
&&\\[-10pt]
$t_2$&$\frac{x_2^2u_c+u_au_eu_f^2}{x_1}$&$\frac{x_1u_au_bu_f+x_2^2u_cu_d+u_au_du_eu_f^2}{x_1x_2}$	\\[4pt] \hline
&&\\[-10pt]
$t_3$&$\frac{x_1^2u_au_b^2+2x_1u_au_bu_du_eu_f+x_2^2u_cu_d^2u_e+u_au_d^2u_e^2u_f^2}{x_1x_2^2}$&$\frac{x_1u_au_bu_f+x_2^2u_cu_d+u_au_du_eu_f^2}{x_1x_2}$	\\[4pt] \hline
&&\\[-10pt]
$t_4$&$\frac{x_1^2u_au_b^2+2x_1u_au_bu_du_eu_f+x_2^2u_cu_d^2u_e+u_au_d^2u_e^2u_f^2}{x_1x_2^2}$&$\frac{x_1u_b+u_du_eu_f}{x_2}$\\[4pt] \hline
&&\\[-10pt]
$t_5$&$x_1$&$\frac{x_1u_b+u_du_eu_f}{x_2}$	\\[4pt] \hline
\end{tabular}\\
\caption{Cluster variables for Example~\ref{B2 in detail} (universal)}
\label{univ x table}
\vspace{-1.55pt}
\begin{tabular}{|c||c|c|c|c|c|c|}
\hline
$t$&$t_0$&$t_1$&$t_2$&$t_3$&$t_4$&$t_5$\\
\hline
&&&&&&\\[-10pt]
\hline
&&&&&&\\[-10pt]
$\tB'_t$&
$\begin{bsmallmatrix*}[r]
0&1\\
-2&0\\
3&-2\\
1&2\\
-1&1\\
\end{bsmallmatrix*}$
&
$\begin{bsmallmatrix*}[r]
0&-1\\
2&0\\
-3&1\\
-1&3\\
1&1\\
\end{bsmallmatrix*}$
&
$\begin{bsmallmatrix*}[r]
0&1\\
-2&0\\
-1&-1\\
5&-3\\
3&-1\\
\end{bsmallmatrix*}$
&
$\begin{bsmallmatrix*}[r]
0&-1\\
2&0\\
1&-1\\
-5&2\\
-3&2\\
\end{bsmallmatrix*}$
&
$\begin{bsmallmatrix*}[r]
0&1\\
-2&0\\
1&1\\
-1&-2\\
1&-2\\
\end{bsmallmatrix*}$
&
$\begin{bsmallmatrix*}[r]
0&-1\\
2&0\\
-1&2\\
1&-2\\
-1&-1\\
\end{bsmallmatrix*}$\\[14pt]
\hline
&&&&&&\\[-10pt]
$y_{1,t}$&$\frac{u_\alpha^3u_\beta}{u_\gamma}$&$\frac{u_\gamma}{u_\alpha^3u_\beta}$&$\frac{u_\beta^5u_\gamma^3}{u_\alpha}$&$\frac{u_\alpha}{u_\beta^5u_\gamma^3}$&$\frac{u_\alpha u_\gamma}{u_\beta}$&$\frac{u_\beta}{u_\alpha u_\gamma}$ \\[5 pt]\hline
&&&&&&\\[-10pt]
$y_{2,t}$&$\frac{u_\beta^2u_\gamma}{u_\alpha^2}$&\small$u_\alpha u_\beta^3u_\gamma$&$\frac{1}{u_\alpha u_\beta^3u_\gamma}$&$\frac{u_\beta^2u_\gamma^2}{u_\alpha}$&$\frac{u_\alpha}{u_\beta^2u_\gamma^2}$&$\frac{u_\alpha^2}{u_\beta^2u_\gamma}$ \\[5 pt]\hline
\end{tabular}\\
\caption{Extended exchange matrices and coefficients for Example~\ref{B2 in detail}}
\label{other tB table}
\vspace{-1.55pt}
\begin{tabular}{|c||c|c|}
\hline
$t$&$x_{1;t}$&$x_{2;t}$\\ \hline
&&\\[-10pt]
\hline
$t_0$&$x'_1$&$x'_2$	\\ \hline
&&\\[-10pt]
$t_1$&$\frac{(x'_2)^2u_\gamma+u_\alpha^3u_\beta}{x'_1}$&$x'_2$	\\[4pt] \hline
&&\\[-10pt]
$t_2$&$\frac{(x'_2)^2u_\gamma+u_\alpha^3u_\beta}{x'_1}$&$\frac{x'_1u_\alpha u_\beta^3u_\gamma+(x_2')^2u_\gamma+u_\alpha^3u_\beta}{x'_1x'_2}$	\\[4pt] \hline
&&\\[-10pt]
$t_3$&$\frac{(x'_1)^2u_\beta^5u_\gamma^2+2x'_1u_\alpha^2u_\beta^3u_\gamma+(x'_2)^2u_\alpha u_\gamma+u_\alpha^4 u_\beta}{x'_1(x'_2)^2}$&$\frac{x'_1u_\alpha u_\beta^3u_\gamma+(x_2')^2u_\gamma+u_\alpha^3u_\beta}{x'_1x'_2}$	\\[4pt] \hline
&&\\[-10pt]
$t_4$&$\frac{(x'_1)^2u_\beta^5u_\gamma^2+2x'_1u_\alpha^2u_\beta^3u_\gamma+(x'_2)^2u_\alpha u_\gamma+u_\alpha^4 u_\beta}{x'_1(x'_2)^2}$&$\frac{x'_1u^2_\beta u_\gamma+u_\alpha^2}{x'_2}$\\[4pt] \hline
&&\\[-10pt]
$t_5$&$x'_1$&$\frac{x'_1u^2_\beta u_\gamma+u_\alpha^2}{x'_2}$	\\[4pt] \hline
\end{tabular}\\
\caption{Cluster variables for Example~\ref{B2 in detail}}
\label{other x table}
\end{table}

Now consider another extended exchange matrix $\tB'$ with the same underlying exchange matrix $B=\begin{bsmallmatrix*}[r]0&1\\-2&0\end{bsmallmatrix*}$.
This time, the indexing set $I'$ is $\set{\alpha,\beta,\gamma}$, so that $\tB'$ is indexed as 
\begin{equation}
\begin{array}{c}
\mbox{\hspace{14 pt}\tiny 1 \hspace{4 pt} 2}\\[1 pt]
\begin{smallmatrix}1\\2\\\alpha\\\beta\\\gamma\end{smallmatrix}\begin{bsmallmatrix*}[r]
0&1\\
-2&0\\
3&-2\\
1&2\\
-1&1\\
\end{bsmallmatrix*}
\end{array}
\end{equation}
The extended exchange matrices, coefficients, and cluster variables in this cluster algebra are shown in Tables~\ref{other tB table} and~\ref{other x table}.

The coefficient rows of $\tB$, indexed by $a,b,c,d,e,f$, are vectors contained in the rays of $\F_B$.
We name these vectors $\v_a$, $\v_b$, etc.
Let $\v_\alpha$, $\v_\beta$, and $\v_\gamma$ similarly name the coefficient rows of $\tB'$.
We find the unique coefficient specialization from $\A_R(\x,\tB)$ to $\A_R(\x',\tB')$ as described in Remarks~\ref{explicit spec cone basis} and~\ref{cone and pos}.
First, $\v_\alpha$ is in the $B$-cone spanned by $\v_e$ and $\v_f$, with $\v_\alpha=\v_e+\v_f$.
Similarly, $\v_\beta$ is in the $B$-cone spanned by $\v_a$ and $\v_b$, with $\v_\beta=\v_a+2\v_b$, and $\v_\gamma$ is in the $B$-cone spanned by $\v_b$ and $\v_c$, with $\v_\gamma=\v_b+\v_c$.
Accordingly, we obtain a coefficient specialization mapping
\begin{equation}\label{coef spec example}
\begin{array}{llll}
x_1\mapsto x'_1&\qquad u_a\mapsto u_\beta&\qquad u_c\mapsto u_\gamma&\qquad u_e\mapsto u_\alpha\\[2pt]
x_2\mapsto x'_2&\qquad u_b\mapsto u_\beta^2 u_\gamma&\qquad u_d\mapsto 1&\qquad u_f\mapsto u_\alpha.
\end{array}
\end{equation}
\end{example}

We have dealt with the rank-$2$ exchange matrices of finite type.
It remains to consider the cases with $ab\le-4$.
For $m=0,1,2,\ldots$, define
\begin{equation}\label{Pm def}
P_m=(-1)^{\lfloor m/2\rfloor}\sum_{i\ge 0}\binom{m-i}{i}(ab)^{\lfloor m/2\rfloor-i}.
\end{equation}
This is a polynomial in $-ab$.
The first several polynomials $P_m$ are shown in Table~\ref{pm table}.
\begin{table}[ht]
\begin{tabular}{|l|r|}
\hline
$m$&$P_m$\\\hline
$0$&	$1$	\\\hline
$1$&$1$	\\\hline
$2$&	$-ab-1$\\\hline
$3$&	$-ab-2$\\\hline
$4$&	$a^2b^2+3ab+1$	\\\hline
$5$&	$a^2b^2+4ab+3$	\\\hline
\end{tabular}\\[3pt]
\caption{The polynomials $P_m$}
\label{pm table}
\end{table}

\begin{prop}\label{rk2 inf FB g}
Suppose $a$ and $b$ are integers with $ab\le-4$ and write $B=\begin{bsmallmatrix}0&a\\b&0\end{bsmallmatrix}$.
Then the $\g$-vectors associated to $B^T$ are $\begin{bsmallmatrix}\pm1\\0\end{bsmallmatrix}$, $\begin{bsmallmatrix}0\\\pm1\end{bsmallmatrix}$, and 
\begin{eqnarray}
\label{rays 1 even}
\begin{bsmallmatrix}\sgn(a)P_m\\-aP_{m+1}\end{bsmallmatrix}&&\qquad\text{for $m$ even and $m\ge 0$,}\\
\label{rays 1 odd}
\begin{bsmallmatrix}-bP_m\\\sgn(b)P_{m+1}\end{bsmallmatrix}&&\qquad\text{for $m$ odd and $m\ge 1$,}\\
\label{rays 2 even}
\begin{bsmallmatrix}-bP_{m+1}\\\sgn(b)P_m\end{bsmallmatrix}&&\qquad\text{for $m$ even and $m\ge 0$, and}\\
\label{rays 2 odd}
\begin{bsmallmatrix}\sgn(a)P_{m+1}\\-aP_m\end{bsmallmatrix}&&\qquad\text{for $m$ odd and $m\ge 1$.}
\end{eqnarray}
The rays of $\F_B^\circ$ are spanned by the $\g$-vectors given above.
\end{prop}

To prove Proposition~\ref{rk2 inf FB g}, one matches \eqref{Pm def} with a well-known formula for $U_m(x)$, the Chebyshev polynomials of the second kind, to see that 
\begin{equation}\label{Pm cheby}
P_m=\begin{cases}
U_m(\frac{\sqrt{-ab}}{2})&\text{if $m$ is even,}\\
\frac{1}{\sqrt{-ab}}\,U_m(\frac{\sqrt{-ab}}{2})&\text{if $m$ is odd.}\\
\end{cases}
\end{equation}
Using known formulas for the roots of $U_m(x)$, one shows that $P_m$ is positive for all $m\ge 0$.
One then computes $\g$-vectors using Conjecture~\ref{gvec mu}, which holds in this case by Theorem~\ref{NZ 4.1(v)}.
The second assertion follows by Theorem~\ref{g subfan}.
We omit further details.
(The calculation of $\g$-vectors draws on notes shared with the author by Speyer~\cite{Speyer personal} in connection with joint work on~\cite{framework}.)

The rays described in \eqref{rays 1 even} and \eqref{rays 1 odd} interlace and approach a limit.
Using \eqref{Pm cheby} and a well-known formula for $U_m(x)$, we see that the limiting ray is spanned by 
\begin{equation}\label{vinf eq}
\v_\infty(a,b)=\begin{bsmallmatrix}2\sgn(a)\sqrt{-ab}\\-a(\sqrt{-ab}+\sqrt{-ab-4})\end{bsmallmatrix}.
\end{equation}
Similarly, the rays described in \eqref{rays 2 even} and \eqref{rays 2 odd} interlace and approach the limit 
\begin{equation}
\label{v-inf eq}
\v_{-\infty}(a,b)=\begin{bsmallmatrix}-b(\sqrt{-ab}+\sqrt{-ab-4})\\2\sgn(b)\sqrt{-ab}\end{bsmallmatrix}.
\end{equation}
Let $C_\infty(a,b)$ be the closed cone whose extreme rays are $\v_\infty(a,b)$ and $\v_{-\infty}(a,b)$.
Then $C_\infty(a,b)\setminus\set{\mathbf{0}}$ is the set of all points in $\reals^2$ not contained in any $2$-dimensional cone transitively adjacent to $O$ in $\F_B$.
For $k\in\set{1,2}$, the mutation map $\eta^B_k$ takes $C_\infty(a,b)$ to $C_\infty(-a,-b)$.
Similarly, $\eta^{-B}$ takes $C_\infty(-a,-b)$ to $C_\infty(a,b)$.
(Recall that $\mu_k(B)=-B$ for $k=1$ or $k=2$.)
Both of these cones are contained in an open coordinate orthant, and we conclude that $C_\infty(a,b)\setminus\set{\mathbf{0}}$ is contained in a $B$-class.
That $B$-class cannot be any larger than $C_\infty(a,b)\setminus\set{\mathbf{0}}$ without overlapping a $B$-cone that is transitively adjacent to $O$.
Thus $C_\infty(a,b)\setminus\set{\mathbf{0}}$ is a $B$-class and $C_\infty(a,b)$ is a $B$-cone.
We have proven the following proposition.

\begin{prop}\label{rk2 inf FB}
Suppose $a$ and $b$ are integers with $ab\le-4$ and write $B=\begin{bsmallmatrix}0&a\\b&0\end{bsmallmatrix}$.
Then the mutation fan $\F_B$ consists of $\F_B^\circ$ and the cone $C_\infty(a,b)$.
\end{prop}

For each ray $\rho$ in $\R^\circ(B)$, choose $\v_\rho$ to be the corresponding vector in Proposition~\ref{rk2 inf FB g}.
Each other ray $\rho$ in $\R(B)$ is spanned by $\v_\infty(a,b)$ or $\v_{-\infty}(a,b)$ or both, and we thus choose $\v_\rho$ to be $\v_\infty(a,b)$ or $\v_{-\infty}(a,b)$.
Now, Proposition~\ref{rk2 gvec b-coherent} tells us that in any $B$-coherent linear relation supported on $\set{\v_\rho:\rho\in\R(B)}$, all of the vectors appearing in Proposition~\ref{rk2 inf FB g} appear with coefficient zero.
The one or two remaining vectors form a linearly independent set, so they also appear with coefficient zero.
We conclude that there exists no nontrivial $B$-coherent linear relation supported on a finite subset of $\set{\v_\rho:\rho\in\R(B)}$.
Theorem~\ref{NZ 4.1(iv)} implies that each pair of vectors $\v_\rho$ and $\v_{\rho'}$ spanning a $\g$-vector cone for $B^T$ are a $\integers$-basis for $\integers^2$.

When $\v_\infty(a,b)$ and $\v_{-\infty}(a,b)$ are related by a positive scaling, the cone  $C_\infty$ degenerates to a limiting ray spanned by $\v_\infty(a,b)=\begin{bsmallmatrix}2\sgn(a)\\-a\end{bsmallmatrix}$.
This happens if and only if $ab=-4$, so that $(a,b)$ is $(\pm 1,\mp 4)$, $(\pm 2,\mp 2)$, or $(\pm 4,\mp 1)$.
These are exactly the rank-$2$ cases where $B$ is of affine type, as we define later in Definition~\ref{def affine type}.
In these cases, the mutation fan $\F_B$ consists of the $\g$-vector cones for $B^T$, together with the limiting ray $\rho_\infty$.
We have proved the following.

\begin{prop}\label{rk2 affine basis}
Suppose $a$ and $b$ are integers with $ab=-4$ and write $B=\begin{bsmallmatrix}0&a\\b&0\end{bsmallmatrix}$.
Then for any $R$, the $\g$-vectors for $B^T$, together with the shortest integer vector that is a positive rational scaling of $\begin{bsmallmatrix}2\sgn(a)\\-a\end{bsmallmatrix}$, constitute a positive $R$-basis for $B$.
\end{prop}

Some of the rank-$2$ universal extended exchange matrices of affine type are given below in~\eqref{rk2 affine universal}.
\begin{equation}
\label{rk2 affine universal}
\begin{bsmallmatrix*}[r]
0&1\\
-4&0\\[1pt]\hline\\
-1&0\\
1&-1\\
3&-2\\
\cdot&\cdot\\[-3pt]
\cdot&\cdot\\[-3pt]
\cdot&\cdot\\[1pt]\hline\\
0&-1\\
4&-3\\
8&-5\\
\cdot&\cdot\\[-3pt]
\cdot&\cdot\\[-3pt]
\cdot&\cdot\\[1pt]\hline\\
0&1\\
4&-1\\
8&-3\\
\cdot&\cdot\\[-3pt]
\cdot&\cdot\\[-3pt]
\cdot&\cdot\\[1pt]\hline\\
1&0\\
3&-1\\
5&-2\\
\cdot&\cdot\\[-3pt]
\cdot&\cdot\\[-3pt]
\cdot&\cdot\\[1pt]\hline\\
2&-1\\\end{bsmallmatrix*}
\qquad\qquad
\begin{bsmallmatrix*}[r]
0&2\\
-2&0\\[1pt]\hline\\
-1&0\\
0&-1\\
1&-2\\
2&-3\\
3&-4\\
4&-5\\
\cdot&\cdot\\[-3pt]
\cdot&\cdot\\[-3pt]
\cdot&\cdot\\[1pt]\hline\\
0&1\\
1&0\\
2&-1\\
3&-2\\
4&-3\\
5&-4\\
\cdot&\cdot\\[-3pt]
\cdot&\cdot\\[-3pt]
\cdot&\cdot\\[1pt]\hline\\
1&-1
\end{bsmallmatrix*}
\end{equation}
Each matrix contains several infinite sequences of coefficient rows, separated by horizontal lines for the sake of clarity.
The calculations are easy when $-ab=4$, because it is known that $U_m(1)=m+1$, so that $P_m=m+1$ for $m$ even and $P_m=\frac{1}{2}(m+1)$ for $m$ odd.
As before, Propositions~\ref{antipodal FB} and~\ref{pi FB} ensure that the remaining cases are obtained from the cases shown by negating all entries and/or swapping the columns and then swapping the first two rows.
The mutation fans for the two cases in~\eqref{rk2 affine universal} are shown in Figure~\ref{affine mutation fans}.
\begin{figure}[ht]
\begin{tabular}{ccc}
\includegraphics{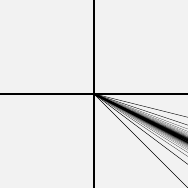}&&
\includegraphics{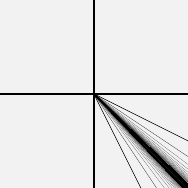}
\\[2 pt]
$B=\begin{bsmallmatrix*}[r]0&1\\-4&0\end{bsmallmatrix*}$&&
$B=\begin{bsmallmatrix*}[r]0&2\\-2&0\end{bsmallmatrix*}$
\end{tabular}
\caption{Mutation fans $\F_B$ for the affine examples in~\eqref{rk2 affine universal}}
\label{affine mutation fans}
\end{figure}

When $ab\le-5$, the set $C_\infty(a,b)$ is a $2$-dimensional cone whose extreme rays are algebraic, but not rational.
(Consider~\eqref{vinf eq} and~\eqref{v-inf eq} and note that the only pair of perfect squares differing by $4$ is $\set{0,4}$.)
Two of these cases are shown in Figure~\ref{wild mutation fans}, with $C_\infty(a,b)$ shaded gray.
The additional points drawn in the figure are explained in the proof of Proposition~\ref{rk2 wild basis}.
\begin{figure}[ht]
\begin{tabular}{ccc}
\includegraphics{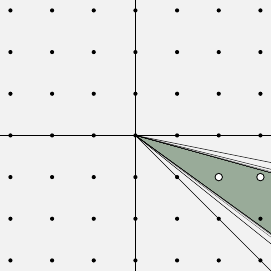}&&
\includegraphics{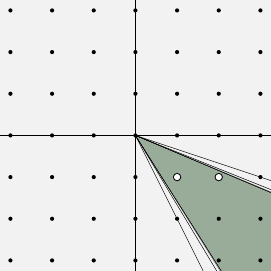}
\\[2 pt]
$B=\begin{bsmallmatrix*}[r]0&1\\-5&0\end{bsmallmatrix*}$&&
$B=\begin{bsmallmatrix*}[r]0&2\\-3&0\end{bsmallmatrix*}$
\end{tabular}
\caption{Mutation fans $\F_B$ in two infinite non-affine rank-$2$ cases}
\label{wild mutation fans}
\end{figure}

\begin{prop}\label{rk2 wild basis}
Suppose $a$ and $b$ are integers with $ab\le-5$ and write $B=\begin{bsmallmatrix}0&a\\b&0\end{bsmallmatrix}$.
\begin{enumerate}
\item\label{wild pos basis}
 If $R$ is a field containing $\rationals[\sqrt{-ab},\sqrt{-ab-4}]$, then the $\g$-vectors for $B^T$, together with $\v_\infty(a,b)$ and $\v_{-\infty}(a,b)$, constitute a positive $R$-basis for~$B$.
\item\label{wild no pos basis}
No positive $\integers$-basis or $\rationals$-basis exists for $B$.
\item\label{wild cone basis}
If $R$ is a field, then the $\g$-vectors for $B^T$, together with any two linearly independent vectors in $C_\infty(a,b)\cap R^2$, constitute a cone $R$-basis for $B$.
\item\label{wild int basis}
There exist integer vectors in $C_\infty(a,b)$ that, together with the $\g$-vectors for $B^T$, constitute a cone $\integers$-basis for $B$.
\end{enumerate}
\end{prop}
\begin{proof}
Assertion~\eqref{wild pos basis} is immediate from the discussion above.
Assertion \eqref{wild no pos basis} also follows in light of Corollary~\ref{pos basis g cor}.
To establish Assertions~\eqref{wild cone basis} and~\eqref{wild int basis}, we point out that the extreme rays of $C_\infty(a,b)$ are not themselves $B$-cones.
(If an extreme ray of $C_\infty(a,b)$ were a $B$-cone, then the nonzero points on that ray would be a $B$-class, by Proposition~\ref{relint class easy}.
But $C_\infty(a,b)\setminus\set{\mathbf{0}}$ is already a $B$-class.)
To satisfy Definition~\ref{cone basis def}, we therefore don't need vectors in the extreme rays of $C_\infty(a,b)$.  
We only need two vectors in $C_\infty(a,b)$ that span $R^2$.
If $R$ is a field, these can be any linearly independent vectors in $C_\infty(a,b)$.

If $R=\integers$, then we need to find an $\integers$-basis for $\integers^2$ in $C_\infty(a,b)$.
If $a'\ge a\ge0$ and $0\ge b\ge b'$, then $C_\infty(a,b)\subseteq C_\infty(a',b')$.
Thus, by symmetry, it is enough to check the minimal cases $(a,b)=(1,-5)$ and $(a,b)=(2,-3)$.
In both cases, $C_\infty(a,b)$ contains a $\integers$-basis for $\integers^n$, shown by the white dots in Figure~\ref{wild mutation fans}.
(The black and white dots in the pictures are the points $\integers^2$.)
\end{proof}

\section{Rays in the Tits cone}\label{Tits sec}
In this section, we define a subset $\R^{\pm\Tits}(B)$ of the rays $\R(B)$ of $\F_B$.
These are the rays spanned by $\g$-vectors for $B^T$ that are contained in the Tits cone or in its antipodal cone, in a sense that we make precise below.
We prove the following proposition.

\begin{prop}\label{Tits b-coherent}
Suppose $B$ is acyclic and satisfies the Standard Hypotheses.
Suppose also that every submatrix $B_\br{j}$ of $B$ is of finite Cartan type.
Choose a nonzero vector $\v_\rho$ in each ray $\rho\in\R(B)$.
Then any $B$-coherent linear relation supported on $\set{\v_\rho:\rho\in\R(B)}$ is in fact supported on $\set{\v_\rho:\rho\in\R(B)\setminus\R^{\pm\Tits}(B)}$.
\end{prop}

Proposition~\ref{Tits b-coherent} will allow us to construct positive bases for exchange matrices of finite type, generalizing Proposition~\ref{rk2 finite basis}.
To illustrate the usefulness of the proposition beyond finite type, we also construct positive bases for a rank-$3$ exchange matrix of affine type.
A similar construction can be carried out for any rank-$3$ exchange matrix of affine type.
Based on these constructions and on insight from \cite{afframe}, we make a general conjecture about bases for exchange matrices of affine type.

We now proceed to fully explain and then prove Proposition~\ref{Tits b-coherent}.
For Standard Hypotheses, see Definition~\ref{standard hyp}.
An exchange matrix $B$ is \newword{acyclic} if, possibly after reindexing by a permutation of $[n]$, it has the following property:
If $b_{ij}>0$ then $i<j$.
Define $B_\br{j}$ to be the matrix obtained from $B$ by deleting row $j$ and column~$j$.

\begin{definition}[\emph{Cartan companion of $B$}]\label{Cartan comp def}
Recall from Definition~\ref{seed} that $B$ is an $n\times n$ skew-symmetrizable integer matrix, and specifically that $\delta(i) b_{ij}=-\delta(j) b_{ji}$ for all $i,j\in[n]$.
The \newword{Cartan companion} of $B$ is the $n\times n$ matrix $A$ with diagonal entries $2$ and off-diagonal entries $a_{ij}=-|b_{ij}|$.
Then $\delta(i) a_{ij}=\delta(j) a_{ji}$ for all $i,j\in[n]$, and accordingly $A$ is \newword{symmetrizable}.
In fact, $A$ is a \newword{Cartan matrix} in the usual sense.
(See \cite{Kac}, or for a treatment specific to the present purposes, see \cite[Section~2.2]{typefree}.)
\end{definition}

\begin{definition}[\emph{Tits cone and $\R^{\pm\Tits}(B)$}]\label{Tits def}
Let $V$ be a real vector space of dimension $n$ with a basis $\Pi=\set{\alpha_i:i\in[n]}$ and write $V^*$ for its dual vector space.
The basis vectors $\alpha_i$ are called the \newword{simple roots}.
The \newword{simple co-roots} associated to $A$ are $\alpha_i\ck= \delta(i)^{-1} \alpha_i$.
These are the simple roots associated to the transpose $A^T$.
Let $\br{\cdot,\cdot}:V^*\times V\to\reals$ denote the canonical pairing.
We identify $V^*$ with $\reals^n$ by identifying dual basis to $\Pi$ (called the \newword{fundamental co-weights}) with the standard unit basis.
Specifically, we identify the dual basis vector dual to $\alpha_i$ with~$\e_i$.

Associated to the Cartan matrix $A$ is a Coxeter group $W$, defined as a group generated by reflections on $V$.
For each $i\in[n]$, we define a reflection $s_i$ by specifying its action on the simple roots: $s_i(\alpha_j)=\alpha_j-a_{ij} \alpha_i$.
The group $W$ is the group generated by all of these reflections, and it is a Coxeter group.
We define a symmetric bilinear form $K$ on $V$ by setting $K(\alpha\ck_i, \alpha_j)=a_{ij}$.
The element $s_i$ is a reflection with respect to the form $K$, and thus $W$ acts by isometries on $V$ with respect to the form $K$.
The action of $W$ on $V$ induces a dual action on $V^*$ in the usual way.  
The action of $s_i$ on $V^*$ is a reflection fixing the hyperplane $\alpha_i^\perp$.
More specifically, the action of $s_i$ fixes $\e_j$ for $j\neq i$, and sends $\e_i$ to $-\e_i+\sum_{j\neq i}a_{ij}\e_j$,

Define $D = \bigcap_{i\in[n]} \set{x\in V^*: \br{x,\alpha_i}\ge 0}\subset V^*$.
Under the identification of $V^*$ with $\reals^n$, the cone $D$ is the nonnegative orthant $O$.
The cones $wD$ are distinct, for distinct $w\in W$.
The union of these cones is called the \newword{Tits cone} $\Tits(A)$.
As the name suggests, the Tits cone is a cone.
By definition, it is preserved under the action of $W$ on $V^*$.
We write $-\Tits(A)$ for the image of $\Tits(A)$ under the antipodal map.

Continuing to identify $V^*$ with $\reals^n$ as above, we define $\R^{\pm\Tits}(B)$ to be the set of rays in $\R(B)$ that are contained in $\Tits(A)\cup(-\Tits(A))$.
(Since $\Tits(A)$ is a cone and contains $\mathbf{0}$, any ray is either completely contained in $\Tits(A)\cup(-\Tits(A))$ or intersects $\Tits(A)\cup(-\Tits(A))$ only at $\mathbf{0}$.)
\end{definition}

\begin{definition}[\emph{Cartan type of $B$}]\label{def Cartan type}
A Cartan matrix $A$ is of \newword{finite type} if the associated Coxeter group $W$ is finite.
When $A$ is of finite type, then $\Tits(A)$ is all of $V^*$.
Cartan matrices of finite type are classified (e.g.\ as ``type $D_n$,'' etc.).
We define Cartan matrices of \newword{affine type} in Definition~\ref{def affine type}.
If $A$ is the Cartan companion of $B$, then we use the phrase \newword{Cartan type of $B$} to refer to the type of~$A$.
\end{definition}

The following result \cite[Theorem~1.4]{ca2} links Definition~\ref{def Cartan type} to Definition~\ref{fin def}.
\begin{theorem}\label{finite and Cartan}
An exchange matrix $B$ is of finite type if and only if it is mutation equivalent to an exchange matrix of finite Cartan type.
\end{theorem}

Having explained Proposition~\ref{Tits b-coherent}, we now prepare to prove it.
We begin with a known result on the $\g$-vector cones in finite type.
\begin{theorem}\label{g complete}
If $B$ is of finite type, then the $\g$-vector cones associated to $B$ are the maximal cones of a complete simplicial fan.
\end{theorem}

One way to obtain Theorem~\ref{g complete} from the literature is the following:
If $B$ is of finite Cartan type, then the $\g$-vector cones coincide with the maximal cones of the Cambrian fan of \cite{camb_fan}, which is complete by definition and simplicial by \cite[Theorem~1.1]{camb_fan}.
This was conjectured (and proved in a special case) in \cite[Section~10]{camb_fan} and proved (for all $B$ of finite Cartan type) in~\cite{YZ}.
To obtain the theorem for $B$ of finite type but not of finite Cartan type, one applies Theorems~\ref{NZ 4.1(v)} and~\ref{finite and Cartan} and appeals to the case of finite Cartan type.

Next, we need a known result on $\g$-vectors.
\begin{prop}\label{g restrict}
Suppose the entries in column $n$ of $B$ are nonpositive and suppose $t$ is a vertex of $\T_n$ connected to $t_0$ by a path with no edges labeled $n$.
Then the $\g$-vector cone $\Cone_t^{B;t_0}$ is the nonnegative linear span of $\e_n$ and $\Cone_t^{B_\br{n};t_0}$.
\end{prop}
Here, we realize $\T_{n-1}$ by deleting all edges labeled $n$ from $\T_n$ and taking the connected component of $t_0$.
The cone $\Cone_t^{B_\br{n};t_0}\subset\reals^{n-1}$ is embedded into $\reals^n$ by appending an $n\th$ entry $0$ to all vectors in the cone.

The fact that $\Cone_t^{B;t_0}$ has $\e_n$ as an extreme ray is obvious from Definition~\ref{gvec def}.
The rest of Proposition~\ref{g restrict} is not obvious from Definition~\ref{gvec def}.
We sketch how the rest of the proposition can be obtained from \cite{framework}.

By \cite[Theorem~3.27]{framework} we associate a \newword{framework} to any exchange matrix $B$; this is an assignment of $n$ vectors to each vertex of $\T_n$ satisfying certain conditions.
By \cite[Theorem~3.24(3)]{framework}, the vectors assigned to $t\in\T_n$ are inward-facing normals to the facets of the cone $\Cone_t^{B;t_0}$.
A framework in particular satisfies the Transition condition and the Sign condition.
The Sign condition says that each inward-facing normal $\beta$ has a sign $\sgn(\beta)\in\set{\pm1}$ such that every entry of $\beta$ weakly agrees in sign with $\sgn(\beta)$.
If $t$ and $t'$ are adjacent vertices of $\T_n$ then the cones $\Cone_t^{B;t_0}$ and $\Cone_{t'}^{B;t_0}$ share a facet.
Say $\beta$ is normal to that facet, facing inward with respect to $\Cone_t^{B;t_0}$, and suppose $\gamma$ is some other inward-facing normal to $\Cone_t^{B;t_0}$.
The Transition condition is assertion that $\gamma+[\sgn(\beta)\omega(\beta\ck,\gamma)]_+\,\beta$ is an inward facing normal to $\Cone_{t'}^{B;t_0}$.
Here $\beta\ck$ is a certain nonzero scaling of $\beta$ and $\omega$ is bilinear form whose matrix is essentially $B$.

Now $\Cone_{t_0}^{B;t_0}$ has $\e_n$ as an inward-facing normal.
Taking $\beta$ to be some other inward-facing normal of $\Cone_{t_0}^{B;t_0}$, the assumption that the entries in column $n$ of $B$ are nonpositive translates to the assertion that $\sgn(\beta)\omega(\beta\ck,\e_n)$ is nonpositive, so the Transition condition implies that the $\g$-vector cones adjacent to $\Cone_{t_0}^{B;t_0}$ also have $\e_n$ as an inward-facing normal.
Repeating the argument, we see that $\e_n$ is an inward-facing normal to any cone $\Cone_t^{B;t_0}$ with $t$ connected to $t_0$ by a path with no labels $n$.
Restricting the Transition condition for cones $\Cone_t^{B;t_0}$ to the facets defined by $\e_n$, we obtain exactly the Transition condition for cones $\Cone_t^{B_\br{n};t_0}$, and we conclude by induction that the facet of $\Cone_t^{B;t_0}$ defined by $\e_n$ is exactly $\Cone_t^{B_\br{n};t_0}$.  

Theorem~\ref{g complete} and Proposition~\ref{g restrict} allow us to prove the following Lemma which is the key step in the proof of Proposition~\ref{Tits b-coherent}.

\begin{lemma}\label{unique ray below with hypotheses}
Let $B$ be an exchange matrix satisfying the Standard Hypotheses.
Suppose the entries in column $n$ of $B$ are nonnegative and suppose $B_\br{n}$ is of finite type.
Then the unique ray of $\F_B$ intersecting $\set{\x\in\reals^n:x_n>0}$ is $\reals_{\ge0}\e_n$.
\end{lemma}
\begin{proof}
Theorem~\ref{g complete} (applied to $B^T$) says that $\reals^{n-1}$ is covered by cones $\Cone_t^{B^T_\br{n};t_0}$ for $t\in\T_{n-1}$.
Thus Proposition~\ref{g restrict} implies that the closed halfspace of $\reals^n$ consisting of vectors with nonnegative $n\th$ entry is covered by cones $\Cone_t^{B^T;t_0}$ with $t$ connected to $t_0$ by paths with no labels $n$.
By Theorem~\ref{g subfan}, each of these cones is in $\F_B$.
Each of these cones contains the ray $\reals_{\ge0}\e_n$, which is therefore the unique ray of $\F_B$ intersecting $\set{\x\in\reals^n:x_n>0}$.
\end{proof}

\begin{definition}[\emph{Length of a ray in $\Tits(A)$}]\label{length of ray}
The \newword{length} of an element $w\in W$ is the number of letters in a shortest expression for $w$ as a product of generators $s_i$.
We write $\ell(w)$ for this length.
Given an element $w$ and a generator $s_i$, it is well-known that $\ell(s_iw)<\ell(w)$ if and only if $wD$ is contained in the halfspace $\set{x\in V^*: \br{x,\alpha_i}\le 0}$.
This is identified with the set of points in $\reals^n$ with non-positive $i\th$ coordinate.
We define the \newword{length} $\ell(\rho)$ of a ray $\rho$ contained in $\Tits(A)$ to be the minimum length of $w$ such that $\rho\subset wD$.
If $s_i\rho$ is the image of $\rho$ under the reflection $s_i$, then we have $\ell(s_i\rho)<\ell(\rho)$ if and only if the $i\th$ coordinate of nonzero vectors in $\rho$ is negative.
\end{definition}

We now prove Proposition~\ref{Tits b-coherent}.

\begin{proof}[Proof of Proposition~\ref{Tits b-coherent}]
Since $B$ is acyclic, we may as well assume that $B$ is indexed so that if $b_{ij}>0$ then $i<j$.
The skew-symmetrizability of $B$ means that if $b_{ji}<0$ then $i<j$.
In particular, the entries in column $n$ of $B$ are all nonnegative.

Consider a $B$-coherent linear relation $\sum_{\rho\in S}c_\rho\v_\rho$ with $S$ a finite subset of $\R(B)$.
Let $\rho$ be a ray in $S$ contained in $\Tits(A)$.
(Rays in $-\Tits(A)$ are dealt with later.)
We argue by induction on $\ell(\rho)$, letting $B$ vary, that $c_\rho=0$.

First, suppose the $n\th$ coordinate of $\v_\rho$ is positive.
Then Lemma~\ref{unique ray below with hypotheses} implies that $\v_\rho$ is the unique vector in $\set{\v_{\rho'}:\rho'\in S}$ whose $n\th$ entry is positive.
Now Proposition~\ref{one positive or one negative} implies that $c_\rho=0$.

Next suppose the $n\th$ coordinate of $\v_\rho$ is negative.
Write $\v_\rho=\sum_{i\in[n]}v_i\e_i$.
The entries $b_{nj}$ are all nonpositive, so $\eta_n^B(\v_\rho)=\v_\rho-2v_n\e_n-\sum_{j\in[n-1]}v_nb_{nj}\e_j$.
Since the $b_{nj}$ are nonpositive, they equal the entries $a_{nj}$ of the Cartan companion of $B$, and we observe that $\eta_n^B(\v_\rho)=s_n(\v_\rho)$.
Proposition~\ref{eta FB} implies that the rays $\R(\mu_n(B))$ of $\F_{\mu_n(B)}$ are obtained by applying $\eta_n^B$ to each ray in $\R(B)$.
Thus $\sum_{\rho\in\R(B)}c_\rho\eta_n^B(\v_\rho)$ is a $\mu_n(B)$-coherent linear relation.
Since all entries in column $n$ are nonnegative, all entries in row $n$ are nonpositive.
It is thus apparent that the operation $\mu_n$ does nothing to $B$ other than reverse signs in row $n$ and column $n$.
In particular, $\mu_n(B)$ is acyclic and the Cartan type of submatrices of $\mu_n(B)$ is the same as the Cartan type of submatrices of $B$.
We see that $\mu_n(B)$ satisfies the hypotheses of the proposition.
Since the $i\th$ coordinate of $\v_\rho$ is negative, we have $\ell(\eta_n^B(\rho))=\ell(s_i\rho)<\ell(\rho)$, so by induction on $\ell(\rho)$, we conclude that $c_\rho=0$ as desired.

Finally, suppose the $n\th$ coordinate of $\v_\rho$ is zero.
As before, $\sum_{\rho\in\R(B)}c_\rho\eta_n^B(\v_\rho)$ is a $\mu_n(B)$-coherent linear relation, and as before, $\mu_n(B)$ is obtained from $B$ by reversing signs in row $n$ and column $n$.
We observe that $\eta^B_n(\rho)=\rho$.
The matrix $\mu_n(B)$ has nonnegative entries in column $n-1$.
Now we consider the sign of the $(n-1)\st$ coordinate of $\v_\rho$.
If it is positive or negative, then we conclude as above that $c_\rho=0$.
If it is zero, the we replace $\mu_n(B)$ with $\mu_{(n-1)n}(B)$ and consider the sign of the $(n-2)\nd$ coordinate of $\v_\rho$.
Continuing in this manner, we eventually find a positive or negative coefficient, since $\v_\rho$ is nonzero, and when we do, we conclude that $c_\rho=0$.

The base case $\ell(\rho)=0$ is handled as part of the above argument:
If $\ell(\rho)=0$ then all coordinates of $\v_\rho$ are nonnegative and we eventually complete the argument without induction.

This completes the proof that $c_\rho=0$ for a ray $\rho$ in $\R(B)$ contained in $\Tits(A)$.
If $\rho$ is in $-\Tits(A)$, then \eqref{eta antipodal} implies that $\sum_{\rho\in\R(B)}c_\rho(-\v_\rho)$ is a $(-B)$-coherent linear relation.
Proposition~\ref{antipodal FB} implies that the vectors $-\v_\rho$ span the rays of $\F_{-B}$.
The argument above shows that $c_\rho=0$.
\end{proof}

\begin{remark}\label{cambrian approach}
Proposition~\ref{Tits b-coherent} can also be proved using the Cambrian frameworks of \cite[Section~5]{cambrian}.
Readers familiar with Cambrian lattices and sortable elements will recognize that the proof given here is patterned after the usual induction on length and rank common to proofs involving sortable elements.
\end{remark}

\begin{remark}\label{hypotheses}
It is natural to ask whether Proposition~\ref{Tits b-coherent} can be proved with weaker hypotheses.
Indeed, it was stated without assuming the Standard Hypotheses and without the hypotheses on submatrices in an earlier version of this paper, but an error was later found in the proof.
\end{remark}

We now use Proposition~\ref{Tits b-coherent} to prove the following theorem.

\begin{theorem}\label{finite g univ}  
Let $B$ be a skew-symmetrizable exchange matrix of finite type satisfying the Standard Hypotheses and let $R$ be any underlying ring.
Then the $\g$-vectors associated to $B^T$ constitute a positive $R$-basis for $B$.
\end{theorem}

We expect that every skew-symmetrizable exchange matrix of finite type satisfies the Standard Hypotheses.
(Indeed, as indicated in Remark~\ref{when sign-coherence known}, since every such matrix is mutation equivalent to an acyclic matrix, the result may soon appear in print.)

\begin{proof}
Theorem~\ref{g subfan} says that the $\g$-vector cones associated to $B^T$ are the maximal cones in a subfan of $\F_B$.
Theorem~\ref{g complete} says that this subfan is complete, and therefore it must coincide with the entire fan $\F_B$.
When $B$ is of finite Cartan type with Cartan companion $A$, the Tits cone $\Tits(A)$ is all of $\reals^n$.
Thus Proposition~\ref{Tits b-coherent} says that there is no non-trivial $B$-coherent linear relation supported on the $\g$-vectors.
By Theorem~\ref{NZ 4.1(iv)} and Proposition~\ref{positive cone basis}, we conclude that the $\g$-vectors are a positive $R$-basis for $B$, for any $R$.
The case where $B$ is of finite type but not of finite Cartan type now follows by Theorems~\ref{NZ 4.1(v)} and~\ref{finite and Cartan} and by the observation that a mutation map $\eta^B_\k$ takes an $R$-basis for $B$ to an $R$ basis for $\mu_\k(B)$.
\end{proof}

The $\g$-vectors for $B$ of finite Cartan type can be found explicitly in various ways, including using sortable elements and Cambrian lattices as described in~\cite{framework}, or by the methods of~\cite{YZ}.

\begin{remark}\label{me and ca4}
Theorems~\ref{basis univ} and~\ref{finite g univ} say that the $\g$-vectors for $B^T$ are the coefficient rows of a positive universal extended exchange matrix.
Another construction of universal coefficients, not conditioned on any conjectures, was already given in \cite[Theorem~12.4]{ca4}.
Furthermore, the construction from \cite{ca4} yields cluster algebras with completely universal coefficients, rather than only universal geometric coefficients.
That is, an arbitrary cluster algebra (not necessarily of geometric type) with initial exchange matrix $B$ admits a unique coefficient specialization from 
the universal cluster algebra.
(The relevant definition of coefficient specialization is \cite[Definition~12.1]{ca4}.)
We now explain the connection between Theorem~\ref{finite g univ} and \cite[Theorem~12.4]{ca4}.
This connection provided the original motivation for this research.

A more detailed description of \cite[Theorem~12.4]{ca4}, in the language of this paper, is the following:
Let $B$ be a bipartite exchange matrix of finite Cartan type with Cartan companion $A$.
(An exchange matrix $B$ is \newword{bipartite} if there is a function $\ep:[n]\to\set{\pm 1}$ such that $b_{ij}>0$ implies that $\ep(i)=1$ and that $\ep(j)=-1$.)
The \newword{co-roots} associated to $A$ are the vectors in the $W$-orbits of the simple co-roots.
Since $A$ is of finite type, there are finitely many co-roots.
A co-root is \newword{positive} if it is in the nonnegative linear span of the simple co-roots, and \newword{almost positive} if it is positive or if it is the negative of a simple co-root.
We construct an integer extended exchange matrix extending $B$ whose coefficient rows are indexed by almost positive co-roots.
The coefficient row indexed by a co-root $\beta\ck$ has entries $\ep(i)[\beta\ck:\alpha\ck_i]$ for $i=1,\ldots,n$, where $[\beta\ck:\alpha\ck_i]$ stands for the coefficient of $\alpha_i\ck$ in the expansion of $\beta\ck$ in the basis of simple co-roots.
The assertion of \cite[Theorem~12.4]{ca4} is that this extended exchange matrix defines a universal cluster algebra for $B$.

Let $L$ be the linear map taking a positive root $\alpha_i$ to $-\ep(i)\e_i$.
(Recall that we have identified $\e_i$ with an element in the basis for $V^*$ dual to the basis of simple roots in $V$.)
One can define almost positive roots by analogy to the definitions in the previous paragraph, so that the almost positive co-roots for $A$ are the almost positive roots for $A^T$.
As conjectured in \cite[Conjecture~1.4]{cambrian} and proved in \cite[Theorem~9.1]{camb_fan}, the map $L$ takes almost positive roots into rays in the Cambrian fan for $B$ and induces a bijection between almost positive roots and rays of the Cambrian fan.
(The difference in signs between \cite[Conjecture~1.4]{cambrian} and \cite[Theorem~9.1]{camb_fan} is the result of a difference in sign conventions.  See the end of \cite[Section~1]{cyclic}.)

It now becomes possible to relate \cite[Theorem~12.4]{ca4} to Theorem~\ref{finite g univ}.
To apply \cite[Theorem~12.4]{ca4} to $B^T$, we pass from co-roots to roots (thus exchanging $A$ with $A^T$) and we introduce a global sign change to the function $\ep$.
Thus we write a universal extended exchange matrix for $B^T$ by taking, for each almost positive root of $A$, a coefficient row with entries $-\ep(i)[\beta:\alpha_i]$ for $i=1,\ldots,n$, where $[\beta:\alpha_i]$ stands for the coefficient of $\alpha_i$ in the expansion of $\beta$ in the basis of simple roots.
In other words, the coefficient row indexed by an almost positive root $\beta$ is $L(\beta)$.
This means we have chosen a nonnegative vector in each ray of the Cambrian fan for $B^T$.
But as mentioned above, the Cambrian fan for $B^T$ is the fan whose maximal cones are the $\g$-vector cones for $B^T$, and thus we have essentially recovered Theorem~\ref{finite g univ} for the case of bipartite $B$.
\end{remark}

We conclude by sketching the construction of an $R$-basis for the rank-$3$ exchange matrix $B=\begin{bsmallmatrix*}[r]0&2&0\\-1&0&1\\0&-2&\hspace{6pt}0\\\end{bsmallmatrix*}$.
This exchange matrix is of affine Cartan type in the sense of the following definition.

\begin{definition}[\emph{Affine type}]\label{def affine type}
A Cartan matrix is of affine type if the associated symmetric bilinear form is positive semidefinite and every proper principal submatrix is of finite type.
For our purposes, the key property of a Cartan matrix $A$ of affine type is that the closure of $\Tits(A)$ is a half-space.
More specifically, $A$ has a $0$-eigenvector $\t=(t_1,\ldots,t_n)$ with nonnegative entries.
We think of $\t$ as the simple roots coordinates of a vector $\beta=\sum_{i\in[n]}t_i\alpha_i$ in $V$.
The closure of $\Tits(A)$ is then $\set{\x\in V^*:\br{\x,\beta}\ge0}$, where, under our identification of $\reals^n$ with $V^*$, we interpret $\br{\x,\beta}$ as the usual pairing of $\x$ with $\t$.
The Tits cone is the union of the open halfspace $\set{\x\in V^*:\br{\x,\beta}>0}$ with the singleton $\set{\mathbf{0}}$.
The Cartan matrices of affine type are classified, for example, in~\cite{Macdonald}.

As in Definition~\ref{def Cartan type}, an exchange matrix $B$ is of affine Cartan type if its Cartan companion is of affine type. 
By analogy with Theorem~\ref{finite and Cartan}, we say that an exchange matrix $B$ is of \newword{affine type} if it is mutation equivalent to an acyclic exchange matrix of affine Cartan type.
(The classification of Cartan matrices of finite type implies that an exchange matrix of finite Cartan type is acyclic.
A non-acyclic exchange matrix of affine Cartan type is of finite type.)
\end{definition}

Examples including Proposition~\ref{rk2 affine basis} and the example worked out below, together with insights from~\cite{afframe}, suggest the following conjecture.
More partial results towards the conjecture are described in the introduction. 

\begin{conj}\label{affine univ conj} 
Suppose $B$ is an exchange matrix of affine type.
There is a unique integer vector $\v_\infty$ such that the $\g$-vectors associated to $B^T$, together with $\v_\infty$, constitute a positive $R$-basis for $B$ for every $R$.
\end{conj}

For the rest of the section, we take $B$ to be the exchange matrix $\begin{bsmallmatrix*}[r]0&2&0\\-1&0&1\\0&-2&\hspace{6pt}0\\\end{bsmallmatrix*}$.
This is of affine Cartan type with Cartan companion $A=\begin{bsmallmatrix*}[r]2&-2&0\\-1&2&-1\\0&-2&2\\\end{bsmallmatrix*}$.
The matrix $B$ has a special property not shared by all exchange matrices of affine Cartan type.
It is of the form $DS$, where $D$ is an integer diagonal matrix and $S$ is an integer skew-symmetric matrix.  
Specifically, $B=\begin{bsmallmatrix*}[r]2&0&0\\0&1&0\\0&0&2\\\end{bsmallmatrix*}\begin{bsmallmatrix*}[r]0&1&0\\-1&0&1\\0&-1&\hspace{6pt}0\\\end{bsmallmatrix*}$.
The mutation class of $B$ is 
\begin{equation}\label{mut class particular B}
\set{\pm B,\, \pm\begin{bsmallmatrix*}[r]0&-2&0\\1&0&1\\0&-2&\hspace{6pt}0\\\end{bsmallmatrix*},\, \pm\begin{bsmallmatrix*}[r]0&-2&2\\1&0&-1\\-2&2&\hspace{6pt}0\\\end{bsmallmatrix*}},
\end{equation}
and each of these exchange matrices shares the same special property for the same~$D$.
By results of \cite{Demonet}, as explained in Remark~\ref{when sign-coherence known}, we conclude that the Standard Hypotheses hold for $B$.

The vector $\t=\begin{bsmallmatrix*}[r]1&1&1\end{bsmallmatrix*}$ is a $0$-eigenvector of $A$.
The closure of the Tits cone $\Tits(A)$ is the set of vectors in $\reals^n$ whose scalar product with $\t$ is nonnegative.
In light of Theorem~\ref{NZ 4.1(v)}, one can use Conjecture~\ref{gvec mu} and induction to calculate the $\g$-vectors for cluster variables associated to $B^T$.
One finds that there are exactly two $\g$-vectors 
\begin{equation}\label{v pm}
\v_+=\begin{bsmallmatrix*}[r]0&\hspace{6pt}1&-1\end{bsmallmatrix*}\quad\mbox{and}\quad\v_-=\begin{bsmallmatrix*}[r]1&-1&\hspace{6pt}0\end{bsmallmatrix*}
\end{equation}
in the boundary of $\Tits(A)$.
The remaining $\g$-vectors consist of six infinite families, with the $\g$-vector rays in each family limiting to the same ray in the boundary of $\Tits(A)$.
The vector 
\begin{equation}\label{v inf}
\v_\infty=\begin{bsmallmatrix*}[r]1&\hspace{6pt}0&-1\end{bsmallmatrix*}
\end{equation}
is the smallest nonzero integer vector in this limiting ray.
The infinite families of $\g$-vectors are $\set{\v_i+n\v_\infty:n\in\integers_{\ge0}}$ with $(\v_i:i=1\ldots6)$ being the vectors
\begin{equation}\label{the vi}
\begin{bsmallmatrix*}[r]
0&1&0
\end{bsmallmatrix*}
\quad
\begin{bsmallmatrix*}[r]
0&0&1
\end{bsmallmatrix*}
\quad
\begin{bsmallmatrix*}[r]
0&2&-1
\end{bsmallmatrix*}
\quad
\begin{bsmallmatrix*}[r]
-1&0&0
\end{bsmallmatrix*}
\quad
\begin{bsmallmatrix*}[r]
0&-1&0
\end{bsmallmatrix*}
\quad
\begin{bsmallmatrix*}[r]
1&-2&0
\end{bsmallmatrix*}
\end{equation}

The $\g$-vector cones for $B^T$ are the maximal cones of a simplicial fan occupying almost all of $\reals^3$.
This fan is depicted in Figure~\ref{FB B fig}.
\begin{figure}[ht]
\scalebox{.94}{\includegraphics{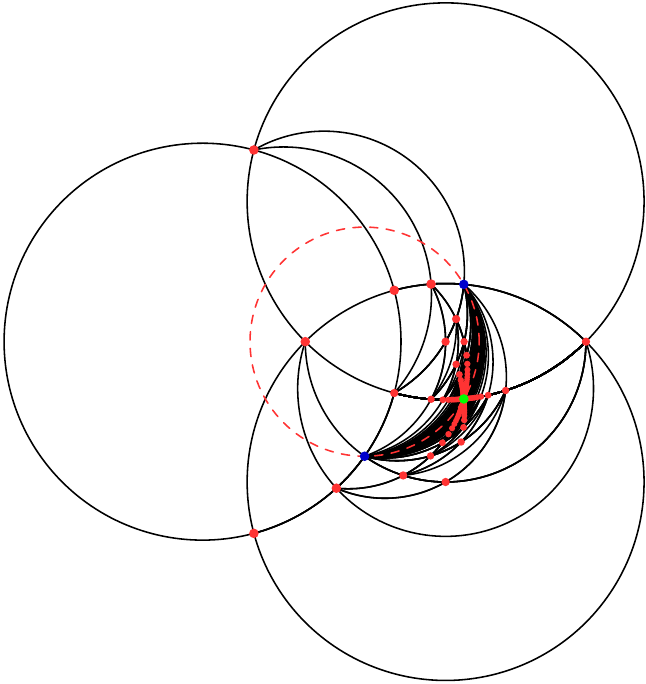}
\begin{picture}(0,0)(135,-164)
\put(-2.5,28.5){$\v_1$}
\put(-38,8){$\v_2$}
\put(16,32){$\v_3$}
\put(-69,98){$\v_4$}
\put(-67,-102){$\v_5$}
\put(-21,-80){$\v_6$}
\put(48,31){$\v_+$}
\put(-5.5,-64.5){$\v_-$}
\end{picture}}
\caption[The mutation fan $\F_B$]{
The mutation fan $\F_B$ for $B=\begin{bsmallmatrix*}[r]0&2&0\\-1&0&1\\0&-2&\hspace{6pt}0\\\end{bsmallmatrix*}$
}
\label{FB B fig}
\end{figure}
The picture is interpreted as follows:  
Intersecting each nonzero cone with a unit sphere about the origin, we obtain a collection of points, arcs and spherical triangles. 
These are depicted in the plane by stereographic projection, with the ray spanned by $\t$ projecting to the origin.
The rays spanned by $\v_+$ and $\v_-$ are indicated by blue (or dark gray) dots, the rays spanned by $\v_i+n\v_\infty$ are indicated by red (or medium gray) dots, and the limiting ray spanned by $\v_\infty$ is indicated by a green (or light gray) dot.
The dotted circle indicates the boundary of $\Tits(A)$.

Theorem~\ref{g subfan} says that the $\g$-vector cones determine a subfan of $\F_B$.
The points not contained in this subfan form an open cone consisting of positive linear combinations of $\v_+$ and $\v_-$.
This open cone is covered by the nonnegative span $C_+$ of $\v_+$ and $\v_\infty$ and the nonnegative span $C_-$ of $\v_-$ and $\v_\infty$.
Since $C_+$ is a limit of $B$-cones, it is contained in a $B$-cone by Proposition~\ref{lim B cone}, and similarly $C_-$ is contained in a $B$-cone.
These two cones cover the set of points not contained in $\g$-vector cones, so either $C_+$ and $C_-$ are each $B$-cones or $C_+\cup C_-$ is a single $B$-cone.
The latter possibility is ruled out by Proposition~\ref{contained Bcone}, taking $\k$ to be the empty sequence.
Thus the maximal cones of $\F_B$ are the $\g$-vector cones and $C_+$ and $C_-$.

Having determined the fan $\F_B$, we now prove Conjecture~\ref{affine univ conj} for this $B$.
The rays of $\F_B$ are spanned by the $\g$-vectors for $B^T$ and the vector $\v_\infty$.
We choose the vectors $(\v_\rho:\rho\in\R(B))$ to be these $\g$-vectors and $\v_\infty$.
Proposition~\ref{Tits b-coherent} says that any $B$-coherent linear relation supported on $\set{\v_\rho:\rho\in\R(B)}$ is in fact supported on $\set{\v_+,\v_\infty,\v_-}$.
Taking $\k$ to be the empty sequence and $j=2$ in Proposition~\ref{one positive or one negative} (twice), we see that the relation is supported on $\set{\v_\infty}$, and thus is trivial.
Appealing to Theorem~\ref{NZ 4.1(iv)} and Proposition~\ref{positive cone basis}, we have established the following proposition.
\begin{prop}\label{affine B3 R-basis}
For $B=\begin{bsmallmatrix*}[r]0&2&0\\-1&0&1\\0&-2&\hspace{6pt}0\\\end{bsmallmatrix*}$ and any underlying ring $R$, the set
\begin{equation}
\set{\v_+,\v_\infty,\v_-}\cup\bigcup_{i=1}^6\set{\v_i+n\v_\infty:n\in\integers_{n\ge0}}
\end{equation}
is a positive $R$-basis for $B$, where $\v_+$, $\v_\infty$, $\v_-$ and the $\v_i$ are as in \eqref{v pm}, \eqref{v inf}, and \eqref{the vi}.
\end{prop}

\addtocontents{toc}{\mbox{ }}

\section*{Acknowledgments}
Thanks to Ehud Hrushovski for enlightening the author on the subject of endomorphisms of the additive group $\reals$ (in connection with Remark~\ref{why linear}).
Thanks to David Speyer for pointing out the role of the polynomials $P_m$ in describing the $\g$-vectors associated to rank-$2$ exchange matrices of infinite type.
(See Section~\ref{rk2 sec}.)
Thanks to an anonymous referee of \cite{unisurface} for pointing out that Proposition~\ref{basis exists} needs the hypothesis that $R$ is a field.
Thanks to Kiyoshi Igusa and Dylan Rupel for helping to detect an error in an earlier version.
Thanks to an anonymous referee of this paper for many helpful suggestions which improved the exposition.

\end{document}